\documentclass{amsart}

\usepackage{mathptmx}
\usepackage{helvet}
\usepackage{courier}
\usepackage{graphicx}
\usepackage{multicol}
\usepackage{footmisc}
\usepackage[T1]{fontenc}
\usepackage{euscript}
\usepackage{eufrak}
\usepackage{epic}

\usepackage{amsfonts,amsthm,amsmath}
\usepackage{amssymb}
\usepackage{amscd}
\usepackage{enumerate}
\usepackage[usenames,dvipsnames]{color}
\usepackage{tikz}
\usepackage{soul}
\usepackage{array}
\usepackage{array,tabularx}

\numberwithin{equation}{section}
\numberwithin{equation}{subsection}

\theoremstyle{plain}

\newtheorem{theorem}[equation]{Theorem}
\newtheorem{lemma}[equation]{Lemma}
\newtheorem{proposition}[equation]{Proposition}

\newtheorem{corollary}[equation]{Corollary}

\newtheorem{question}[equation]{Question}
\newtheorem{conjecture}[equation]{Conjecture}

\theoremstyle{definition}\newtheorem{notation}[equation]{Notation}
\newtheorem{example}[equation]{Example}\newtheorem{examples}[equation]{Examples}
\newtheorem{remark}[equation]{Remark}

\newtheorem{definition}[equation]{Definition}
\newtheorem{problem}[equation]{Problem}

\newtheorem{bekezdes}[equation]{}

\newcommand{\bZ}{{\mathbb Z}}

\newcommand{\calC}{{\mathcal C}}
\newcommand{\cJ}{{\mathcal J}}

\newcommand{\cV}{{\mathcal V}}

\newcommand{\cS}{{\mathcal S}}\newcommand{\calS}{{\mathcal S}}
\newcommand{\cP}{{\mathcal P}}
\newcommand{\cI}{{\mathcal I}}

\newcommand{\cO}{{\mathcal O}}\newcommand{\calO}{{\mathcal O}}

\newcommand{\cL}{{\mathcal L}}

\newcommand{\calL}{{\mathcal L}}

\newcommand{\calQ}{{\mathcal Q}}

\DeclareMathOperator{\Hom}{Hom}

\newcommand{\tX}{\widetilde{X}}

\newcommand{\C}{{\calc}}

\newcommand{\rank}{{\rm rank}\, }

\newcommand{\ix}{\index}

\def\blfootnote{\xdef\@thefnmark{}\@footnotetext}

\newcommand{\frakv}{\mathfrak{v}}

\newcommand{\bt}{{\bf t}}

\newcommand{\bH}{{\mathbb H}}
\newcommand{\ocalj}{\overline{{\mathcal V}}}
\newcommand{\cali}{{\mathcal I}}
\newcommand{\calj}{{\mathcal V}}

\newcommand{\calF}{{\mathcal F}}

\newcommand{\setQ}{\mathbb{Q}}

\newcommand{\hh}{\mathfrak{h}}

\newcommand{\RR}{\mathfrak{R}}

\newcommand{\calv}{\mathcal{V}}

\newcommand{\sw}{{\mathfrak{sw}}}

\newcommand{\Z}{\mathbb{Z}}
\newcommand{\Q}{\mathbb{Q}}
\newcommand{\R}{\mathbb{R}}

\begingroup\makeatletter\ifx\SetFigFontNFSS\undefined%
\gdef\SetFigFontNFSS#1#2#3#4#5{%
  \reset@font\fontsize{#1}{#2pt}%
  \fontfamily{#3}\fontseries{#4}\fontshape{#5}%
  \selectfont}%
\fi\endgroup%

\newcommand{\chic}{\mathfrak{r}}

\newcommand{\bS}{{\mathbb S}}

\newcommand{\calt}{{\mathcal T}}

\def\C{\mathbb C}
\def\Q{\mathbb Q}
\def\R{\mathbb R}
\def\bH{\mathbb H}

\def\Z{\mathbb Z}
\def\N{\mathbb N}

\newcommand{\cale}{{\mathcal E}}

\author{Tam\'as \'Agoston}
\address{Alfr\'ed R\'enyi Institute of Mathematics,
Re\'altanoda utca 13-15, H-1053, Budapest, Hungary }
\email{agoston.tamas@renyi.hu}

\author{Andr\'as N\'emethi}
\address{Alfr\'ed R\'enyi Institute of Mathematics,
Re\'altanoda utca 13-15, H-1053, Budapest, Hungary \newline
 \hspace*{4mm} ELTE - University of Budapest, Dept. of Geometry, Budapest, Hungary \newline \hspace*{4mm}
BCAM - Basque Center for Applied Math.,
Mazarredo, 14 E48009 Bilbao, Basque Country – Spain}
\email{nemethi.andras@renyi.hu }

\title{Analytic lattice cohomology of surface singularities
}

\begin{document}

\thanks{The  authors are partially supported by NKFIH Grant ``\'Elvonal (Frontier)'' KKP 126683.}

\begin{abstract}
We construct the analytic lattice cohomology associated with the analytic type of any complex normal surface singularity.
 It is the categorification of the geometric genus of the  germ, whenever the link is a rational homology sphere.

It is the analytic analogue of the topological lattice cohomology, associated with the link of the germ whenever it is a rational homology sphere. This topological lattice cohomology  is the
 categorification of  the Seiberg--Witten invariant, and
conjecturally it is  isomorphic with the Heegaard Floer cohomology.

We compare the two lattice cohomologies: in some simple cases they coincide, but in general, the analytic cohomology is sensitive to the analytic structure.
We expect a deep connection with deformation theory.
We provide several basic  properties  and  key examples, and we formulate several conjectures
and problems.

This is the initial  article of a series, in which we develop the analytic lattice cohomology of singularities.
\end{abstract}

\maketitle


\date{}

\pagestyle{myheadings} \markboth{{\normalsize  T. \'Agoston , A. N\'emethi}} {{\normalsize Analytic lattice cohomology }}


\section{Introduction}\label{s:intr}

\bekezdes
It is an exciting project to compare the topological and  analytic invariants of a complex normal surface singularity
$(X,o)$.
The topological invariants are characterized by the link $M$
(an oriented closed 3--manifold), or by a dual resolution graph $\Gamma$
associated with a resolution $\tX$ of $(X,o)$. The analytic invariants usually are associated with the local ring,
e.g. by the structure of its ideals, or by a certain (multi-graded) filtration of it,
 or they are defined using
different coherent sheaves of
$(X,o)$ or of  the resolution $\tX$.  The topological invariants are discrete, and usually (in principle) can be computed combinatorially  from
$\Gamma$ (or, directly  from the fundamental group of the link, or by certain topological constructions of $M$).
However, the analytic invariants might reflect essentially the variation of the analytic structure, they might even have
moduli. Their computation and identification is usually hard, even if they are discrete.

Therefore, usually, as a first step, we try to find topological bounds  for them, or
redescribe  them
topologically
whenever the analytic structure is `nice', special,  (or, in the opposite case, if it is generic). In this procedure,
for several analytic invariants one tries to find some topological candidate,
which it  is equal to for special analytic structures. But it can also happen, conversely,  that for a known
topological object we wish to find its analytic counterpart or its analytic refinement.

In this way, we get `pairs' of invariants. The main  advantage of these pairs is not just that they coincide for
certain analytic structures, but  (if they are really `matching  pairs') they also behave rather similarly in their own categories
(e.g. they satisfy similar types of identities, they behave  similarly with respect to
 surgeries or  certain geometrical constructions, see a few examples below).
For such pairs and  their interactions the reader might consult \cite{Pairs}.

Let us mention some of them  (for more details  see also  \cite{Nfive,NJEMS,Adv,NGr,Pairs,Book}).
We fix a good resolution $\tX \to X$ (and in some of the cases below  we need to  assume that the link is a rational homology sphere).

\vspace{4mm}

\begin{tabular}{|l|l|} \hline
topological & analytic \\
\hline
Artin's fundamental  cycle $Z_{min}$ & maximal ideal cycle $Z_{max}$\\
integral Lipman's monoid (cone) $\calS_{top}$ & monoid $\calS_{an}$ of divisors of function \\
multivariable topological Poincar\'e series $Z(\bt)$ & multivariable  Poincar\'e series $P(\bt)$ of divisorial filtration\\
the canonical Seiberg-Witten invariant of $M$ & the geometric genus $p_g=h^1(\calO_{\tX})$\\
all the Seiberg--Witten invariants of $M$ & the equivariant geometric genera\\
topological lattice cohomology $\bH^*_{top}$& \hspace{2cm}????\\
topological graded root & \hspace{2cm}????\\
\hline\end{tabular}

\vspace{3mm}

Let us list also some matching pairs of statements regarding the above objects.

$\bullet$ $Z_{min}$ is the minimal nonzero element of $\calS_{top}$,  $Z_{max}$ is the minimal nonzero element of $\calS_{an}$.

$\bullet$ $Z(\bt)$ is supported on $\calS_{top}$, $P(\bt)$ is supported on $\calS_{an}$.

$\bullet$
The additivity formula of Okuma for $p_g$ \cite{Opg} is the analogue of the additivity formula
from \cite{BN} valid for the Seiberg--Witten invariant (in the first case the correction term is the periodic constant
of a reduced series associated with $Z(\bt)$, while in the analytic case this $Z(\bt)$ is replaced by $P(\bt)$).

$\bullet$ The multivariable periodic constant of $Z(\bt)$ is the Seiberg--Witten invariant of the link,
the multivariable periodic constant of $P(\bt)$ is the geometric genus of $(X,o)$.

\vspace{1mm}

Let us mention also that on the topological side the theory and list of properties of the lattice cohomology and graded roots
run interweaved: the graded root is an `improvement' of 
$\bH^0_{top}$. They were defined in \cite{NOSz,Nlattice,NGr} (from the topology of the link, whenever the link is a rational homology sphere, see below).

\bekezdes
The main motivation for the present article is the fact that in the literature there is no analytic analogue
of the topological lattice cohomology $\bH^*_{top}$, though $\bH^*_{top}$
 is deeply connected to  several  other topological invariants that  have well--defined matching pairs.
E.g., $\bH^*_{top}$ is defined as the cohomology of weighted cubes, and a particular sum of weights of cubes
provides $Z(\bt)$ (see \ref{rem:P0AN} here) --- \ and $Z(\bt)$ admits $P(\bt)$ as matching pair.
Or, the Euler characteristic of the cohomology theory
$\bH^*_{top}=\oplus_{q\geq 0}\bH^q_{top}$  is the Seiberg--Witten invariant (i.e. $\bH^*$ is a categorification of the Seiberg--Witten invariant, see Theorem \ref{th:ECharLC}) --- \  and the Seiberg--Witten invariant is matched by the (equivariant) geometric genus.

 Hence, it is rather natural to ask: {\it is there any natural cohomology theory
(in this case, a graded $\Z[U]$--module, just like  $\bH^*_{top}$ is), associated with the analytic structure $(X,o)$,
 which is the categorification of the geometric genus $p_g$ ? Is there an analogue to the topological graded root?
If the answer is yes, is it able to effectively reflect the change  of the analytic structures (on a fixed topological type)?}

The aim of this article is to give positive answers to these questions.
\bekezdes  Let us assume that the link is a rational homology sphere.
Then the topological lattice cohomology of the link $M$,  $\bH^*_{top}(M)$, is well--defined \cite{Nlattice}.
It  has a  rather different structure than any cohomology
theory associated with  analytic spaces by complex analytic or algebraic geometry (e.g. like  the various  sheaf--cohomologies).
It  has several  gradings: first of all, it has a direct sum decomposition according to the spin$^c$--structures $\sigma$
of $M$ (${\rm Spin}^c(M)$  is an $H_1(M,\Z)$ torsor).  Then each summand $\bH^*_{top}(M,\sigma)$ has a decomposition
$\oplus _{q\geq 0}\bH^q(M,\sigma)$, where each $\bH^q_{top}(M,\sigma)$ is a $\Z$--graded $\Z[U]$--module.
Probably the presence of this  additional $U$--action is the most outstanding property compared with the usual cohomology theories.

Conjecturally (see \cite{Nlattice})
$\bH^*_{top}(M)$  is isomorphic to the Heegaard Floer cohomology $HF^+$ of Ozsv\'ath and Szab\'o  (which is defined for any 3--manifold),
for  $HF$--theory see their
 long list of articles, e.g.
\cite{OSz,OSz7}. This conjecture was verified for several families of plumbed 3--manifolds (associated with
negative definite connected graphs),  cf. \cite{NOSz,OSSz3},
but the general case is still open.
(In fact,   Heegaard Floer theory is isomorphic with several other theories:
the Monopole Floer Homology of Kronheimer  and Mrowka, or  the Embedded Contact Homology of Hutchings.
They are based on different geometrical  aspects of the 3--manifold $M$.)
$\bH^*_{top}$  is the categorification of the Seiberg--Witten invariant (similarly as $HF^+$ is).
For several properties of the lattice cohomology and applications in singularity theory see \cite{NOSz,NSurgd,NGr,Nexseq,NeLO}. For its connection with the classification of
projective rational plane cuspidal curves (via superisolated surface  singularities) see \cite{NSurgd,BLMN2,BodNem,BodNem2,BCG,BL1}.
 It  provides sharp topological bounds for certain sheaf cohomologies (e.g. for $p_g$), see e.g.  \cite{NSig,NSigNN}.
An improvement of $\bH^0_{top}$ is the set of graded roots parametrized by the spin$^c$--structures of $M$ \cite{NOSz,NGr}
 (they have no analogues for general arbitrary 3--manifolds). The  graded root is a special
tree with  $\Z$--graded vertices, it  provides
a very visual presentation of  $\bH^0_{top}$ (e.g., the $U$--action is coded in the edges).
 Hence, in particular it visualizes   $HF^+$ too,
when the  Heegaard Floer homology is known to be isomorphic to $\bH^0_{top}$ (see e.g. \cite{NOSz}). In such cases
the use of   graded roots  is significantly more convenient than any other method, see e.g.
 \cite{DM,K1,K2,K3}.

\bekezdes  Our goal is to define the corresponding analytic lattice cohomology associated with the analytic type of a normal surface singularity. The definition
and the development of the first key properties
(some of them under some topological  assumptions, e.g  that the link is a rational homology sphere)
are presented in this article.
In fact, in the present note
 we  present the analytic lattice cohomology  associated with the {\it canonical spin$^c$--structure $\sigma_{can}$} only.
It will be denoted by $\bH^*_{an,0}(X,o)$.
 The case of  other spin$^c$--structures will be treated in \cite{AgNeIII}.
(For the general part,
 under the assumption that the link is a rational homology sphere,
 we need to generalize  the present constructions to the level  of the universal abelian covering of $(X,o)$ and we also to recall several technical
details of  the theory of  the `natural line bundles'  of a
resolution. Therefore,  we decided to separate these technical parts  into \cite{AgNeIII}.)
Accordingly, in the present note, in the presentation of the topological lattice cohomology  we
restrict ourself  only to  the  module (summand)
which corresponds
to $\sigma_{can}$.

In this article, under the  assumption that all the exceptional curves of a resolution are rational, we show that
the analytic lattice cohomology $\bH^*_{an,0}$ just constructed  is a categorification of the
geometric genus. (This means that the Euler characteristic of the cohomology is the geometric genus.)
 We also show that it admits a graded $\Z[U]$--module morphism $\bH^*_{an}(X,o)\to \bH^*_{top}(M, \sigma_{can})$ (which in some
`nice' cases is an isomorphism). Furthermore, through   several examples we show that it
 is sensitive to the change of the analytic structure. In fact, for several fixed topological types we even  classify the
possible graded $\Z[U]$--modules $\bH^*_{an,0}$ associated with all the possible analytic structures supported on that topological type.

We are certain  that the new theory will have similar power and applicability as the topological version
(or, as the $HF$--theory),
with the difference that  in this case
its  applications will be mostly in the analytic theory of singularities.
 E.g. (as several examples
show already  in this article, see also \cite{AgNeCurves}) it is deeply related with the deformation of singularities.

\bekezdes In order to define a lattice cohomology one needs a  free $\Z$--module with a fixed basis and a weight
function (with certain properties) defined on the lattice points. In our case the lattice is that  of
the  divisors supported on the exceptional curve of a resolution. In the topological case the weight function is provided by the
Riemann--Roch expression. In the analytic case it is a combination of the coefficients of the Hilbert function associated with the divisorial filtration and the dimension  of a (sheaf)
cohomology  vectorspace. In both cases one proves that the output lattice cohomology is independent of the choice of the resolution.

In the Gorenstein case (in the presence of a certain symmetry) the analytic weight function can be deduced merely from the
coefficients of the Hilbert function of the divisorial filtration.

An important observation is that the general construction/definition  of the lattice cohomology  is very flexible: by providing different weight functions one
obtains different lattice cohomologies, and indeed, there are  several possibilities to construct weight functions
with remarkable associated cohomologies. This is
suggested already in the present note  by considering different filtrations (e.g. Newton filtration) and the associated weight functions.

The construction can be generalized  to higher dimensional complex normal isolated singularities and also
to  curve singularities (here the divisorial filtration is replaced by the valuative one), cf. \cite{AgNeCurves,AgNeHigh}.
In  the higher dimensional
case it is the categorification of $h^{n-1}(\calO_{\tX})$
 (as in the surface case), while in the curve case it is the categorification of the delta--invariant.

\bekezdes  The structure of the article and some of the main results are the following:
Section \ref{s:Prem1} contains the general definition of the lattice cohomology and graded root
 associated with a weight function, we follow \cite{NOSz,NGr,Nlattice}. Here we also recall certain  basics regarding
the topological lattice cohomology, the key Reduction Theorem \cite{LN1}
(when we can reduce the rank of the lattice according to
`bad' vertices), and several needed technical statements regarding the case of `almost rational graphs'
\cite{NOSz}.
Section \ref{ss:NSSAnlattice} (after some analytic preliminaries, vanishing theorems, etc.)  provides the definition of the
analytic lattice cohomology associated with a normal surface singularity via one of its resolutions. We prove
that it is independent of the resolution, and  when 
the link is a rational homology sphere we prove that
 it  is the categorification of
the geometric genus. We also prove an Analytic Reduction Theorem which allows us to compute it in
various  cases,
e.g. for rational, weighted homogeneous, Gorenstein elliptic, and certain superisolated  germs
(the special properties of almost rational  graphs are used in this last case).
We also determine it for the generic analytic structure. On the other hand, by several examples we show how it indicates the
variation of the analytic structure. In \ref{bek:DEFAN} we connect it (by a Conjecture) with $p_g$--constant flat deformations.

Section \ref{ss:CombLattice} contains a combinatorial setup of the theory: several proofs which depend basically only on these
combinatorial properties  are separated here. This part can and will be applied to several other weight functions in the forthcoming  articles of the series.

In the body of the article we also
present several examples and problems/conjectures regarding the new theory.

\section{Preliminaries. Basic properties of lattice cohomology}\label{s:Prem1}

\subsection{The lattice cohomology associated with  a weight function}\label{ss:latweight} \cite{NOSz,Nlattice}

\bekezdes
 We consider a free $\Z$-module, with a fixed basis
$\{E_v\}_{v\in\calv}$, denoted by $\Z^s$, $s:=|\calv|$.
Additionally, we consider a {\it weight  function} $w_0:\Z^s\to \Z$ with the property
\begin{equation}\label{9weight}
\mbox{for any integer $n\in\Z$, the set $w_0^{-1}(\,(-\infty,n]\,)$
is finite.}\end{equation}
%
%

\bekezdes\label{9complex} {\bf The weighted cubes.}
The space
$\Z^s\otimes \R$ has a natural cellular decomposition into cubes. The
set of zero-dimensional cubes is provided  by the lattice points
$\Z^s$. Any $l\in \Z^s$ and subset $I\subset \calv$ of
cardinality $q$  defines a $q$-dimensional cube $\square_q=(l, I)$, which has its
vertices in the lattice points $(l+\sum_{v\in I'}E_v)_{I'}$, where
$I'$ runs over all subsets of $I$.
 The set of $q$-dimensional cubes  is denoted by $\calQ_q$ ($0\leq q\leq s$).

Using $w_0$ we define
$w_q:\calQ_q\to \Z$  ($0\leq q\leq s$) by
$w_q(\square_q):=\max\{w_0(l)\,:\, \mbox{$l$ is a vertex of $\square_q$}\}$.
(For a more general definition of a `system of weight functions' when the system $w_q(\square_q)$ is not
determined by $w_0$, see \cite{Nlattice}, that generality
will be not used here.)

For each $n\in \Z$ we
define $S_n=S_n(w)\subset \R^s$ as the union of all
the cubes $\square_q$ (of any dimension) with $w(\square_q)\leq
n$. Clearly, $S_n=\emptyset$, whenever $n<m_w:=\min\{w_0\}$. For any  $q\geq 0$, set
$$\bH^q(\R^s,w):=\oplus_{n\geq m_w}\, H^q(S_n,\Z)\ \ \mbox{and}\ \
\bH^q_{red}(\R^s,w):=\oplus_{n\geq m_w}\, \widetilde{H}^q(S_n,\Z).$$
Then $\bH^q$ is $\Z$ (in fact, $2\Z$)-graded, the
$d=2n$-homogeneous elements 
consist of  $H^q(S_n,\Z)$.
Also, $\bH^q$ is a $\Z[U]$-module; the $U$-action is given by
the restriction map $r_{n+1}:H^q(S_{n+1},\Z)\to H^q(S_n,\Z)$.
Namely,  $U*(\alpha_n)_n=(r_{n+1}\alpha_{n+1})_n$. The same is true for $\bH^*_{red}$.
 Moreover, for
$q=0$, there exists  an augmentation
(splitting)
 $H^0(S_n,\Z)\simeq
\Z\oplus \widetilde{H}^0(S_n,\Z)$, hence an augmentation of the graded
$\Z[U]$-modules
$$\bH^0\simeq\calt^+_{2m_w}\oplus \bH^0_{red}=(\oplus_{n\geq m_w}\Z)\oplus (
\oplus_{n\geq m_w}\widetilde{H}^0(S_n,\Z))\ \ \mbox{and} \ \
\bH^*\simeq\calt^+_{2m_w}\oplus \bH^*_{red},$$
where $\calt_{2m}^+$  equals
$\Z\langle U^{-m}, U^{-m-1},\ldots\rangle$ as a $\Z$-module and it has the natural $U$--multiplication.

Though
$\bH^*_{red}(\R^s,w)$ has finite $\Z$-rank in any fixed
homogeneous degree, in general,  without certain additional properties of $w_0$, it is not
finitely generated over $\Z$, in fact, not even over $\Z[U]$.

\bekezdes\label{9SSP} {\bf Restrictions.} Assume that $T\subset \R^s$ is a subspace
of $\R^s$ consisting of a union of some cubes (from $\calQ_*$). For any $q\geq 0$ define $\bH^q(T,w)$ as
$\oplus_{n\geq\min{w_0|T}} H^q(S_n\cap T,\Z)$. It has a natural graded $\Z[U]$-module
structure.  The restriction map induces a natural  graded (degree zero)
$\Z[U]$-module  homomorphism
$r^*:\bH^*(\R^s,w)\to \bH^*(T,w)$.
In our applications to follow, $T$ --- besides the trivial $T=\R^s$ case --- will be one of the following:
%
(a) the first quadrant $(\R_{\geq o})^s$,
(b) the rectangle $[0,c]=\{x\in \R^s\,:\, 0\leq x\leq c\}$ for some lattice point $c\geq 0$,
(c) a path of composed edges in the lattice, cf. \ref{9PCl} and \ref{bek:path}.

\bekezdes \label{9F} {\bf The `Euler characteristic' of $\bH^*$.}
Fix $T$ as in  \ref{9SSP} and we will assume that each $\bH^*_{red}(T,w)$ has finite $\Z$--rank.
%
The Euler characteristic of $\bH^*(T,w)$ is defined as
$$eu(\bH^*(T,w)):=-\min\{w(l)\,:\, l\in T\cap \Z^s\} +
\sum_q(-1)^q\rank_\Z(\bH^q_{red}(T,w)).$$

\begin{lemma}\cite{NJEMS}\label{bek:LCSW} If $T=[0,c]$ for a lattice point $c\geq 0$, then
\begin{equation}\label{eq:Ecal}
 \sum_{\square_q\subset T} (-1)^{q+1}w_k(\square_q)=eu(\bH^*(T,w)).\end{equation}
 \end{lemma}

\subsection{Path lattice cohomology}\label{9PCl}\cite{Nlattice}

\bekezdes \label{bek:pathlatticecoh}
Fix $\Z^s$  as in \ref{ss:latweight} and fix also the weight functions
 $\{w_q\}_q$   as in \ref{9weight}. 
 Consider also a sequence $\gamma:=\{x_i\}_{i=0}^t$  so that $x_0=0$,
$x_i\not=x_j$ for $i\not=j$, and $x_{i+1}=x_i\pm E_{v(i)}$ for
$0\leq i<t$. We write $T$ for the
union of 0-cubes marked by the points $\{x_i\}_i$ and of the
segments of type  $[x_i,x_{i+1}]$.
Then, by \ref{9SSP} we get a graded $\Z[U]$-module $\bH^*(T,w)$,
which is called the {\em
path lattice cohomology} associated with the `path' $\gamma$ and weights
$\{w_q\}_{q=0,1}$. It is denoted by $\bH^*(\gamma,w)$.
It has an augmentation with $\calt^+_{2m_\gamma}$,
where $m_\gamma:=\min_i\{w_0(x_i)\}$, hence  one also gets the {\em reduced path lattice
cohomology} $\bH^0_{red}(\gamma,w)$ with
$$\bH^0(\gamma,w)\simeq \calt_{2m_\gamma}^+\oplus
\bH^0_{red}(\gamma,w).$$
It turns out that  $\bH^q(\gamma,w)=0$ for $q\geq 1$ and  its `Euler characteristic' can be defined as  (cf.  \ref{9F})
\begin{equation}\label{eq:euh0}
eu(\bH^*(\gamma,w)):=-m_\gamma+\rank_\Z\,(\bH^0_{red}(\gamma,w)).\end{equation}

\begin{lemma} \label{9PC2}
One has the following expression of $eu(\bH^*(\gamma,w))$ in terms of the values of $w_0$:
\begin{equation}\label{eq:pathweights}
eu(\bH^*(\gamma,w))=-w_0(0)+\sum_{i=0}^{t-1}\,
\max\{0, w_0(x_{i})-w_0(x_{i+1})\}.
\end{equation}
\end{lemma}

\begin{remark} It is convenient to
 compare $\bH^*(T,w)$ and $\bH^0(\gamma,w)$ for certain (big) rectangles  $T$, or for any $T$. In such cases it is convenient to
 consider the following `truncated Euler characteristic' of $\bH^*(T,w)$ as well (as an analogue of (\ref{eq:euh0}),
 even if $\bH^{\geq 1}(T,w)\not=0$):
 $$eu(\bH^0(T,w)):=-\min\{w(l)\,:\, l\in T\cap \Z^s\} +
\rank_\Z(\bH^0_{red}(T,w)).$$
\end{remark}

\subsection{Graded roots and their cohomologies}\label{s:grgen} \cite{NOSz,NGr}

\begin{definition}\label{def:2.1} \
  Let $\RR$ be an infinite tree with vertices $\calv$ and edges
$\cale$. We denote by $[u,v]$ the edge with
 end-vertices  $u$ and $v$.  We say that $\RR$ is a {\em graded root}
with grading $\chic:\calv\to \Z$ if

(a) $\chic(u)-\chic(v)=\pm 1$ for any $[u,v]\in \cale$;

(b) $\chic(u)>\min\{\chic(v),\chic(w)\}$ for any $[u,v],\
[u,w]\in\cale$, $v\neq w$;

(c) $\chic$ is bounded below, $\chic^{-1}(n)$ is finite for any
$n\in\Z$, and $|\chic^{-1}(n)|=1$ if $n\gg 0$.

%

\vspace{1mm}

\noindent
An isomorphism of graded roots is a graph isomorphism, which preserves the gradings.
\end{definition}


\begin{examples}\label{ex:2.3e} \

(1) For any  $n\in\Z$, let
$\RR_{(n)}$ be the tree with $\calv=\{v^{k} \}_{ k\geq n}$ and
$\cale=\{[v^{k},v^{k+1}]\}_{k\geq n}$. The grading is
$\chic (v^{k})=k$.

\vspace{2mm}

(2) Let $I$ be a finite index set. For each $i\in I$ fix  an
integer $n_i\in \Z$; and for each pair $i,j\in I$ fix
$n_{ij}=n_{ji}\in\Z$ with the next properties:

(i) $n_{ii}=n_i$; \
(ii) $n_{ij}\geq \max\{n_i,n_j\}$; and
\
(iii) $n_{jk}\leq \max\{n_{ij},n_{ik}\}$ for any $ i,j,k\in I$.

For any $i\in I$ consider $\RR_i:=\RR_{(n_i)}$ with vertices
$\{v_i^{k}\}$ and edges $\{[v_i^{k},v_i^{k+1}]\}$, $(k\geq n_i)$.
In the disjoint union $\sqcup_i\RR_i$,  for any pair $(i,j)$,
 identify $v_i^{k}$ and $v_j^{k}$,
resp.\ $[v_i^{k},v_i^{k+1}]$  and $[v_j^{k},v_j^{k+1}]$, whenever
$k\geq n_{ij}$. Write $\bar{v}_i^{k}$ for the class of $v_i^k$.
Then $\sqcup_i\RR_i/_\sim$ is a graded root with
$\chic(\bar{v}_i^{k})=k$. It will be denoted by
$\RR=\RR(\{n_i\},\{n_{ij}\})$.

%

\vspace{2mm}

(3) Any map $\tau:\{0, 1,\ldots,T_0\}\to \Z$ produces a starting
data for construction (2). Indeed, set $I=\{0,\ldots,T_0\}$,
$n_i:=\tau(i)$ ($i\in I$), and $n_{ij}:=\max\{n_k\,:\, i\leq k\leq
j\}$ for $i\leq j$. Then  $\sqcup_i\RR_i/_\sim $ constructed in (2)
using this data will be denoted by $(\RR_\tau,\chic_\tau)$.

This construction can be extended to the case of a map $\tau:\N\to \bZ$, whenever
$\tau$  has the property that there exists some $k_0\geq 0$ such that $\tau(k+1)\geq \tau(k)$ for any $k\geq k_0$.
In this case one can take any $T_0\geq k_0$ and construct the root associated with the
restriction of $\tau$ to $\{0,\ldots, T_0\}$.
It is independent of the choice of $T_0$ (since all
$\RR_k$ contributions for $k\geq k_0$ are superfluous). By definition, this is the root associated with $\tau$.
\end{examples}

\begin{definition}\label{9.2.6}{\bf The
  $\Z[U]$-modules associated with a graded root.}
Let us identify a graded root $(\RR,\chic)$ with its topological realization provided by vertices (0--cubes)
and segments (1--cubes). Define $w_0(v)=\chic(v)$, and $w_1([u,v])=\max\{\chic(u),\chic(v)\}$ and let $S_n$ be the
union of all cubes with weight $\leq n$. Then  we might set (as above) $\bH^*(\RR,\chi)=\oplus_{n\geq \min\chic}\
H^*(S_n,\Z)$. However, at this time $\bH^{\geq 1}(\RR,\chic)=0$; we set $\bH(\RR,\chic):=\bH^0(\RR,\chic)$.
Similarly, one defines $\bH_{red}(\RR,\chic)$ using the reduced cohomology, hence
$\bH(\RR,\chic)\simeq\calt_{2\min \chic}^+\oplus \bH_{red}(\RR,\chic)$.
\end{definition}

For a detailed concrete description of $\bH(\RR)$ in terms of the combinatorics of the root see \cite{NOSz}.

\begin{example}\label{9.2.9} (a) $\bH(\RR_n)=\calt_{2n}^+$.

(b) Let $(\RR_\tau,\chic_\tau)$ be
a graded root associated with  some  function $\tau:\N\to \Z$,
cf.\ \ref{ex:2.3e}(3).  Then
$$\rank_\Z \bH_{red}(\RR_\tau,\chic_\tau)=-\tau(0)+\min_{i\geq 0}\tau(i)+\sum_{i\geq 0}\,
\max\{ \tau(i)-\tau(i+1),0\}.$$
\end{example}

\bekezdes\label{bek:GRootW}{\bf The graded root associated with a weight function.}
Fix a free $\Z$-module and a system of weights $\{w_q\}_q$.
Consider the sequence of  topological  spaces (finite cubical  complexes) $\{S_n\}_{n\geq m_w}$
with $S_n\subset S_{n+1}$, cf. \ref{9complex}.
Let $\pi_0(S_n)=\{\calC_n^1,\ldots , \calC_n^{p_n}\}$ be the set of connected components of $S_n$.

Then  we define the  graded graph  $(\RR_w,\chic_w)$ as follows. The
vertex set  $\calv(\RR_w)$ is $\cup_{n\in \Z} \pi_0(S_n)$.
The grading $\chic_w:\calv(\RR_w)\to\Z$ is
$\chic_w(\calC_n^j)=n$, that is, $\chic_w|_{\pi_0(S_n)}=n$.
Furthermore, if  $\calC_{n}^i\subset \calC_{n+1}^j$ for some $n$, $i$ and $j$,
then we introduce an edge $[\calC_n^i,\calC_{n+1}^j]$. All the edges
 of $\RR_w$ are obtained in this way.
\begin{lemma}\label{lem:GRoot} $(\RR_w,\chic_w)$ satisfies all the required properties
of the definition of a graded root, except possibly  the last one: $|\chic_w^{-1}(n)|=1$ whenever $n\gg 0$.
\end{lemma}

The property  $|\chic_w^{-1}(n)|=1$ for $n\gg 0$ is not always satisfied.
However, the graded roots associated with connected negative definite plumbing graphs
(see below) satisfies this condition as well.

\begin{proposition}\label{th:HHzero} If $\RR$ is a graded root associated with $(T,w)$ and
 $|\chic_w^{-1}(n)|=1$ for all $n\gg 0$ then $\bH(\RR)=\bH^0(T,w)$.
\end{proposition}

\section{Surface singularities and the topological lattice cohomology}\label{s:SSTLC}

\subsection{The combinatorics  of the resolutions}\label{s:prel}
\cite{Nfive,NOSz,NGr}

\bekezdes
Let $(X,o)$ be the germ of a complex analytic normal surface singularity with link $M$.
Let $\phi:\widetilde{X}\to X$ be a good   resolution   of $(X,o)$ with
 exceptional curve $E:=\phi^{-1}(0)$,   and  let $\cup_{v\in\calv}E_v$ be
the irreducible decomposition of $E$. (By good resolution we mean that each $E_v$ is smooth, and $E$ is a normal crossing divisor.)
 Let $\Gamma$ denote  the dual resolution graph of $\phi$.  Note that $\partial \tX\simeq M$.


The lattice $L:=H_2(\widetilde{X},\mathbb{Z})$ is  endowed
with the natural  negative definite intersection form  $(\,,\,)$. It is
freely generated by the classes of  $\{E_v\}_{v\in\mathcal{V}}$.
 The dual lattice is $L'={\rm Hom}_\Z(L,\Z) \simeq\{
l'\in L\otimes \Q\,:\, (l',L)\in\Z\}$.
$L$ is embedded in $L'$ with
 $ L'/L\simeq {\rm Tors}( H_1(M,\mathbb{Z}))$, which is abridged by $H$.
 The class of $l'$ in $H$ is denoted by $[l']$.

We define the Lipman cone as $\calS':=\{l'\in L'\,:\, (l', E_v)\leq 0 \ \mbox{for all $v$}\}$, and we  also
set $\calS:=\calS'\cap L$. If $s'\in\calS'\setminus \{0\}$ then
all its $E_v$--coordinates  are strictly positive.

There is a natural partial ordering of $L'$ and $L$: we write $l_1'\geq l_2'$ if
$l_1'-l_2'=\sum _v r_vE_v$ with every  $r_v\geq 0$. We set $L_{\geq 0}=\{l\in L\,:\, l\geq 0\}$ and
$L_{>0}=L_{\geq 0}\setminus \{0\}$.
The support of a cycle $l=\sum n_vE_v$ is defined as  $|l|=\cup_{n_v\not=0}E_v$.

If  $H_1(M,\Q)=0$ then
each $E_v$ is rational, and the dual graph of any good resolution is a tree.

Artin's fundamental cycle $Z_{min}\in L_{>0}$ is defined as the smallest non-zero cycle in $\calS$
 (with respect to $\geq $).

The {\it (anti)canonical cycle} $Z_K\in L'$ is defined by the
{\it adjunction formulae}
$(Z_K, E_v)=(E_v,E_v)+2-2g_v$ for all $v\in\mathcal{V}$, where
$g_v$  denotes the genus of $E_v$.
(The cycle $-Z_K$  is the first Chern class of the line bundle $\Omega^2_{\tX}$.)
We write $\chi:L'\to \Q$ for the (Riemann--Roch) expression $\chi(l'):= -(l', l'-Z_K)/2$.

The singularity (or, its topological type) is called numerically Gorenstein if $Z_K\in L$.
(Since $Z_K\in L$ if and only if the line bundle $\Omega^2_{X\setminus \{o\}}$ of holomorphic 2--forms
on $X\setminus \{o\}$ is topologically trivial, see e.g. \cite{Du}, the $Z_K\in L$ property
is independent of the resolution). $(X,o)$ is called Gorenstein if $Z_K\in L$ and
$\Omega^2_{\tX}$ (the sheaf of holomorphic 2--forms) is isomorphic to $ \calO_{\tX}(-Z_K)$ (or,
equivalently, if the line bundle $\Omega^2_{X\setminus \{o\}}$ is holomorphically trivial).
If $\tX$ is a minimal resolution then (by the adjunction formulae) $Z_K\in \calS'$.

\subsection{The topological lattice cohomology associated with $\phi:\tX\to X$}\label{s:latticeplgraphs}\cite{NOSz,Nlattice}

\bekezdes\label{9dEF1} We consider a good resolution $\phi$ as above
and we assume that the link  $M$ is a
rational homology sphere. We write $s:=|\cV|$.
Then we automatically have  a free $\Z$-module $L=\Z^s$ with
a fixed bases $\{E_v\}_v$. The Riemann--Roch expression $\chi$ defines a weight function
$w_0(l)=\chi(l)$, hence  a set
of  weight functions
$ w_q(\square_q)=\max\{\chi(v)\,:\, v\ \mbox{is a vertex of
$\square_q$}\}$.

\begin{definition}\label{9DEF}  The  $\Z[U]$-modules $\bH^*(\R^s,w)$ and
$\bH^*_{red}(\R^s,w)$ obtained in this way are
called the {\em topological lattice cohomologies} associated with the canonical spin$^c$--structure.
They are denoted  by $\bH^*(\Gamma,-Z_K)$, respectively
$\bH^*_{red}(\Gamma,-Z_K)$ (or $\bH^*_{top,0}(M)$ and $\bH^*_{top,red,0}(M)$ respectively).

The same weight function defines a graded root $(\RR(\Gamma, -Z_K),\chic)=\RR_{top,0}(M)$ as well.
It is called the
{\it topological graded root} associated with the canonical spin$^c$--structure.
\end{definition}

\begin{proposition}\label{9STR2}
(a) $\bH^*_{red}(\Gamma,-Z_K)$ is finitely generated over $\Z$.

(b) The degree $d>0$ summands  of \ $\bH^0_{red}(\Gamma,-Z_K)$ are zero.

(c)  $\bH^*(\Gamma,-Z_K)$  and $(\RR(\Gamma, -Z_K),\chic)$
depend only on $M$ and they are independent of the choice of the good resolution $\phi$.
The root  $(\RR(\Gamma, -Z_K),\chic)$  satisfies the property  $|\chic^{-1}(n)|=1$ for $n\gg 0$.

(d) The restriction
$\bH^*(\Gamma,-Z_K)\to \bH^*((\R_{\geq 0})^s,-Z_K)$
induced by the inclusion $(\R_{\geq 0})^s\hookrightarrow \R^s$ is an isomorphism
of graded $\Z[U]$ modules.
\end{proposition}

\bekezdes {\bf The Euler characteristic and the Seiberg--Witten invariant.}\label{s:LCSW} \
The Seiberg--Witten invariant ${\rm Spin}^c(M)\to \Q$ associates a rational number
$\sw _{\sigma}(M)$  to each spin$^c$--structure $\sigma$ of the link.

\begin{theorem}\label{th:ECharLC} \cite{NJEMS} \ 
$eu(\bH^*(\Gamma,-Z_K))=\sw_{\sigma_{can}}(M)-(Z_K^2+|\cV|)/8$.
\end{theorem}
In other words, {\it 
$\bH^*(\Gamma,-Z_K)$
is the categorification of $\sw_{\sigma_{can}}(M)$    (normalized by $(Z_K^2+|\cV|)/8$)}.

\bekezdes\label{bek:SWIC}
The geometric genus of any normal surface singularity $(X,o)$ is defined as
  $h^1(\tX, \calO_{\tX})$, and it is denoted by $p_g$.
($h^1(\tX, \calO_{\tX})$ is independent of the choice of $\tX$.)
 A priori (and usually) it is an analytic invariant.

We say, following \cite{NJEMS,NeNi,NeNiII},
 that the singularity $(X,o)$ (with rational homology link)
 satisfies the {\bf Seiberg--Witten Invariant Conjecture (SWIC) for the
 canonical spin$^c$--structure} if $p_g(X,o)=eu(\bH^*(\Gamma,-Z_K))$, i.e., $p_g$ can be characterized topological by
 the normalized Seiberg--Witten invariant.

 The SWIC  is satisfied in the following cases:
 weighted homogeneous singularities, superisolated singularities associated with rational unicuspidal curves,
 suspension singularities of type $f(x,y)+z^n$ with $f$ irreducible, splice quotient singularities (e.g.
 rational or minimally elliptic singularities),  see \cite{BN,Spany,BLMN2,NJEMS,NeNi,NeNiII,NeNiIII,NO2}.

 \bekezdes\label{bek:pgbounds} {\bf Euler characteristic type bounds for $p_g$.}
 Though the SWIC (for the canonical spin$^c$-structure) is not true in general, there are some generally valid
 topological  bounds provided by lattice cohomology.

 Let ${\mathcal K}$ be the (topologically defined )
set of cycles $\lfloor Z_K\rfloor _++ L_{\geq 0}$. Here  $\lfloor Z_K\rfloor _+$ is the effective part of the integral part
of $Z_K$.
 By \ref{th:GR},   $h^1(\cO_l)=p_g$ for any $l\in {\mathcal K}$ and for any analytic structure supported by $\Gamma$, i.e.
 $h^1(\cO_l)$ in this zone already `stabilizes'. The lattice point  $\lfloor Z_K\rfloor _+$
is essential from topological point of view  as well:
 all the (path) lattice cohomologies can be computed already in the
 rectangle $R(0, \lfloor Z_K\rfloor _+)$, cf. \cite{NOSz,Nlattice}.

\bekezdes \label{bek:path}
In the sequel we denote by  $\cP$ the set of paths $\gamma=\{x_i\}_{i=0}^t$ with $x_0=0$ and arbitrary end-cycle
$x_t=c$ from ${\mathcal K}$.
A path is {\it increasing} if $\epsilon_i=+1$ for every $i$.

Let $\bH^0(\gamma,\Gamma, -Z_K)$ be the path lattice cohomology associated with $\gamma$ and the weight function
associated with $\chi$. 
Let $eu(\bH^0(\gamma,\Gamma,-Z_K))$ be its Euler characteristic defined as in (\ref{eq:euh0}).
It turns out (see e.g. \cite{NSig,NSigNN}, or \cite{Book})  that
$\min_{\gamma \in \cP}\, eu(\bH^0(\gamma,\Gamma,-Z_K))$ is realized by an increasing path,  and this expression
 is independent of the choice of this $c$, whenever $c\in {\mathcal K}$.
Furthermore,
\begin{equation}\label{cor:INEQeu}
p_g\leq \min_{\gamma \in \cP}\, eu(\bH^0(\gamma,\Gamma,-Z_K))\leq
 eu (\bH^0(\Gamma, -Z_K)).
\end{equation}

\begin{example}\label{ex:MINforwh}
The equality  $p_g=\min_{\gamma \in \cP}\, eu(\bH^0(\gamma,\Gamma,-Z_K))$ is realized
(under the assumption that  the link is a rational homology sphere)
by the following families:

(a)  singularities which  satisfies the
SWIC for the canonical spin$^c$ structure (cf. \ref{bek:SWIC})
and $\bH^{\geq 1}(\Gamma,-Z_K)=0$ (use (\ref{cor:INEQeu})). This includes e.g.
weighted homogeneous singularities,
rational singularities, or
maximally elliptic singularities (i.e. elliptic singularities which satisfy
 $p_g=\mbox{length of the elliptic sequence}\ \ell_{seq}$).

(b)  superisolated singularities, cf.  \cite{NSig};

(c) local Weil divisors in affine toric 3-varieties  with  nondegenerate Newton principal part, cf.
\cite{NSigNN}.

\end{example}

\begin{remark}\label{rem:SPLIT}
Assume that $p_g=\min_{\gamma \in \cP}\, eu(\bH^0(\gamma,\Gamma,-Z_K))$, and that
$\min_{\gamma \in \cP}\, eu(\bH^0(\gamma,\Gamma,-Z_K))$ is realized by the increasing
path $\gamma=\{x_i\}_{i=0}^t$,  $x_{i+1}=x_i+ E_{v(i)}$.
Then all the exact sequences $0\to \cO_{ E_{v(i)}}(-x_i)\to \cO_{x_{i+1}}\to \cO_{x_i}\to 0$ cohomologically must split
(for details see e.g. \cite{NSig,NSigNN}). In particular, $H^0(\cO_{x_{i+1}})\to H^0(\cO_{x_i})$ is surjective for every $i<t$.
\end{remark}

\subsection{Measure of non-rationality. `Bad' vertices}\label{ss:BadVer} \cite{NOSz,Book}\

Rational graphs play a distinguished role in this theory: they are the graphs with $\bH^*_{red}(\Gamma)=0$.
Usually, in the analysis of a graph $\Gamma$, we wish to understand how far is $\Gamma$ to be rational.

\begin{example}\label{ex:LatHomRat} {\bf Rational graphs.}
Recall that $(X,o)$ is called rational if $p_g=0$. By a result of Artin \cite{Artin62,Artin66}
$p_g=0$ if and only if
 $\chi(l)\geq 1$ for all $ l\in L_{>0}$  (hence it is a topological property of $M$  readable from $\Gamma$, those graph which satisfy it are called rational).
The links of
 any  rational singularity is a
 rational homology sphere. 
 The class of  rational  graphs is closed while taking subgraphs or/and
decreasing the Euler numbers $E_v^2$.



We have the following characterizations of the rationality in  terms of topological lattice
cohomology (or graded roots):
 $\Gamma$ is rational $\Leftrightarrow$
 $\bH^0_{red}(\Gamma,-Z_K)=0$ $\Leftrightarrow$
$\bH^*_{red}(\Gamma,-Z_K)=0$ $\Leftrightarrow$   $\RR_{-Z_K}=\RR_{(0)}$, cf. \cite{NOSz}.
\end{example}

If $M$ is a $\Q HS^3$, then by
decreasing all the Euler numbers of $\Gamma$
we  obtain a rational graph.
The next definition aims to
identify those vertices where such a decrease is really necessary.

\begin{definition}\label{def:SWrat} Let $\Gamma$ be a resolution graph such that $M$ is a rational homology sphere.
A subset of vertices
$\ocalj=\{v_1,\ldots, v_{\overline{s}}\}\subset \cV$ is called {\it B--set}, (set of `bad vertices')
if by replacing the Euler numbers
 $e_v=E_v^2$ indexed by $v\in \ocalj$ by some more negative integers
$e'_v\leq e_v$ we get a rational graph.

A graph is called AR-graph (`almost rational graph') if it admits a B--set of cardinality $\leq 1$.
\end{definition}
\begin{example}\label{ex:*sets}
(a) A possible B--set can be chosen in many different ways, usually
it is not determined uniquely even if it is minimal with this property.
Usually we allow non-minimal $B$--sets as well.

(b) If $H_1(M,\Q)=0$ then the set of nodes is a B--set.
Hence  any  star-shaped graph (with $H_1(M,\Q)=0$) is AR.
Other  AR families are: rational and elliptic graphs and graphs of superisolated singularities associated with a
rational unicuspidal curve \cite{NOSz,NGr}.

(c)
The class of  AR graphs is closed while taking subgraphs or/and
decreasing the Euler numbers.

\end{example}

\bekezdes\label{bek:XI}
 {\bf The definition of the lattice points $x(\bar{l})$.}
Assume that $\ocalj:=\{v_k\}_{k=1}^{\overline{s}}$ is any subset of $\calj$.
Then we split the set of vertices $\calj$ into the disjoint union $\overline{\calj}\sqcup\calj^*$.
Let $\{m_v(x)\}_v$ denote  the coefficients of a cycle $x\in L\otimes \setQ$, that is $x=\sum_{v\in\calj}m_v(x)E_v$.

 Our goal is to define some universal cycles  $x(\bar{l})\in L$
 associated with $\bar{l}\in L(\ocalj)$.

\begin{proposition}\label{lemF1} \ \cite[Lemma 7.6]{NOSz}, \cite{LN1}
For any
$\bar{l}:=\sum_{v\in \ocalj}\ell_v
E_v\in L(\ocalj)$
there exists a unique cycle
$x(\bar{l})\in L$
satisfying the next properties:
\begin{itemize}
\item[(a)] \ \ $m_{v}(x(\bar{l}))=\ell_v$ \ for any distinguished  vertex $v\in\ocalj$;
\item[(b)] \ \ $(x(\bar{l}),E_v)\leq0$ \ for every `non-distinguished vertex' $v\in\calj^*$;
\item[(c)] \ \ $x(\bar{l})$ is minimal with the two previous properties (with respect to $\leq$).
\end{itemize}
\end{proposition}

\bekezdes
 Let us   fix  a B--set $\ocalj\subset \calv$
as in \ref{def:SWrat} (with cardinality $\bar{s}$).
Our goal is to provide an equivalent description of the lattice cohomology using cubes in an $\bar{s}$--dimensional lattice.

Firs recall the isomorphism $\bH^*(\Gamma,-Z_K)\to \bH^*((\R_{\geq 0})^s,-Z_K)$ from
\ref{9STR2}{\it (d)}.

Having this in mind,    for each $\bar{l}=\sum_{v\in \ocalj} \ell_vE_v\in L(\ocalj)$,
with every  $\ell_v\geq 0$, we define the  universal  cycle $x(\bar{l})$ associated with
$\bar{l}$ as in \ref{lemF1}.
Then,  
define   the function
$\overline{w}_0:(\Z_{\geq 0})^{\bar{s}}\to \Z$ by
$\overline{w}_0(\bar{l}):=\chi(x(\bar{l}))$.
 Then $\overline{w}_0$ defines a set $\{\overline{w}_q\}_{q=0}^{\bar{s}}$ of weight functions  as in \ref{9dEF1} by
$\overline{w}_q(\square)=
 \max\{\overline{w}_0(v)\,:\, v \ \mbox{ is a vertex of $\square$}\}$.

\begin{theorem}\label{th:red} {\bf (Topological Reduction Theorem)}  \cite{LN1}
 There exists  a graded $\Z[U]$-module isomorphism
\begin{equation}\label{eq:reda}
\bH^*((\R_{\geq 0})^s,-Z_K)\cong\bH^*((\R_{\geq 0})^{\bar{s}},\overline{w}) \ \ \mbox{and} \
\ \RR((\R_{\geq 0})^s,-Z_K)\cong \RR((\R_{\geq 0})^{\bar{s}},\overline{w}).
\end{equation}
\end{theorem}

\subsection{Concatenated  computation  sequences of AR graphs}\label{ss:CCSAR} \cite{NOSz}

\bekezdes Assume  that $\Gamma $ is an AR  resolution graph. Let $\{v_0\}$ be an B--set.


Note that by Reduction Theorem \ref{th:red} $\bH^{\geq 1}(M(\Gamma), -Z_K)=0$.
The Reduction  Theorem provides a formula for $\bH^{0}(\Gamma, -Z_K)$ too using a lattice cohomology associated with
$\Z_{\geq 0}\subset \R_{\geq 0}$ and weight function $\Z_{\geq 0}\ni \ell\mapsto \chi(x(\ell))$. In the next discussion we present another (equivalent) version,
which represents
$\bH^{0}(\Gamma, -Z_K)$ as the lattice cohomology of an increasing
path $T=\gamma=\{x_i\}_{i\geq 0} $ embedded in the lattice $L_{\geq 0}$.
The point is that $\gamma $ determines the   1-chain $C_\gamma:=
\cup_{i\geq 0}[x_i,x_{i+1}]$ of  1-cubes in
$L\otimes \R$ (without any loop), such that $C_\gamma\cap S_{n} \hookrightarrow S_{n}$ is a homotopy equivalence.
In particular, all the connected components of $S_n$ (whenever $S_n\not=\emptyset$) are contractible.

The construction of $\gamma$ runs  as follows: it is
defined as a series of concatenated computation sequences.
 It contains,   as intermediate terms,  all the universal cycles $\{x(\ell)\}_{\ell\geq 0}$ in an
 increasing order.
 The first term is $x_0=x(0)=0$.
The part of the sequence starting from  $x(\ell)$  and ending with $x(\ell+1)$
starts with $x(\ell)$ and the next term is $x(\ell)+E_{v_0}$. Then, the continuation is a
generalized Laufer-type computation sequence connecting  $x(\ell)+E_{v_0}$
with $x(\ell+1)$. Indeed, the multiplicity of $E_0$ in both $x(\ell)+E_{v_0}$
and  $x(\ell+1)$ is $\ell+1$, and  by  the universal (minimality) property of $x(\ell)$ we get
 that  $x(\ell+1)\geq x(\ell)+E_{v_0}$.
Then, there is a (Laufer type) generalized  computation sequence, or increasing path, $\gamma^{(\ell+1)}=\{x_i^{(\ell+1)}\}_i$,
which connects
 $x(\ell)+E_{v_0}$ and  $x(\ell+1)$ (see \cite{NOSz}).
Then we proceed inductively.


In general, it is not easy to concretely identify  the cycles $x(\ell)$.
Fortunately, in several applications we only need
the values $\tau(\ell)=\chi(x(\ell))$. In
most of the cases they are computed inductively using \ref{lem:xellprops}{\it (d)}, hence
basically  one needs only to  know
 $(x(\ell),E_{v_0})$ for any $\ell$.


\begin{proposition}\label{lem:xellprops}\cite{NOSz}
(a) The path $\{x_i\}_i$ is  increasing: $x_{i+1}=x_i+E_{v(i)}$.

(b) For any $E_v$-coefficient one has  $\lim_{\ell\to \infty} m_v(x(\ell))=\infty$ (where $v\in\calv$).

(c)  $\chi$ along  each part (subsequence) $\gamma^{(\ell)}$ is constant.  This also implies that
$\bH^0((\R_{\geq 0})^{\bar{s}},\overline{w})$ from the reduction theorem equals  the path lattice cohomology
$\bH^0(\gamma, -Z_K)$).

(d)  Set $\tau(\ell)=\chi(x(\ell))$. Then $\tau(\ell+1)=\tau(\ell)+1-(x(\ell),E_{v_0})$ and
there exists $\ell_0$ such that $\tau(\ell+1)\geq \tau(\ell) $ for $\ell\geq \ell_0$.

(e) $eu(\bH^*(\Gamma,-Z_K) = eu(\bH^*(\gamma,-Z_K)
=\sum _{\ell\geq 0}
\max\{\, \tau(\ell)-\tau(\ell+1), 0\,\}$.
\end{proposition}






\section{The analytic lattice cohomology of surface singularities}\label{ss:NSSAnlattice}
\setcounter{equation}{0}

\subsection{Review  of some analytic properties}\label{ss:AnPrel}

\bekezdes\label{bek:PG}
 Let $(X,o)$ be a normal surface singularity and we fix a good resolution $\phi$.
In this subsection we go over some statements that will help in the later discussion and proofs.
We start with a vanishing theorem.

\begin{theorem}\label{th:GR} {\bf Generalized Grauert--Riemenschneider Theorem.}
\cite{GrRie,Laufer72,Ram,Book}
 Consider a line bundle $\cL\in {\rm Pic}(\widetilde{X})$ such that
$c_1(\cL(Z_K))\in \Delta -\cS_{\Q}$ for some
$\Delta\in L\otimes \Q$ with  $\lfloor \Delta \rfloor =0$.
Then $h^1(l,\cL|_{l})=0$ for any $l\in L_{>0}$.
In particular, $h^1(\widetilde{X},\cL)=0$ too. (Here $\calS_\Q$ denotes  the rational cone generated by $\calS$.)
\end{theorem}

In particular, if
$\calL\in {\rm Pic}(\tX)$ and $l\in L_{>0}$ satisfies $l\in c_1(\calL)+Z_K+\calS$,
then $h^1(\tX, \calL(-l))=0$, hence $H^1(\tX, \calL)=H^1(l, \calL|_l)$.

We denote by $\lfloor Z_K\rfloor $ the integral part of $Z_K$, and by
$\lfloor Z_K\rfloor_+ $ its effective part. The above statements imply the following.
If $\lfloor Z_K\rfloor_+=0$ then $p_g=0$.
If $\lfloor Z_K\rfloor_+>0$ then for
 any $Z\geq \lfloor Z_K\rfloor_+$, $Z\in L$,
$p_g=h^1(\cO_{Z})$.

Furthermore,  if $l\in \cS$ and $n\in\Z_{\geq 0}$ such that $nl+\lfloor Z_K\rfloor>0$ then
\begin{equation*}
\dim \frac{H^0(\cO_{\tX})}{H^0(\cO_{\tX}(-\lfloor Z_K\rfloor-nl))}=
\chi(\lfloor Z_K\rfloor+nl)+p_g.
\end{equation*}
This implies that  for $l\in \calS\setminus \{0\}$ and $n\gg 0$,
\begin{equation}\label{eq:vanh0}
\dim \frac{H^0(\cO_{\tX})}{H^0(\cO_{\tX}(-nl))}=
-\frac{n^2l^2}{2}+\mbox{lower order terms in $n$}.
\end{equation}

For certain cycles the Grauert--Riemenschneider Theorem \ref{th:GR} can be improved.

\begin{proposition}\label{prop:VAN0}\ {\bf Lipman's Vanishing Theorem.} \cite[Theorem 11.1]{Lipman}, \cite{Book}\
Take  $l\in L_{>0}$ with  $h^1(\cO_l)=0$  and $\cL\in {\rm Pic}(\tX)$ for which
 $(c_1\cL,E_v)\geq 0$ for any $E_v$ in the support of\, $l$. Then  $h^1(l,\cL)=0$.
 \ix{Vanishing Theorem!Lipman's|textbf}
\end{proposition}

\bekezdes
\label{rem:antmatroid} \cite[4.8]{MR} The set
$L_{p_g}:=\{l\in L_{>0}\, : \, h^1(\cO_l)=p_g\}$ has a unique minimal element, denoted by $Z_{coh}$, and called
the {\it cohomological cycle} of $\phi$. It has the property that
 $h^1(\cO_l)<p_g$
for any $l\not\geq Z_{coh}$ ($l>0$).

By the consequences of Theorem \ref{th:GR} we obtain that $Z_{coh}\leq  \lfloor Z_K\rfloor_+$.

In fact, the proof of the above statement in \cite{MR}  shows the following.
Let $l_1, l_2\in L_{>0}$ be effective cycles, set
$l=\min\{l_1, l_2\}$ and $\overline{l}=\max\{l_1,l_2\}$. Then
$h^1(\calO_{\overline{l}})+h^1(\calO_l)\geq h^1(\calO_{l_1})+h^1(\calO_{l_2})$.  We will refer
to this inequality as  the {\it `opposite' matroid rank inequality} of $h^1$.

In particular, for any $l\in L_{>0}$ we have $h^1(\calO_l)=h^1(\calO_{\min \{l,Z_{coh}\}})$.

\bekezdes\label{bek:LauferDual} \cite{Laufer72},
\cite[p. 1281]{Laufer77}
Following Laufer, we can identify the dual space $H^1(\tX,\cO_{\tX})^*$ with the space of global holomorphic
2-forms on $\tX\setminus E$ up to the subspace of those forms which can be extended
holomorphically over $\tX$:
$H^1(\tX,\cO_{\tX})^*\simeq 
H^0(\tX\setminus E,\Omega^2_{\tX})/ H^0(\tX,\Omega^2_{\tX})$. 
Here $H^0(\tX\setminus E,\Omega^2_{\tX})$ can be replaced by
$H^0(\tX,\Omega^2_{\tX}(Z))$ for
any $Z>0$ with  $h^1(\cO_Z)=p_g$. Indeed, for any  $Z>0$, from the exact sequence of sheaves
$0\to\Omega^2_{\tX}\to \Omega^2_{\tX}(Z)\to \calO_{Z}(Z+K_{\tX})\to 0$ (where $\Omega^2_{\tX}=\cO_{\tX}(K_{\tX})$)
and from the vanishing
$h^1(\Omega^2_{\tX})=0$ and Serre duality
\begin{equation}\label{eq:duality}
H^0(\Omega^2_{\tX}(Z))/H^0(\Omega^2_{\tX})=H^0(\calO_Z(Z+K_{\tX}))\simeq H^1(\calO_Z)^*.
\end{equation}
If  $H^1(\calO_Z)\simeq H^1(\calO_{\tX})$ then   the inclusion
$H^0(\Omega^2_{\tX}(Z))/H^0( \Omega^2_{\tX})\hookrightarrow
H^0(\tX\setminus E, \Omega^2_{\tX})/H^0(\Omega^2_{\tX})$
is an isomorphism.

\bekezdes \label{bek:San}
For each holomorphic function $f\in \cO_{X,o}$ consider
${\rm div}_E(f)\in \cS$, that part of the divisor of $f\circ \phi$ which is supported on $E$.
Then define  $\cS_{an}\subset \calS$, the {\it analytic semigroup of $\phi$},  as
$\{{\rm div}_E(f)\,:\, f\in \cO_{X,o}\}$.

The monoid $\calS_{an}\setminus \{0\}$ has a unique minimal element
$Z_{max}$ (${\rm div}_E$ of the generic element of the maximal ideal of $\cO_{X,o}$).
It is called {\it the maximal ideal cycle} of $\phi$. One has $Z_{max}\geq Z_{min}>0$.

\subsection{The analytic lattice cohomology, independence  of the choice of the  rectangle}\label{ss:anR}

\bekezdes
Next we  construct the {\it analytic lattice cohomology } $\bH^*_{an,0}(X,o)$
of a normal surface singularity $(X,o)$.
(Though we will mention in the sequel nothing
about the spin$^c$--structures,
this module in fact corresponds to the canonical spin$^c$--structure of the link.
If the link is a rational homology sphere,  the construction of the
analytic lattice cohomologies $\{\bH^*_{an, h}(X,o)\}_{h}$
associated with all spin$^c$--structures will be
presented in the forthcoming  article \cite{AgNeIII}.)

Let us fix a good resolution $\phi$. 

For any $c\in L$, $c\geq Z_{coh}$,  we consider the rectangle $R(0,c)=\{l\in L\,:\,
0\leq l\leq c\}$.
Here we might  consider the $c=\infty$ case too, in this case $R(0,c)=L_{\geq 0}$.
Then we consider the multivariable  Hilbert function
$$\hh:R(0,c)\to\Z, \ \  \hh(l)=\dim H^0(\calO_{\tX})/H^0(\calO_{\tX}(-l))$$
associated with the
divisorial filtration of $\calO_{X,o}$ and the resolution $\phi$, cf. \cite{CDG,CHR,NJEMS}.
Clearly $\hh$ is increasing (that is, $\hh(l_1)\geq \hh(l_2)$ whenever $l_1\geq l_2$) and  $\hh(0)=0$.

Next, for any $l\in R(0,c)$, we consider the function
$\hh^\circ(l)=p_g-h^1(\calO_l)$ too (where, by definition,  $h^1(\calO_{l=0})=0$). Then  $\hh^\circ$ is decreasing,
 $\hh^\circ (0)=p_g$  and $\hh^\circ (c)=0$, cf. \ref{bek:PG}. (Compare with the notations of section \ref{ss:CombLattice}.)

We consider the natural cube-decomposition of $R(0,c)$
 (where the 0-cubes  are the lattice points)  and
the set of cubes $\{\calQ_q\}_{q\geq 0}$ of $R(0,c)$ as in \ref{9complex}.
Then we define the weight function
\begin{equation}\label{eq:wean}
w_0:\calQ_0\to \Z, \ \ \  w_0(l)=\hh(l)+\hh^\circ (l)-p_g=\hh(l)-h^1(\calO_l).
\end{equation}
Clearly,   $w_0(0)=0$.
This weight function has several useful properties.

First of all, note that $0\leq \hh^\circ(l)\leq p_g$ for every $l$, hence when $c=\infty$ then
$\hh$ and $w_0$ have comparable  asymptotic behaviour for $l\gg 0$.
For any $l\in L$ let $s(l)\in L$  be the smallest element with $s(l)\geq l$ and $s(l)\in\calS$. Then
$\hh(l)=\hh(s(l))$. Using this fact, the monotonicity of $\hh$,  and
(\ref{eq:vanh0})
 a computation shows that  $w_0$ satisfies the requirement \ref{9weight}(a), namely,
 $w_0^{-1}((\-\infty, n]) \ \ \mbox{is finite for any $n\in\Z$}$.

 Next,  since $\hh$ is induced by a filtration, it satisfies the
matroid rank inequality
$\hh(l_1)+\hh(l_2)\geq  \hh(\overline{l})+\hh(l)$,
where $l=\min\{l_1, l_2\}$ and $\overline{l}=\max\{l_1,l_2\}$.
On the other hand, $h^1$ satisfies the `opposite' matroid rank inequality, see
\ref{rem:antmatroid}. Therefore, $w_0$ itself satisfies the matroid rank inequality (where $l_1,l_2\geq 0$)
  \begin{equation}\label{eq:matroidw00}
 w_0(l_1)+w_0(l_2)\geq  w_0(\overline{l})+w_0(l).
 \end{equation}
(As a topological comparison:  $\chi$, the topological weight function,  also satisfies a similar inequality.)

Furthermore, similarly as in \ref{9dEF1}, we define
$w_q:\calQ_q\to \Z$ by $ w_q(\square_q)=\max\{w_0(l)\,:\, l \
 \mbox{\,is any vertex of $\square_q$}\}$.
 In the sequel  we write $w$ for the system $\{w_q\}_q$ if there is no confusion.
 These  weight functions $\{w_q\}_q$ define the lattice cohomology $\bH^*(R(0,c),w)$ and the graded root
 $\RR(R(0,c),w)$
associated with $R(0,c)$ and $w$.

\begin{lemma}\label{lem:INDEPAN}
$\bH^*(R(0,c),w)$  and  $\RR(R(0,c),w)$ are  independent of the choice of $c\geq Z_{coh}$.
\end{lemma}
\begin{proof}
Fix some $c\geq Z_{coh}$ and choose $E_v\subset |c- Z_{coh}|$.
Then for any $l\in R(0,c)$ with
$l_v=c_v$ we have $\min\{l, Z_{coh}\}=\min\{l-E_v, Z_{coh}\}$.
Therefore, by \ref{rem:antmatroid}, $h^1(\calO_{l-E_v})=h^1(\calO_l)$, thus
$w_0(l-E_v)\leq w_0(l)$. Then for any $n\in \Z$, a strong deformation retract in the direction $E_v$ realizes
a homotopy equivalence between the spaces $S_n\cap R(0,c)$ and  $S_n\cap R(0,c-E_v)$.
A retract $r:S_n\cap R(0,c)\to S_n\cap R(0,c-E_v)$ can be defined as follows (for notation see \ref{9complex}).
If $\square =(l,I)$ belongs to $ S_n\cap R(0,c-E_v)$ then $r$ on $\square$ is defined as the identity.
If $(l,I)\cap  R(0,c-E_v)=\emptyset$, then $l_v=c_v$, and  we set  $r(x)=x-E_v$. Else,
$\square =(l,I)$ satisfies $v\in I$ and $l_v=c_v-1$. Then we retract $(l,I)$ to $(l, I\setminus v)$ in the $v$--direction.
The strong deformation retract is defined similarly.
\end{proof}

\begin{corollary}\label{cor:veges}
(a) The graded root $\RR(R(0,c),w)$ satisfies $|\chic^{-1}(n)|=1$ for any $n\gg 0$.

(b) $\bH^*_{red}(R(0,c),w)$ is a finitely generated $\Z$-module (for any finite or infinite $c\geq Z_{coh}$).
\end{corollary}
\begin{proof}
For any $n\gg 0$ we have $R(0,c)=S_n$, hence $S_n$ is contractible for such $n$.
\end{proof}

\subsection{The analytic lattice cohomology of $(X,o)$, independence  of $\phi$}\label{ss:anphi}

\bekezdes Let us abridge   $\bH^*(R(0,c),w)$ as $\bH^*_{an}(\phi)$ and
 $\RR(R(0,c),w)$ as $\RR_{an}(\phi)$.
\begin{theorem}\label{th:annlattinda} Assume that the resolution graph is a tree (a property independent of the
resolution). Then
$\bH^*_{an}(\phi)$  and $\RR_{an}(\phi)$ are  independent of the choice of the resolution $\phi$.
\end{theorem}
(Later in \ref{bek:ADD} we will drop the assumption regarding the graph.)
\begin{proof}
Let us fix a resolution $\phi$, and denote the blow up of a point of $E_{v_0}\setminus \cup_{w\not= v_0}E_w$
by $\pi$, and set $\phi':= \phi\circ \pi$.
Let $\Gamma$ and $\Gamma'$ be the corresponding graphs, $L(\Gamma),\ L(\Gamma')$  the lattices
and  $\, (\,,\,), \ (\,,\,)'$ the  intersection forms.

 We denote the new
$(-1)$-vertex of $\Gamma'$ by $E_{new}$.  We  use the same notation  for  $E_v\in L$ and for
its strict transform in $L'$. We have the  natural morphisms:
 $\pi_*:L(\Gamma')\to
L(\Gamma)$ defined by $\pi_*(\sum x_vE_v+x_{new}E_{new})=\sum
x_vE_v$, and $\pi^*:L(\Gamma)\to L(\Gamma') $  defined by
$\pi^*(\sum x_vE_v)=\sum x_vE_v +x_{v_0}E_{new}$. Then
$(\pi^*x,x')'=(x,\pi_*x')$.  Thus
$(\pi^*x,\pi^*y)'=(x,y)$ and $(\pi^*x,E_{new})'=0$ for any
$x,y\in L(\Gamma)$.
Associated with $\phi$,  let $\hh$  be the
 Hilbert function, $w_0$ the analytic weight and
 $S_n(\phi)=\cup\{\square\,:\, w(\square)\leq n\}$. We use similar notations  $\hh'$, $w_0'$ and
 $S_n(\phi')$ for  $\phi'$.

 Note that for any $x\in R$,
   $H^0(\tX', \calO_{\tX'}(-\pi^*x))=H^0(\tX, \calO_{\tX}(-x))$.
   Even more, for $a\leq 0$ and $x$ as above,
    $H^0(\tX', \calO_{\tX'}(-\pi^*x-aE_{new}))=H^0(\tX', \calO_{\tX'}(-\pi^*x))$.
   Indeed, take the exact sequence of sheaves
   $$0\to \calO_{\tX'}(-\pi^*x)\to \calO_{\tX'}(-\pi^*x-aE_{new})\to \calO_{-aE_{new}}(-\pi^*x-aE_{new})\to 0$$
   and use that $h^0(\calO_l(l)\otimes \calL)=0$ for any $l>0$ and line bundle $\calL$ with
    $(c_1\calL,E_v)=0$ for any $E_v\in |l|$.
   This last vanishing follows from the Grauert--Riemenschneider Theorem via Serre duality.
Therefore,
\begin{equation}\label{eq:MON1}
 \hh'(\pi^*x+aE_{new}) \ \left\{ \begin{array}{l}
 = \hh(x) \ \mbox{ for any $a\leq 0$} \\
 \mbox{is increasing for $a\geq 0$}.\end{array}\right.
\end{equation}
Using the exact sequence $0\to \calO_{aE_{new}}(-\pi^*x)\to \calO_{\pi^*x+aE_{new}}\to \calO_{\pi^*x}\to 0$ and
Lipman's vanishing $h^1(\calO_{aE_{new}}(-\pi^*x))=0$ from \ref{prop:VAN0}, we get that
$h^1(\calO_{\pi^*x+aE_{new}})=h^1(\calO_{\pi^*x})$ for any $a\geq 0$. Furthermore, from
$0\to \calO_{E_{new}}(-\pi^*x +E_{new})\to \calO_{\pi^*x}\to \calO_{\pi^*x-E_{new}}\to 0$
we get that $h^1(\calO_{\pi^*x-E_{new}})=h^1(\calO_{\pi^*x})$ too. On the other hand  (by Leray spectral sequence),
$h^1(\tX',\calO_{\pi^*x})=h^1(\tX,\calO_x)$. Therefore,
 \begin{equation}\label{eq:MON2}
 h^1(\calO_{\pi^*x+aE_{new}}) \ \left\{ \begin{array}{l}
  \mbox{is increasing  for $a\leq  -1$}, \\
= h^1(\calO_x) \ \mbox{ for any $a\geq  -1$}. \\
\end{array}\right.
\end{equation}
These combined provide
\begin{equation}\label{eq:HOM2a}
a\mapsto w_0'(\pi^*x+aE_{new}) \ \left\{ \begin{array}{l}
\mbox{is  decreasing for $a\leq  -1$},\\
 = w_0(x) \ \mbox{ for  $a= -1$ and $a=0$,} \\
 \mbox{is increasing for $a\geq 0$}.\end{array}\right.
\end{equation}
Recall that we can compute $\bH^*_{an}(\phi)$ using the cube
$R(0,c)$ with $c\geq Z_{coh}(\phi)$, cf. \ref{lem:INDEPAN}. But then $\pi^*c\geq Z_{coh}(\phi')$,
hence $\bH^*_{an}(\phi')$ can be computed in $R(0, \pi^*c)$,
and $\pi^*$ sends  the lattice points of $R(0,c)$  into $R(0, \pi^*c)$.

Furthermore, if $w_0'(\pi^*x+aE_{new})\leq n$, then
$w_0(x)\leq n$ too. In particular, the projection $\pi_{\R}$ in the direction of $E_{new}$
induces a well-defined map $\pi_{\R}:S_n(\phi')\to S_n(\phi)$.
We claim that (whenever $S_n(\phi)$ is non-empty)
$\pi_{\R}$  is a   homotopy equivalence  (with all fibers non-empty and contractible).
\bekezdes \label{bek:proof1}
We proceed in two steps.
First we prove that $\pi_{\R}:S_n(\phi')\to S_n(\phi)$ is onto.


Consider a zero dimensional cube (i.e. lattice point) $x\in S_n(\phi)$. Then $w_0(x)\leq n$. But then $w_0'(\pi^*x)
=w_0(x)\leq n$ too, hence $\pi^*(x)\in S_n(\phi')$ and $x=\pi_{\R}(\pi^*x)\in {\rm im}(\pi_{\R})$.

Next take  a cube $(x,I)\subset  S_n(\phi)$ ($I\subset \cV$). This means that
$w_0(x+E_{I'})\leq n$ for any $I'\subset I$. But then
\begin{equation}\label{eq:eps}
\pi^*(x+E_{I'})=\pi^*x+ E_{I'}+\epsilon\cdot E_{new},
\end{equation}
where $\epsilon=0$ if $v_0\not \in I'$ and $\epsilon =1$ otherwise. Hence
\begin{equation}\label{eq:eps2}
w_0'(\pi^*x+E_{I'}) =w_0'(\pi^*(x+E_{I'})-\epsilon E_{new})\stackrel {(\ref{eq:HOM2a})}{=}
w_0(x+E_{I'})\leq n.
\end{equation}
Therefore $(\pi^*x,I)\in S_n(\phi')$ and $\pi_{\R}$ projects $(\pi^*x, I)$  isomorphically  onto $(x,I)$.

Next, we show that $\pi_{\R}$ is in fact a homotopy equivalence.
In order to prove this fact it is enough to verify that if
$\square\in S_{n}(\phi)$ and $\square ^\circ$ denotes its relative interior,
then $\pi_{\R}^{-1}(\square^\circ) \cap S_{n}(\phi')$ is contractible.

Let us start again with a lattice point $x\in S_n(\phi)$. Then $\pi_{\R}^{-1}(x)\cap S_n(\phi')$ is a real  interval
(whose endpoints are lattice points, considered in the real line of the $E_{new}$ coordinate).
Let us denote it by $\cI(x)$. Now, if $\square=(x,I)$, then we have to show that
all the intervals $\cI(x+E_{I'})$ associated with all the subsets $I'\subset I$ have a common
lattice point. But this is exactly what we verified above: the $E_{new}$ coordinate of $\pi^*(x)$ is such a common
point. Therefore, $\pi_{\R}^{-1}(\square ^\circ)\cap S_n(\phi')$
has a strong deformation retraction (in the $E_{new}$ direction)
to the contractible space $(\pi^*x, I)^\circ$.

For any $l\in L$ let  $N(l)\subset \R^s$ denote the union of all cubes which have $l$ as one of their vertices.
Let $U(l)$ be its interior. Write $U_n(l):=U(l)\cap S_n(\phi)$. If $l\in S_n(\phi)$ then $U_n(l)$ is a contractible
neighbourhood of $l$ in $S_n(\phi)$. Also,  $S_n(\phi)$ is covered by $\{U_n(l)\}_l$.
Moreover, $\pi_{\R}^{-1}(U_n(l))$ has the homotopy type of $\pi_{\R}^{-1}(l)$, hence it is contractible.
More generally, for any cube $\square$,
$$\pi_{\R}^{-1}(\cap _{\mbox{$v$ vertex of $\square$}} U_n(l)) \sim \pi_{\R}^{-1}(\square ^\circ)$$
which is contractible by the above discussion. Since all the intersections of $U_n(l)$'s are of these type,
we get that the inverse image of any intersection is contractible. Hence by \v{C}ech covering
(or Leray spectral sequence) argument, $\pi_{\R}$
induces an isomorphism
$H^*(S_n(\phi'),\Z)=H^*(S_n(\phi),\Z)$. In fact, this already shows that
$\bH^*_{an}(\phi')=\bH^*_{an}(\phi)$ and $\RR_{an}(\phi')=\RR_{an}(\phi)$.
The fact that these identifications preserve the $U$--action follows from the natural inclusions of the spaces $S_n$.

In order to prove the homotopy equivalence,
one can use   quasifibration, defined in  \cite{DoldThom};
 see also \cite{DadNem}, e.g. the relevant Theorem 6.1.5.
Since
$\pi_{\R}:S_{n}(\phi')\to  S_{n}(\phi)$  is a quasifibration,
and  all the fibers are contractible, the homotopy equivalence follows.

\bekezdes \label{bek:proof2}

The case when we blow up an intersection point $E_{v_0}\cap E_{v_1}$ starts very similarly, however at some point
there is a major difference, hence  we need an additional argument.

Below we write   $E_J:=\sum_{v\in J}E_v$ for any subset $J\subset \calv$.

With very similar notations, in this case we define $\pi^*(\sum_vx_vE_v)=\sum_v x_vE_v+(x_{v_0}+
x_{v_1})E_{new}$. Then all the statements till \ref{bek:proof1} (in fact, till (\ref{eq:eps}))
remain valid (including
the key (\ref{eq:HOM2a})).
However, the first part of \ref{bek:proof1} should be modified.
The first difference is  in (\ref{eq:eps}). Indeed, in this case
\begin{equation}\label{eq:eps3}
\pi^*(x+E_{I'})=\pi^*x+ E_{I'}+\epsilon\cdot E_{new},
\end{equation}
where $\epsilon$ is the cardinality of $I'\cap \{v_0,v_1\}$. This can be 0, 1 or 2. Therefore, if
$\{v_0,v_1\}\not\subset I$, then $\epsilon \in \{0,1\} $ for any $I'$, hence for such cubes $(x,I)$
all the arguments of \ref{bek:proof1} work.

Assume in the sequel that $\{v_0, v_1\}\subset I$. Write $J=I\setminus \{v_0,v_1\}$.

 There are two cube candidates  of $L(\Gamma')\otimes \R$ which might
cover the cube $(x,I)\in S_n(\phi)$. One of them is $(\pi^*x,I)$ (as above).  However,  by  (\ref{eq:HOM2a})
the lattice points $\pi^*(x+E_I)=\pi^*x+E_I+2E_{new}$ and $\pi^*(x+E_I)-E_{new}=\pi^*x+E_I+E_{new}$ are in
$S_n(\phi')$, but the vertex $\pi^*(x)+E_I$ of $(\pi^*x,I)$ might not be in $S_n(\phi')$.

Another candidate is $(\pi^*x+E_{new},I)$, but here again
$\pi^*x$ and $\pi^*x-E_{new}$ are in $S_n(\phi')$  but the vertex $\pi^*x+E_{new}$ of
 $(\pi^*x+E_{new},I)$ might be not.
So both cubes a priori are obstructed if we apply (\ref{eq:HOM2a}) only.

Next we analyze these obstructions in more detail and we show that one of the candidate cubes works.

\bekezdes\label{bek:A} {\bf Case 1.}
 Assume that $w_0'(\pi^*x)=w_0'(\pi^*x+E_{new})$. Then by (\ref{eq:MON1}) and (\ref{eq:MON2})
 we obtain that $\hh'(\pi^*x)=\hh'(\pi^*x+E_{new})$.  By the matroid inequality of $\hh'$
 we get that $\hh'(\pi^*x+E_{J'})=\hh'(\pi^*x+E_{J'}+E_{new})$ for any $J'\subset J$.
 This again via (\ref{eq:MON1}) and (\ref{eq:MON2})
 shows that $w_0'(\pi^*x+E_{J'})=w_0'(\pi^*x+E_{J'}+E_{new})$.
 In particular,
 $$w_0'(\pi^*x+E_{J'}+E_{new})=w_0'(\pi^*x+E_{J'})=w_0'(\pi^*(x+E_{J'}))=w_0(x+E_{J'})\leq n.$$
That is, the  vertices of type $\pi^*x+E_{J'}+E_{new}$ of $(\pi^*x+E_{new},I)$
are in $S_n(\phi')$. For all other vertices we already know this fact (use (\ref{eq:HOM2a})).
Hence $(\pi^*x+E_{new},I)$ is in $S_n(\phi')$ and it projects via $\pi_{\R}$ bijectively to $(x,I)$.
Furthermore, $\pi_{\R}^{-1}(x,I)^\circ \cap S_n(\phi')$ admits a deformation retract to
 $(\pi^*x+E_{new},I)^\circ $, hence it is contractible.

\bekezdes \label{bek:B} {\bf Case 2.}
 Assume that $w_0'(\pi^*x+E_I)=w_0'(\pi^*x+E_I+E_{new})$,
 or $w_0'(\pi^*(x+E_I)-2E_{new})=w_0'(\pi^*(x+E_I)-E_{new})$.
  Then by (\ref{eq:MON1}) and (\ref{eq:MON2})
 we obtain that $h^1(\calO_{\pi^*x+E_I})=h^1(\calO_{\pi^*x+E_I+E_{new}})$.  By the opposite
 matroid inequality of $h^1$ and (\ref{eq:MON1}) and (\ref{eq:MON2}) again we obtain that
 $w_0'(\pi^*x+E_I-E_{J'})=w_0'(\pi^*x+E_I-E_{J'}+E_{new})$.
 In particular,
 $$w_0'(\pi^*x+E_I-E_{J'})=w_0'(\pi^*x+E_I-E_{J'}+E_{new})=w_0'(\pi^*(x+E_I-E_{J'})-E_{new})=w_0(x+E_I-E_{J'})\leq n.$$
That is, the  vertices of type $\pi^*x+E_I-E_{J'}$ of $(\pi^*x,I)$
are in $S_n(\phi')$. For all other vertices we already know this fact (use (\ref{eq:HOM2a})).
Hence $(\pi^*x,I)$ is in $S_n(\phi')$ and it projects via $\pi_{\R}$ bijectively to $(x,I)$.
Furthermore, $\pi_{\R}^{-1}(x,I)^\circ \cap S_n(\phi')$ admits a deformation retract to
 $(\pi^*x,I)^\circ $, hence it is contractible.

\bekezdes \label{bek:C} {\bf Case 3.} Assume that the assumptions from {\bf Case 1} and {\bf Case 2} do not hold.
This means that
$$\left\{ \begin{array}{l}
\hh'(\pi^*x)<\hh'(\pi^*x+E_{new}), \ \mbox{and} \\
h^1(\calO_{\pi^*x+E_I})< h^1(\calO_{\pi^*x +E_I+E_{new}}).\end{array}\right.$$
This reads as follows (cf. (\ref{eq:duality})
$$\left\{ \begin{array}{l}
(a) \ \ H^0(\calO_{\tX'}(-\pi^*x-E_{new}) \subsetneq  H^0(\calO_{\tX'}(-\pi^*x)), \ \mbox{and} \\
(b) \ \ H^0(\tX', \Omega^2_{\tX'}(\pi^* x +E_I)) \subsetneq
H^0(\tX', \Omega^2_{\tX'}(\pi^* x +E_I+E_{new})).
\end{array}\right.$$
Part {\it (a)}  means the following: there exists a function $f\in H^0(\tX', \calO_{\tX'})$
such that ${\rm div}_{E'}(f)\geq \pi^*x$, and in this inequality the $E_{new}$--coordinate entries are equal.
By part {\it (b)}, there exists a global  2--form $\omega $ such that
${\rm div}_{E'}(\omega)\geq -\pi^*x-E_I-E_{new}$ and the $E_{new}$--coordinate entries are equal.

Therefore, the form $f\omega\in H^0(\tX'\setminus E', \Omega^2_{\tX'})$ has the property that
${\rm div}_{E'}(f\omega)\geq -E_I-E_{new}$ with equality at the $E_{new}$ coordinate.
In particular, again by duality (\ref{eq:duality}), we obtain that  in $\tX'$ the following
strict inequality  holds:
\begin{equation}\label{eq:ROSSZ}
h^1(\calO_{E_I+E_{new}})>h^1(\calO_{E_I}) \ \ (\cV'=\cV\cup\{new\}, \ I\subset \cV).\end{equation}
But if the graph is a tree then this strict inequality cannot happen.

\bekezdes In particular, for any $I\subset \cV$ either $\{v_0,v_1\}\not\subset  I$, or in the opposite case
either {\bf Case 1} or {\bf Case 2} applies.  Hence,
in any case,
$\pi_{\R}^{-1}(x,I)^\circ \cap S_n(\phi')$ is contractible. Therefore, $S_n(\phi)$ and $S_n(\phi')$
have the same homotopy type by the argument from the end of \ref{bek:proof1}.
\end{proof}

\bekezdes\label{bek:ADD}  {\bf Addendum to Theorem \ref{th:annlattinda}.}
In the proof of Theorem \ref{th:annlattinda}  the assumption (namely that $\Gamma$ should be a tree)
was used only to show that (\ref{eq:ROSSZ}) cannot hold.
However, the failure of (\ref{eq:ROSSZ}) can be guaranteed in other way as well.
Let $\Gamma$ be a resolution graph, which is not a tree. If $\lambda$ is a (simple) closed 1--cycle
in  (the topological realization of) $\Gamma$ then let $\ell(\lambda)$ be its combinatorial length
(number of edges in it).

Furthermore, in general, if $V\subset \calv$ ($V\not=\emptyset$)
 then $h^1(\calO_{E_V})=\sum_{v\in V}g_v+b_1(|E_V|)$, where $b_1(|E_V|)$ denotes the number of independent
 closed 1--cycles  of the support $|E_V|$.
 Hence,  (\ref{eq:ROSSZ}) can happen in $\Gamma'$  if and only if $I$ considered in $\calv$
 contains a closed 1--cycle $\lambda$  in $\Gamma$, $\lambda$ has vertices $\{v_0,v_1\}\cup J$,
 its lift $\lambda'$ to $\Gamma'$
 has vertices  $\{v_0,v_1\}\cup J\cup\{new\}$, hence this cycle $\lambda'$
 is broken if we delete the new vertex $\{new\}$ of $\Gamma'$.
 Thus, $|I|\geq \ell(\lambda)$.

 Let $\ell_{min}(\phi)$ be the smallest $\ell(\lambda)$, where $\lambda$ is a closed 1--cycle of $\Gamma$.
 Since $\phi$ is a good resolution, hence no  $E_v$  has self-intersection, $\ell_{min}(\phi)\geq 2$.
 The above discussion shows that for any $|I|<\ell_{min}(\phi)$ (\ref{eq:ROSSZ}) cannot hold.
 Hence, any $q$--cube with $q<\ell_{min}(\phi)$ can be lifted (as in the proof), it is in the image of
 $\pi_{\R}$, and its inverse image is contractible. Note also that the $ \pi_{\R}$--image of
 any cube from $S_n(\phi')$ is in $S_n(\phi)$. Hence $S_n(\phi)$ is obtained from ${\rm im}(\pi_{\R})$
  by adding cubes of dimension $\geq \ell_{min}(\phi)$ and
  ${\rm im}(\pi_{\R})$ has the homotopy type of
 $S_n(\phi')$. In particular, $\bH^{\leq (\ell_{min}(\phi)-2)}_{an}(\phi)=\bH^{\leq (\ell_{min}(\phi)-2)}_{an}(\phi')$.

 Since $\ell_{min}(\phi)\geq 2$ we obtain that in any case
 $\bH^0_{an}(\phi)=\bH^0_{an}(\phi')$ and $\RR_{an}(\phi)=\RR_{an}(\phi)$; hence these objects are well
 defined without any restriction regarding the link.

 Now we concentrate on $\bH^q_{an}$ with fixed $q$. Let us consider two resolutions $\phi_1$ and $\phi_2$ such that
 $\ell_{min}(\phi_1), \ell_{min}(\phi_2)\geq q+2$. Let $\phi_{12}$ be the resoltuion which dominates both $\phi_1$ and $\phi_2$. Then along the sequence of resolutions which connects $\phi_1$ and $\phi_2$
via $\phi_{12}$ the module  $\bH^q_{an}$ is stable by the above discussion.

 In particular, $\bH^q_{an}(\phi)$ is well--defined, whenever it is computed in a resolution
 $\phi$ with $\ell_{min}(\phi)\geq q+2$.

\begin{definition} Assume that $\Gamma$ is a tree. Then $\bH^*_{an}(\phi)$ and $\RR_{an}(\phi)$
are independent of $\phi$.
In the sequel we will use  the notations  $\bH^*_{an,0}(X,o)$  and $\RR_{an,0}(X,o)$ for them.
(The index `zero' means `canonical spin$^c$--structure.)
They are called the {\it  analytic lattice cohomology of $(X,o)$}  and the {\it
analytic graded root of $(X,o)$} (associated with the canonical spin$^c$--structure).
\end{definition}
They are  analytic invariant of the germ $(X,o)$. By Corollary \ref{cor:veges}
$\bH^*_{an,red,0}(X,o)$ has finite $\Z$--rank, hence the Euler characteristic $eu(\bH^*_{an,0}(X,o))$ is well--defined.

\subsection{The `Combinatorial Duality Property' of the pair $(\hh, \hh^\circ)$}\label{ss:anCDP}

\bekezdes
The following reinterpretation of $\hh^\circ:R(0,c)\to \Z$  will be helpful.

From \ref{bek:LauferDual} we have that
$\dim\, H^0(\calO_{\tX}(K_{\tX}+Z))/
H^0( \calO_{\tX}(K_{\tX}))=h^1(\calO_{Z})$
 for any $Z>0 $.
 On the other hand, by the opposite matroid rank inequality  \ref{rem:antmatroid}
 we also have $h^1(\cO_Z)=h^1(\cO_{\min\{Z,\lfloor Z_K\rfloor_+\}})$. 
Hence,  $\hh^\circ(l)=p_g-h^1(\calO_l)=h^1(\calO_c)-h^1(\calO_l)$ appears as
\begin{equation}\label{eq:reinter}
\hh^\circ(l):= \dim\, \frac{H^0(\tX, \calO_{\tX}(K_{\tX}+c))}{
H^0(\tX, \calO_{\tX}(K_{\tX}+l))}=
\dim\, \frac{H^0(\tX, \calO_{\tX}(K_{\tX}+\lfloor Z_K\rfloor_+))}{
H^0(\tX, \calO_{\tX}(K_{\tX}+\min\{l, \lfloor Z_K\rfloor_+\} ))}
.\end{equation}
(In (\ref{eq:reinter}) $\lfloor Z_{K}\rfloor _+$ can be replaced by $Z_{coh}$ as well.)


\begin{example} \label{ex:GorAN}
{\bf Gorenstein germs.}
Assume that $(X,o)$ is Gorenstein. Then $Z_K\in L$, and let us assume (for simplicity) that
$Z_K\geq 0$ (this happens e.g. in the good minimal resolution, cf. \cite{Book}).
Then for any $l\in R(0, Z_K)$, by (\ref{eq:reinter})  we have $\hh^\circ (l)=
\dim\, H^0(\calO_{\tX})/ H^0(\calO_{\tX}(-Z_K+l))=\hh(Z_K-l)$.
Therefore, $w_0(l)=\hh(l)+\hh(Z_K-l)-p_g$
is obtained as the {\it symmetrization} of $\hh$.
In particular $w_0(l)=w_0(Z_K-l)$. Note that this is true for the topological weight function too:
$\chi(l)=\chi(Z_K-l)$. However, in the analytic case, the symmetry might fail for non-Gorenstein germs
(even if we consider  a numerically Gorenstein topological type).
\end{example}

The following property will be crucial in the Euler characteristic computation.

\begin{lemma}\label{lem:hsimult} {\bf (CDP)}  Assume that $g_v=0$ for any $v\in\calv$. Then
there exists no  $l\in L_{\geq 0}$ and  $v\in\calv$ such that  the differences
$\hh(l+E_v)-\hh(l)$ and $\hh^\circ (l)-\hh^\circ (l+E_v)$ are simultaneously strictly positive.
\end{lemma}
\begin{proof}
%

If $\hh(l+E_v)>\hh(l)$ then there exists a global function $f\in H^0(\calO_{\tX})$ with ${\rm div}_Ef\geq l$, where the
$E_v$-coordinate is  $({\rm div}_Ef)_v = l_v$.
Similarly, if $\hh^\circ (l)>\hh^\circ(l+E_v)$ then there exists a
global 2-form $\omega$ with possible poles along $E$,
 with ${\rm div}_E\omega \geq -l-E_v$ (i.e., the pole order is $\leq l+E_v$), and
 $({\rm div}_E\omega )_v = -l_v-1$.
In particular, the form $f\omega$
satisfies
 ${\rm div}_E\ f\omega \geq -E_v$  and
  $({\rm div}_E f \omega )_v = -1$. This implies $H^0(\Omega_{\tX}^2(E_v))/H^0(\Omega_{\tX}^2)\not=0$, or,
  by (\ref{bek:LauferDual}), $h^1(\calO_{E_v})\not=0$. 
\end{proof}

\bekezdes
In the next discussions  we will assume that the link of $(X,o)$ is a rational homology sphere
(hence the graph in particular is a tree).
We show that  $\bH^*_{an,0}(X,o)$ usually is non-trivial, its Euler characteristic is $p_g$,
it can be `compared' with the topological
lattice cohomology $\bH^*_{top,0}(M)$ of the link,  and it really sees the variation of different
analytic structures supported on a fixed topological type.
It is worth to compare
several statements below with  their topological analogues valid for
$\bH^*_{top,0}(M)$. (Recall that $\bH^*_{top,0}(M)$ was defined whenever the link is a $\Q HS^3$.)

\subsection{The Euler characteristic $eu(\bH^*_{an}(X,o))$}\label{ss:anEu}

\bekezdes Lemma \ref{lem:hsimult} will allow us to determine the  Euler characteristic $eu(\bH^*_{an}(X,o))$
of the analytic lattice cohomology by a combinatorial argument. Surprisingly, this Euler characteristic automatically
equals the Euler characteristic of  path cohomologies  associated with any increasing path
(this equality definitely does not hold
in  the topological versions of the corresponding lattice cohomologies).

First, let us fix the notations.
In the sequel
we will also consider
for any increasing path $\gamma$ connecting 0 and $c$
 (that is, $\gamma=\{x_i\}_{i=0}^t$, $x_{i+1}=x_i+E_{v(i)}$, $x_0=0$ and $x_t=c$, $c\geq Z_{coh}$)
 the  path lattice cohomology $\bH^0(\gamma,w)$ as in \ref{bek:pathlatticecoh}.
Accordingly,  we have the numerical
Euler characteristic
 $eu(\bH^0(\gamma,w))$ as well.

The proof of the next theorem basically is combinatorial,
it is provided in the next subsection, where we separated certain combinatorial aspects of the
lattice cohomology, cf. Theorem \ref{th:comblattice}.

\begin{theorem}\label{th:euANLAT} Assume that the link is a $\Q HS^3$. Then
$eu(\bH^*_{an,0}(X,o))=p_g(X,o)$. Furthermore, for any increasing path $\gamma$ connecting 0 and $c$ (where
$c\geq Z_{coh}$)
 we also have $eu(\bH^*(\gamma,w))=p_g$.
\end{theorem}
\begin{proof}
We claim that the assumptions of Theorem \ref{th:comblattice} are satisfied. Indeed,
the CDP was verified in \ref{lem:hsimult}, while the stability property of $\hh$  follows
since it is associated with a filtration.
\end{proof}
This  in particular means that $\bH^*_{an,0}(X,o)$ is a {\it categorification  of the geometric genus},
that is, it is a graded cohomology module whose Euler characteristic is $p_g$.

For more comments regarding Theorem \ref{th:euANLAT} see Remark \ref{rem:UJ}.

\subsection{Analytic reduction theorem}\label{ss:anRT}

\bekezdes
Our next goal is to prove a `Reduction Theorem', the analogue of the topological Theorem
\ref{th:red}.
Via such a result, the rectangle $R=R(0,c)$ can be replaced  by another rectangle sitting in a
lattice of smaller rank. The procedure starts with identification of a set of `bad' vertices,
see \ref{ss:BadVer}. More precisely,   we decompose $\calv$ as a disjoint union
$\overline{\calv}\sqcup \calv^*$, where the vertices $\overline {\calv}$ are the `essential' ones, the ones which dominate the others, and the coordinates $\calv^*$ are those which `can be eliminated'.
In the topological context the possible choice of $\overline{\calv}$ was dictated by combinatorial
properties of $\chi$ with a special focus on the topological characterization of rational germs.
In the present context we start with certain  analytic properties of 2-forms (which reflects the dominance
of $\overline{\calv}$ over $\calv^*$).
(Note that $p_g=0$ if and only  if $H^0(\tX\setminus E, \Omega^2_{\tX})=H^0(\tX, \Omega^2_{\tX})$.)

In this section we assume that the link is a rational homology sphere.
\begin{definition}\label{def:DOMAN}
We say that  $\overline{\calv}$ is an B$_{an}$--set
if it satisfy the following
property: if 
some differential form
$\omega\in H^0(\tX\setminus E, \Omega_{\tX}^2) $ satisfies
$({\rm div}_E\omega) |_{\overline{\calv}}\geq -E_{\overline{\calv}}$ \
then necessarily
$\omega\in H^0(\tX, \Omega_{\tX}^2)$.
By (\ref{bek:LauferDual}) this is equivalent with the vanishing  $h^1(\calO_Z)=0$ for any
$Z=E_{\overline{\calv}}+l^*$, where $l^*\geq 0$ and it is supported on $\calv^*$.
\end{definition}


\begin{lemma}\label{lem:AnNodes}
If $\overline{\calv}$ is a B--set, then it is a B$_{an}$--set too.
(For the definition of B--sets see \ref{def:SWrat}.)
 \end{lemma}
\begin{proof}
Let $\Gamma$ be the original graph, and let $\Gamma'$ be that rational graph which is obtained from $\Gamma$ be decreasing
the numbers $\{E_v^2\}_{v\in\overline{\calv}}$. Let $(\,,\,)'$ be its intersection form and $\chi'$ the RR expression.
Fix a cycle of type $Z=E_{\overline{\calv}}+l^*$ with $|l^*|\subset \calv^*$, $l^*\geq 0$. Then, regarding $\Gamma'$, we
claim the following: there exists a sequence $\{x_i\}_{i=0}^t$ such that
\begin{equation}\label{eq:SEQ}
x_0=0, \ x_t=Z, \ x_{i+1}=x_i+E_{v(i)}, \ (x_i, E_{v(i)})'\leq 1 \ \mbox{at every step $i$}.
\end{equation}
We construct the elements $x_i$ by decreasing order. Assume that $x_i$ is already constructed. Then, there exists at least one
$E_u\subset |x_i|$ such that $\chi'(x_i-E_u)\leq \chi'(x_i)$, equivalently $(\dag)$ $(x_i-E_u, E_u)'\leq 1$.
[Indeed, if not, then $(x_i, E_u)'\geq (Z'_K, E_u)'$ for every $E_u\subset |x_i|$, hence by summation
$\chi'(x_i)\leq 0$, a fact which contradicts the rationality of $\Gamma'$.] Then set $x_{i-1}:=x_i-E_u$. Note that
$(x_{i-1},E_u)'=(x_i-E_u,E_u)'\leq 1$ by ($\dag$). Hence the existence of $\{x_i\}_i$ follows.

Now, since $E_{\overline{\calv}}$ is reduced, and the self-intersections from the support of $l^*$ are not modified,
along these sequence $(x_i, E_{v(i)})'=(x_i, E_{v(i)})$. In particular, this sequence has the very same properties (\ref{eq:SEQ})
in $\Gamma$ too.

Then, using the exact sequences $0\to \calO_{E_{v(i)}}(-x_i)\to \calO_{x_{i+1}}\to \calO_{x_i}\to 0$
we get that $h^1(\calO_{x_{i+1}})=h^1(\calO_{x_i})$. Since $h^1(\calO_{x_0})=0$ by induction $h^1(\calO_{x_t})=0$ too.
\end{proof}
 \begin{example}\label{ex:Ran} By the above lemma,
 the set $\overline{\calv}={\mathcal N}$ of nodes is an B$_{an}$--set.
 Moreover, if $\{\overline{v}\}$ is an B--set of an AR graph, then it is an B$_{an}$--set as well.
 \end{example}

\bekezdes Associated with a disjoint  decomposition  $\calv=\overline{\calv}\sqcup \calv^*$,  we write
any  $l\in L$
as $\overline{l}+l^*$, or $(\overline{l}, l^*)$,
where $\overline {l}$ and  $l^*$ are  supported on $\overline{\calv}$ and
 $\calv^*$ respectively. We also write  $\overline{R}$ for the rectangle $R(0, \overline{c})$, the $\overline{\calv}$-projection of $R(0,c)$ with $c=Z_{coh}$.

For any $\overline {l}\in \overline {R}$ define  the weight function
$$\overline{w}_0(\overline{l})=\hh(\overline{l})+\hh^\circ (\overline{l}+c^*)-p_g
=\hh(\overline{l})-h^1(\calO_{\overline{l}+c^*}).$$
Consider all the cubes of $\overline{R}$ and the weight function
$\overline{w}_q:\calQ_q(\overline{R})\to \Z$ by $ \overline{w}_q(\square_q)=\max\{w_0(\overline{l})\,:\, \overline{l} \
 \mbox{\,is any vertex of $\square_q$}\}$.

\begin{theorem}\label{th:REDAN} {\bf Reduction theorem for the analytic lattice cohomology.}
If  $\overline{\calv}$ is an B$_{an}$-set
then
$$\bH^*_{an}(R,w)=\bH^*_{an}(\overline{R}, \overline{w}).$$
\end{theorem}
\begin{proof}
For any $\cali\subset \calv$ write $c_\cali$ for the $\cali$-projection of $c=Z_{coh}$.
We proceed by induction, the proof will be given in $|\calv^*|$ steps.
For any $\overline{\calv}\subset \cali\subset \calv$ we create the inductive setup.
We write $\cali^*=\calv\setminus \cali$, and according to the disjoint union
$\cali\sqcup \cali^*=\calv$ we  consider the coordinate decomposition
$l=(l_\cali,l_{\cali^*})$. We also set $ R_\cali=R(0, c_\cali)$ and the weight function
$$w_\cali(l_\cali)=\hh(l_\cali)+\hh^\circ(l_\cali+c_{\cali^*})-p_g. $$
Then  for $\overline{\calv}\subset \cali\subset \cJ\subset \calv$, $\cJ=\cali\cup \{v_0\}$
($v_0\not\in\cali$),
we wish to prove that $\bH^*_{an}(R_\cali, w_\cali)=\bH^*_{an}(R_{\cJ}, w_{\cJ})$.
For this consider the projection $\pi_{\R}:R_{\cJ}\to R_{\cali}$.

For any fixed $y\in R_\cali$ consider the fiber $\{y+tE_{v_0}\}_{0\leq t\leq c_{v_0},\ t\in \Z}$.

Note that $t\mapsto \hh(y+tE_{v_0})$  is increasing. Let $t_0=t_0(y)$
be the smallest value $t$ for which
$\hh(y+tE_{v_0})< \hh(y+(t+1)E_{v_0})$.
If $t\mapsto \hh(y+tE_{v_0})$ is constant then we take $t_0=c_{v_0}$.
If $t_0<c_{v_0}$, then $t_0$ is characterized by the existence of a function
\begin{equation}\label{eq:1RED}
f\in H^0(\calO_{\tX}) \ \ \mbox{with} \ \
({\rm div}_Ef)|_\cali\geq y, \ \ \  ({\rm div}_Ef)_{v_0}=t_0.
\end{equation}
Symmetrically,  $t\mapsto \hh^{\circ} (y+c_{\cJ^*}+ tE_{v_0})$  is decreasing. Let $t_0^\circ=t_0^\circ (y)$
be the smallest value $t$ for which
$\hh^{\circ} (y+c_{\cJ^*}+tE_{v_0})=\hh^{\circ} (y+c_{\cJ^*}+(t+1)E_{v_0})$. The value
$t_0^\circ$  is characterized by the existence of a form
\begin{equation}\label{eq:2RED}
\omega \in H^0(\tX\setminus E,\Omega_{\tX}^2) \ \ \mbox{with} \ \
({\rm div}_E\omega) |_\cali\geq - y, \ \ \  ({\rm div}_E\omega )_{v_0}=-t^{\circ}_0.
\end{equation}
This shows that there exist a form $f\omega\in H^0(\tX\setminus E, \Omega_{\tX}^2)$ such that
$({\rm div}_Ef\omega )|_\cali\geq 0$ and $({\rm div}_Ef\omega )_{v_0}=t_0-t^{\circ}_0$.
By the B$_{an}$ property we necessarily must have $t_0-t^{\circ}_0\geq 0$.
Therefore, the weight $t\mapsto w_{\cJ}(y+tE_{v_0})=\hh(y+tE_{v_0})+\hh^{\circ } (y+tE_{v_0}+c_{\cJ^*})-p_g$
is decreasing for $t\leq t_0^\circ$, is increasing for $t\geq t_0$. Moreover, for $t_0^\circ \leq t\leq t_0$
it  takes the
constant value $\hh(y)+\hh^{\circ } (y+c_{v_0}E_{v_0}+c_{\cJ^*})-p_g=w_{\cali}(y)$.

Next we fix $y\in R_\cali$ and some $I\subset \cali$ (hence a cube $(y,I)$ in $R_{\cali}$).
We wish to compare the intervals
$[t_0^\circ (y+E_{I'}), t_0 (y+E_{I'})]$ for all subsets $I'\subset I$. We claim that they have at least one
common element (in fact, it turns out that $t_0(y)$ works).

Note that $\hh(y+tE_{v_0})=\hh(y+(t+1)E_{v_0})$ implies $\hh(y+tE_{v_0}+E_{I'})=\hh(y+(t+1)E_{v_0}+E_{I'})$
for any $I'$,
hence $t_0(y)\leq t_0(y+E_{I'})$. In particular, we need to prove that $t_0(y)\geq t_0^\circ (y+E_{I'})$.
Similarly as above, the value $t_0^\circ(y+E_{I'})$  is characterized by the existence of a form
\begin{equation*}
\omega_{I'} \in H^0(\tX\setminus E, \Omega_{\tX}^2) \ \ \mbox{with} \ \
({\rm div}_E\omega_{I'}) |_\cali\geq - y-E_{I'}, \ \ \  ({\rm div}_E\omega_{I'} )_{v_0}=-t^{\circ}_0(y+E_{I'}).
\end{equation*}
Hence the  from $f\omega_{I'}\in H^0(\tX\setminus E, \Omega_{\tX}^2)$ satisfies
${\rm div}_Ef\omega_{I'} |_\cali\geq -E_{I'}$ and $({\rm div}_Ef\omega )_{v_0}=t_0(y)-t^{\circ}_0(y+E_{I'})$.
By the B$_{an}$ property we must have $t_0(y)-t^{\circ}_0(y+E_{I'})\geq 0$.

Set $S_{\cJ,n}$ and $S_{\cali,n}$ for the lattice spaces defined by $w_\cJ$ and $w_\cali$. If
$y+tE_{v_0}\in S_{\cJ,n}$ then $w_\cJ(y+tE_{v_0})\leq n$, hence  by the above discussion $w_\cali(y)\leq n$ too.
In particular, the projection $\pi_{\R}:R_\cJ\to R_\cali$ induces a map $S_{\cJ,n}\to S_{\cali,n}$.
We claim  that it is a homotopy equivalence. The argument is similar to the proof from
\ref{th:annlattinda} via the above preparations.
\end{proof}

\begin{corollary}\label{cor:AR} If $(X,o)$ admits a resolution $\phi$ with a B$_{an}$--set of cardinality $\overline{s}$,
then  $\bH^{\geq \overline{s}}_{an,0}(X,o)=0$. In particular, if \ $\Gamma$ is almost rational (AR) then $\bH^{\geq 1}_{an,0}(X,o)=0$.
\end{corollary}

\begin{proposition}\label{prop:AnLatAR}
Assume that for  $(X,o)$ and for  its  resolution $\phi$ with  graph $\Gamma$ the following facts hold:

(i) $\Gamma$ is almost rational,

(ii) $p_g=\min _{\gamma} eu (\bH^0_{top}(\gamma,\Gamma, -Z_K))$.

\noindent
Then $\bH^*_{an,0} (X,o)=\bH^*_{top}(M, -Z_K)$. (Obviously, we have  $\bH^{\geq 1}_{an,0} (X,o)=\bH^{\geq 1}_{top}(M, -Z_K)=0$, cf. \ref{cor:AR}.)

\end{proposition}
\begin{proof}
We can assume that $\phi$ is minimal good.
Then  $Z_K\geq 0$ (see e.g. \cite[Example 6.3.4]{Book} or \cite{PPP,Veys}).
Write
$\calv$ as $\{v_0\}\sqcup \calv^*$, where $\{v_0\}$ is an B-set.  By \ref{ex:Ran}
it is  an B$_{an}$ set as well. In particular, $\bH^*_{an,0}(X,o)=\bH^*_{an}(\overline{R},\overline {w})$,
where $\overline{w}(\ell)=\hh(\ell E_{v_0})-h^1(\calO_{\ell E_{v_0}+\lfloor Z_K\rfloor^*})$, $0\leq \ell
\leq \lfloor Z_K\rfloor _{v_0}$.

Let $x(\ell)=x(\ell E_{v_0})$ be the universal cycle introduced in \ref{lemF1}. It is the smallest cycle whose
$E_{v_0}$-multiplicity is $\ell$  and
$(x(\ell),  E_v)\leq 0$ for any $v\not=v_0$. If $s=s(\ell E_{v_0})$ is the smallest cycle in $\calS$ with $\ell E_{v_0}\leq s$
(for its existence see \cite{NOSz}), then by the universal properties $\ell E_{v_0}\leq x(\ell)\leq s$. Hence,
since $H^0(\cO_{\tX}(-\ell E_{v_0}))=H^0(\cO_{\tX}(-s))$, $\hh(\ell E_{v_0})=\hh(x(\ell))$.

Next, for any $l^*\in L_{>0}$ with $|l^*|\in \calv^*$ consider the exact sequence
$0\to \calO_{l^*}(-x(\ell))\to \calO_{x(\ell)+l^*}\to \calO_{x(\ell)}\to 0$. By the Lipman's vanishing
\ref{prop:VAN0} we have $h^1( \calO_{l^*}(-x(\ell)))=0$, hence
\begin{equation}\label{eq:daguj} h^1(\calO_{x(\ell)+l^*})= h^1(\calO_{x(\ell)}).\end{equation}
 On the other hand, using $Z_{coh}\leq \lfloor Z_K\rfloor$ (cf. \ref{rem:antmatroid}), the opposite matroid rank inequality
from  \ref{rem:antmatroid}
 and (\ref{eq:daguj}) applied for $l^*$ sufficiently large, we obtain
  $h^1(\calO_{x(\ell)+l^*})= h^1(\calO_{\min\{x(\ell)+l^*, \lfloor Z_K\rfloor\}})=
  h^1(\calO_{\ell E_{v_0}+\lfloor Z_K\rfloor ^*})$.
 Thus  $\overline{w}(\ell)=\hh(x(\ell))-h^1(\calO_{x(\ell)})=w(x(\ell))$.

 Finally, consider the exact sequence $0\to \calO_{\tX}(-x(\ell))\to \calO_{\tX}\to \calO_{x(\ell)}\to 0$ and
 the morphism $r(\ell):H^0(\calO_{\tX})\to H^0(\calO_{x(\ell)})$. Then $\hh(x(\ell))+\dim\, {\rm coker} \,( r(\ell))=h^0(\calO_{x(\ell)})$, hence
 $\overline{w}(\ell)=\chi(x(\ell))-\dim\,{\rm coker}\, (r(\ell))$.

The point in this identity is that in fact $r(\ell)$ is onto for all $\ell$. Indeed, this follows from
assumption {\it (ii)}
and from the fact that
$\min _{\gamma} eu (\bH^0_{top}(\gamma,\Gamma, -Z_K))$ is realized by the concatenated computation sequence $\gamma=\{x_i\}_{i=0}^t$
having as intermediate terms the universal cycles $x(\ell)$ (defined in  \ref{ss:CCSAR}). In this case
$H^0(\calO_{x_{i+1}})\to H^0(\calO_{x_i})$ is necessarily onto, by the splitting of the cohomological exact sequence associated with consecutive members of the optimal path, cf. \ref{rem:SPLIT}.
Since $h^1(\cO_{x_t})=p_g=h^1(\cO_{\tX})$, we also conclude that
$H^0(\calO_{\tX})\to H^0(\calO_{x_i})$  is also surjective for all $x_i$, in particular for the intermediate terms
$x(\ell)$ as well.

In conclusion,  $\overline{w}(\ell)=\chi(x(\ell))$. Then use the topological reduction theorem \ref{th:red}
for AR graphs.
\end{proof}
\begin{example}\label{ex:ARAN}
The two assumptions of Proposition \ref{prop:AnLatAR} are satisfied in the following cases,
cf. \ref{ex:MINforwh} (recall that
we  assume that the link is a rational homology sphere):

$\bullet$ \ rational singularities,

$\bullet$ \ Gorenstein elliptic singularities, and, more generally,
non-necessarily numerically Gorenstein  elliptic germs  with $p_g=\mbox{length of the elliptic sequence}$
(combine  \cite{weakly} and \cite{NNIIIb}),

$\bullet$ \ AR singularities for which the SWIC for the canonical spin$^c$ structure holds,
(in particular, weighted homogeneous germs and
superisolated germs associated with rational unicuspidal curves, or
splice quotient singularities associated with an AR graph $\Gamma$ \cite{NO2}).

\vspace{2mm}

Hence, for all these cases,  $\bH^*_{an,0} (X,o)=\bH^*_{top}(M,-Z_K)$.
For concrete  expressions of $\bH^*_{top}(M,-Z_K)$ in the above  cases see \cite{NOSz,NGr,Nlattice,Book}.
\end{example}

\subsection{Comparison of $\bH^*_{an,0}(X,o)$ with $\bH^*_{top}(M,-Z_K)$ for any $(X,o)$ with $\Q HS^3$ link}\label{an:Comp}

\bekezdes The modules  $\bH^*_{an,0}(X,o)$ with $\bH^*_{top}(M,-Z_K)$
can be compared even without the assumptions of Proposition \ref{prop:AnLatAR}.
Next,  we define a morphism  $\bH^*_{an,0}(X,o)\to \bH^*_{top}(M,-Z_K)$
of graded $\Z[U]$--modules, and we list some properties of the
analytic  graded root associated with the analytic weight function $w_0$.

First,  we compare the analytic and topological
weight functions.
 From the exact sequence $0\to \calO_{\tX}(-l)\to \calO_{\tX}\to \calO_l\to 0$
we obtain:

\begin{lemma}\label{lem:weightTOPAN}
If $w_{an}$ denotes the analytic weight function, then $w_0(l)=\chi(l)-\dim \, {\rm coker}\, (r(l))$,
where $r(l):H^0(\calO_{\tX})\to H^0(\calO_l)$ is the natural morphism.
\end{lemma}


\bekezdes\label{bek:CONSEQW}  {\bf Discussion.} Lemma \ref{lem:weightTOPAN}  has several consequences.
Let us also write $w_{top}=\chi$.

(1) \ $w_{an}\leq \chi$, hence $\min \, w_{an}\leq \min\, \chi=\min\, w_{top}$.

(2) \ Set $S_{an,n}=\cup\{\square \,:\, w_{an}(\square)\leq n\}$ and
$S_{top,n}=\cup\{\square \,:\, w_{top}(\square)\leq n\}$.  Then
$S_{top,n}\subset S_{an,n}$ for any $n\in\Z$.

(3) \ For any $q\geq 0$ there exists a {\it graded $\Z[U]$-module morphism}
$$\mathfrak{H}^q:\bH^q  _{an,0}(X,o)\to
\bH^q_{top}(M,-Z_K).$$
This is an isomorphism in the cases when Proposition \ref{prop:AnLatAR} holds.

Similarly, the inclusion of the connected components induce a graded morphism of graph
$$\RR_{top,0}(M)\to \RR_{an,0}( X,o).$$
(4) The following (new) characterization of rational germs hold: $(X,o)$ is rational if and only if $w_{an}|_{L_{>0}}>0$.
Indeed, if $(X,o)$ is rational then $w_{an}(l)=\hh(l)>0$ for $l>0$.
Conversely use $\chi\geq w_{an}$ and Artin's criterion.

(5) If $\min\, w_{an}=0$, and $(X,o)$ is not rational, then $(X,o)$ is elliptic (use $\chi\geq w_{an}$). Conversely, if $(X,o)$ is
Gorenstein elliptic then $\min\, w_{an}=0$ (use \ref{ex:ARAN}).
Moreover, $\min\, w_{an}=0$ holds for any elliptic singularity  with generic analytic structure as well (see
\ref{ex:GENERICANTYPE}). However,  in general $\min\, w_{an}=0$ does not hold for any elliptic germ, see
\ref{ex:ELLipticAN}. For a complete discussion of the elliptic case see \cite{AgNeEll}.

(6) \ In the next discussions it is convenient to assume that $c$ is conveniently large or $c=\infty$.
Let $C$ be a component of $S_{an, n}$ such that for any
$l\in S_{an, n}$ we have $w_{an}(l)=n$ (that is, the class of $C$ represents a local  minimum of the analytic weight in the
analytic graded root).
Assume that for some $s\in C$ we have $w_{an}(s+E_v)>w_{an}(s)$ for every $v\in\calv$ (that is, $s$
is  a maximal element of $C$). Then $s\in \calS_{an}$.
(Indeed,  $w_{an}(s+E_v)>w_{an}(s)$ implies $\hh(s+E_v)>\hh(s)$, cf. \ref{lem:hsimult}.)
In fact, using the matroid rank inequality of $w_{an}$ one can show that
 $C$  has a unique maximal element.
This shows that the `ends' of the analytic graded root $\RR_{an} $ represent elements of the analytic semigroup $\calS_{an}$.

(7) 
Assume that $m\in L_{>0}$ satisfies
$w_{an}(m-E_v)>w_{an}(m)$ for any $E_v\subset |m|$.
(E.g., it is a minimal element of a component $C$ as in (6).)
By \ref{lem:hsimult}
$h^1(\calO_{m-E_v})< h^1(\calO_{m})$. This can happen only if $h^1(\calO_{E_v}(-m+E_v))=h^0(\calO_{E_v}(m-Z_K))\not=0$,
which implies $(Z_K-m,E_v)\leq 0$ for every $E_v\subset |m|$. In particular, $\chi(m)\leq 0$.

 (8) Assume that $n\geq 0$. We claim that  any connected component of $S_{an, n}$ contains at least one
 element of $S_{top,n}$.
Indeed, assume the opposite, and
 let $C_n$ be a component of $S_{an, n}$ such that  $\chi(l)>n$ for any $l\in C_n$.
 Take a local $w_{an}$-minimum in $C_n$ with value $k\leq n$, consider the component $C$  of
 $S_{an,k}$ which contains it, and let $m$ be
 a minimal element of $C$. Then, by (7), $\chi(m)\leq 0$. But this is  a contradiction since
 $m\in C_n$, hence $\chi(m)>n\geq 0$.

(9) Recall that $S_{top, n}$ is connected for any $n\geq 1$, see  Proposition \ref{9STR2}{\it (b)}.
Therefore, part (8) applied for $n\geq 1$   shows that
$S_{an, n}$ is connected too.
Furthermore, for $n=0$ it implies that at
the 0-graded level ${\mathfrak{H}}^0_{{\rm deg}=0}:\bH^0_{an,0}(X,o)_{{\rm deg}=0}\to \bH^0_{top,0}(M)_{{\rm deg}=0}$  is injective,
or, at the level of degree zero vertices of roots,
 $\RR_{top,0}(M)_{{\rm deg}=0} \to \RR_{an,0}( X,o)_{{\rm deg}=0}$ is surjective.


\begin{problem}\label{prob:4.7.4}
 {\it (a) For a fixed topological type find all the possible graded $\Z[U]$--modules $\{\bH^*_{an,0}\}_{an}$,
associated with all the possible analytic structures supported on that topological type. (For such a concrete classification  see
Examples \ref{ex:ALLANTHESAME}, \ref{ex:CONTtwocuspsAN} and \ref{ex:ELLipticAN}.)

(b) For a fixed topological type (hence for a fixed $\bH^*_{top,0}(M)$) and analytic type $(X,o)$ supported on it
find special properties of
 $\bH^*_{an,0}(X,o)$ (and of the morphism $\bH^*_{an,0}(X,o)\to \bH^*_{top,0}(M)$), which might characterize the
 classification  from part (a).}
\end{problem}

\noindent
Several examples support the following conjecture (see e.g. \ref{ex:CONTtwocuspsAN}), which makes Problem
\ref{prob:4.7.4} more precise.

\begin{conjecture}
Fix a topological type. Then for any analytic type supported on it
$\bH^*_{an,0}(X,o)\to \bH^*_{top,0}(M)$ is injective. At the graded roots level,
the graded graph-morphism $\RR_{top,0}(M)\to \RR_{an,0}( X,o)$ is surjective (at the level of  vertices and edges).

Hence, Problem \ref{prob:4.7.4} reads as follows:
 characterize those graded $\Z[U]$--submodules of $\bH^*_{top,0}(M)$ which appear as $\bH^*_{an,0}(X,o)$.
\end{conjecture}
Note that once $M$ is fixed, there are only finitely many graded $\Z[U]$--modules $\bH^*_{an,0}(X,o)$ and graded roots
$\RR_{an,0}( X,o)$ which satisfy this Conjecture.

\subsection{Examples}\label{ss:anEx}

\notation\label{9zu1} 
  Consider the graded $\Z[U]$-module
$\calt:=\Z[U,U^{-1}]$, and  (following \cite{OSzP}) denote by
 $\calt_0^+$ its quotient by the submodule  $U\cdot \Z[U]$.
This has a grading in such a way that $\deg(U^{-d})=2d$ ($d\geq
0$). Similarly, for any $n\geq 1$,  the quotient of $U^{-(n-1)}\cdot \Z[U]$
by $U\cdot \Z[U]$ (with the same grading) defines the  graded module
$\calt_0(n)$. Hence, $\calt_0(n)$, as a $\Z$-module, is freely
generated by $1,U^{-1},\ldots,U^{-(n-1)}$, and has finite
$\Z$-rank $n$.

More generally, for any graded $\Z[U]$-module $P$ with
$d$-homogeneous elements $P_d$, and  for any  $r\in\Q$,   we
denote by $P[r]$ the same module graded (by $\Q$) in such a way
that $P[r]_{d+r}=P_{d}$. Then set $\calt^+_r:=\calt^+_0[r]$ and
$\calt_r(n):=\calt_0(n)[r]$. Hence, for $m\in \Z$,
$\calt_{2m}^+=\Z\langle U^{-m}, U^{-m-1},\ldots\rangle$ as a $\Z$-module.

\begin{example}\label{ex:EATROOTAN} We claim that the following facts are equivalent:
(i){\bf  $(X,o)$ is rational}, (ii) $\RR_{an,0}(X,o)=\RR_{(0)}$, (iii)
$\bH^*_{an,0}(X,o)=\calt^+_0$, (iv) $\bH^*_{an,0}(X,o)=\calt^+_{2m}$ for some $m$, (v)
$S_{an,0}$ is contractible.

Indeed, if $(X,o)$ is rational then $h^1(\calO_l)=0$ for $l>0$, hence $w_{an}$ is increasing.
In particular, any nonempty $S_{an,n}$ can be contracted to the origin, and the analytic root is $\RR_{(0)}$,
 and $\bH^*_{an,0}=\calt^+_0$. (ii)$\Rightarrow$(iii)$\Rightarrow$(iv)$\Rightarrow$(v) are clear.
Assume that (v) holds. 
Since $w_{an}(E_v)=\hh(E_v)-h^1(\cO_{E_v})=1$, the connectivity of $S_{an,0}$ implies that
$w_{an}|_{L_{>0}}>0$. Hence $\chi |_{L_{>0}}>0$ too, and $(X,o)$ is rational by Artin's criterion.

The above criterions are equivalent with the vanishing of the reduced cohomology $\bH^*_{an,red,0}(X,o)=0$ too.
\end{example}

\begin{example}\label{NNAN}
Consider the hypersurface singularity with nondegenerate principal part from Example \ref{ex:twonodesB}.
In this case $\bH^1_{an,0}$ is nonzero. In this Newton nondegenerate suspension case
$\hh$ is  computed combinatorially using the weighted lattice points under the Newton diagram, and
$\hh^\circ (l)=\hh(Z_K-l)$.

This shows that $\bH^{>0}_{an,0}(X,o)\not=0$ can happen in the analytic case as well.
\end{example}

\begin{example}\label{ex:GENERICANTYPE} {\bf $\bH^*_{an,0}$ for the generic analytic structure.}
Let us fix a {\it non--rational} resolution graph $\Gamma$ with $M(\Gamma)$ a $\Q HS^3$.
Assume that $\tX$ is a resolution space of a singularity $(X,o)$ with dual graph $\Gamma$
 and generic analytic structure in the sense of \cite{LauferDef,NNII}. In \cite{NNII} the following facts are proved:
 (i) $p_g(X,o)=1-\min\,\chi$, \ (ii) $Z_{max}=\max\{ l\,:\, \chi(l)=\min\,\chi\}$, \
  (ii) $Z_{coh}=\min\{ l\,:\, \chi(l)=\min\,\chi\}$. In particular, $Z_{coh}\leq Z_{max}$.

Since $\bH^*_{an,0}$ can be computed in $R(0,Z_{coh})$ (cf. \ref{lem:INDEPAN}), and
$\hh(0)=0$ and  $\hh(l)=1$ for $0<l\leq Z_{max}$ --- hence similar identities hold in $R(0, Z_{coh})$ too ---,
the computation is immediate.
Note also that  $h^1(\calO_{Z_{coh}})=p_g=1-\min\,\chi$.
 Therefore,
  $w_{an}(Z_{coh})=w_{an}(Z_{max})=1-p_g=\min\,\chi$. Furthermore,
  $\bH^0_{an,0}(X,o)\simeq\calt^+_{2\,\min\,\chi}\oplus \calt_0(1)$ and  $\bH^{\geq 1}_{an,0}(X,o)=0$. The analytic graded  root is:

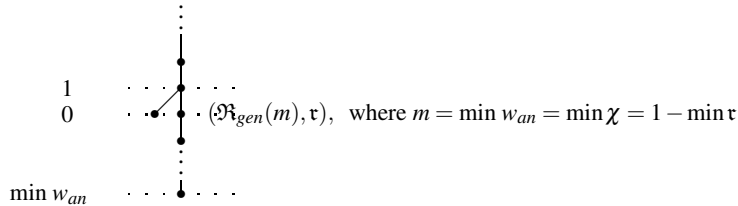
\begin{picture}(300,100)(20,330)

\put(77,380){\makebox(0,0){\small{$0$}}} \dashline{1}(100,350)(140,350)
\put(77,390){\makebox(0,0){\small{$1$}}}\put(70,350){\makebox(0,0){\small{$\min\,w_{an}$}}}
\dashline{1}(100,380)(140,380) \dashline{1}(100,390)(140,390)
\put(120,420){\makebox(0,0){$\vdots$}}
\put(120,365){\makebox(0,0){$\vdots$}}
 \put(120,350){\circle*{3}}
 \put(120,350){\line(0,1){5}}
 \put(120,400){\circle*{3}}
\put(120,390){\circle*{3}} \put(110,380){\circle*{3}}
\put(120,380){\circle*{3}}
\put(120,370){\circle*{3}} \put(120,410){\line(0,-1){40}}
\put(110,380){\line(1,1){10}} 
\put(230,380){\makebox(0,0){\small{$(\RR_{gen}(m),\chic), \ \ \mbox{where}
 \ m=\min\,w_{an}=\min\chi=1-\min\chic$}}}
\end{picture}

\noindent Note that in general $eu(\bH^0_{an,0})\geq 1-\min\,w_{an}$ (use (\ref{eq:euh0})). However, in the present case of the
  generic analytic structure we have the sharp equality:
$eu(\bH^0_{an,0})=p_g=1-\min \, w_{an}=1-\min\,\chi=1-\min\, \chic$.

\end{example}

\begin{notation}\label{not:gen}
A graded root $(\RR,\chic)$ of the shape as in Example \ref{ex:GENERICANTYPE} is denoted by $\RR_{gen}(m)$, where
$m=\min\,\chic$.
\end{notation}

\begin{example}\label{ex:ALLANTHESAME}
Assume that $\Gamma$ is star--shaped with central vertex $-2$,  with five legs, each of them
having one vertex with Euler number $-4$. Then the algorithm from \cite[\S11]{NOSz} (applied for the
canonical spin$^c$--structure) shows that the topological graded root is $\RR_{gen}(-1)$ with
$\min\,w_0=-1$ (cf. Notation \ref{not:gen}).
%
%
%
In fact, that algorithm also shows that that for {\it any} analytic structure $Z_{coh}=x(2)\leq x(3)=Z_{min}$, hence,
the analytic graded root is $R_{gen}(-1)$ independently of the analytic structure.

Note that from   \ref{ex:ARAN}
we also obtain that  $\bH^*_{an,0} (X,o)=\bH^*_{top}(M,-Z_K)$.

 In particular, for this graph,
$\bH^*_{an,0}=\bH^*_{top,0}$ for {\it any} analytic structure supported on $\Gamma$.
\end{example}

\begin{example}\label{ex:CONTtwocuspsAN} Consider the  topological type  fixed by the following
resolution graph $\Gamma$.

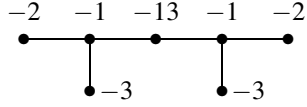
\begin{picture}(300,45)(10,0)
\put(125,25){\circle*{4}} \put(150,25){\circle*{4}}
\put(175,25){\circle*{4}} \put(200,25){\circle*{4}}
\put(225,25){\circle*{4}} \put(150,5){\circle*{4}}
\put(200,5){\circle*{4}} \put(125,25){\line(1,0){100}}
\put(150,25){\line(0,-1){20}} \put(200,25){\line(0,-1){20}}
\put(125,35){\makebox(0,0){$-2$}}
\put(150,35){\makebox(0,0){$-1$}}
\put(175,35){\makebox(0,0){$-13$}}
\put(200,35){\makebox(0,0){$-1$}}
\put(225,35){\makebox(0,0){$-2$}} \put(160,5){\makebox(0,0){$-3$}}
\put(210,5){\makebox(0,0){$-3$}}
\end{picture}

\noindent In this case the link is an integral homology sphere and $\bH^1_{top,0}\not=0$, cf. \cite{Nlattice}.

\noindent The list of possible analytic structures
 supported on $\Gamma$ is the following, see   \cite{NO2cusps}:

(i) non-Gorenstein Kulikov analytic type with $p_g=3$  and $Z_{max}=Z_{min}$,

(ii) complete intersection (hence Gorenstein) splice quotient  with $p_g=3$ and $Z_{max}=2Z_{min}$,

(iii) analytic structures with $p_g=2$ and $Z_{coh}\leq Z_{min}<2Z_{min}\leq Z_{max}$ (they are non--Gorenstein).

\noindent
In all these cases $\bH^{\geq 1}_{an,0}=0$, however the modules $\bH^0_{an,0}$ are all different. Below we give the graded roots
(and also the topological root, for comparison).
The needed information regarding $\hh$ and $\hh^\circ$ can be deduced e.g. from \cite{NO2cusps}.

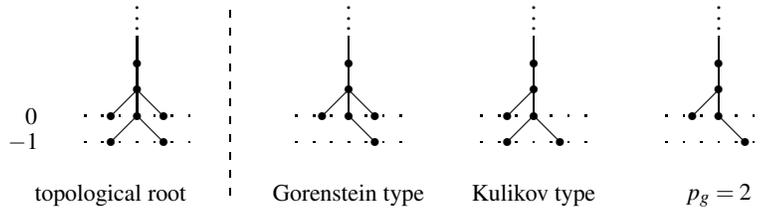
\begin{picture}(300,90)(120,340)

 \put(210,350){\makebox(0,0){\small{topological root}}}
\put(180,380){\makebox(0,0){\small{$0$}}} \put(177,370){\makebox(0,0){\small{$-1$}}}
\dashline{1}(200,370)(240,370)
\dashline{1}(200,380)(240,380)
\put(220,420){\makebox(0,0){$\vdots$}} \put(220,400){\circle*{3}}
\put(220,390){\circle*{3}} \put(210,380){\circle*{3}}
\put(230,380){\circle*{3}} \put(220,380){\circle*{3}}\put(210,370){\circle*{3}}
\put(230,370){\circle*{3}} \put(220,410){\line(0,-1){30}}
\put(210,380){\line(1,1){10}} \put(220,390){\line(1,-1){10}}
\put(210,370){\line(1,1){10}} \put(220,380){\line(1,-1){10}}

\dashline{3}(255,350)(255,420)

 \put(300,350){\makebox(0,0){\small{Gorenstein type}}}
\dashline{1}(280,370)(320,370)
\dashline{1}(280,380)(320,380)
\put(300,420){\makebox(0,0){$\vdots$}} \put(300,400){\circle*{3}}
\put(300,390){\circle*{3}} \put(290,380){\circle*{3}}
\put(310,380){\circle*{3}} \put(300,380){\circle*{3}}
\put(310,370){\circle*{3}}
\put(300,410){\line(0,-1){30}}
\put(290,380){\line(1,1){10}} \put(300,390){\line(1,-1){10}}
\put(300,380){\line(1,-1){10}}

 \put(370,350){\makebox(0,0){\small{Kulikov type}}}
\dashline{1}(350,370)(390,370)
\dashline{1}(350,380)(390,380)
\put(370,420){\makebox(0,0){$\vdots$}} \put(370,400){\circle*{3}}
\put(370,390){\circle*{3}} \put(360,380){\circle*{3}}
\put(370,380){\circle*{3}}\put(360,370){\circle*{3}}
\put(380,370){\circle*{3}}
\put(370,410){\line(0,-1){30}}
\put(360,380){\line(1,1){10}} 
\put(360,370){\line(1,1){10}} \put(370,380){\line(1,-1){10}}

 \put(440,350){\makebox(0,0){\small{$p_g=2$}}}
\dashline{1}(420,370)(460,370)
\dashline{1}(420,380)(460,380)
\put(440,420){\makebox(0,0){$\vdots$}} \put(440,400){\circle*{3}}
\put(440,390){\circle*{3}} \put(430,380){\circle*{3}}
\put(440,380){\circle*{3}}
\put(450,370){\circle*{3}}
\put(440,410){\line(0,-1){30}}
\put(430,380){\line(1,1){10}} 
 \put(440,380){\line(1,-1){10}}
\end{picture}

\noindent The local minima of the topological graded root reflect elements the topological semigroup (Lipman cone)$\calS$
in the  rectangle $R(0, Z_K)$. On the other hand, the local minima
 of each analytic graded root reflect the divisors of functions  from the
analytic semigroup in the corresponding rectangle $R(0, Z_{coh})$
(of the corresponding
analytic type), cf. \ref{bek:CONSEQW}(6).
The above pictures indicate very intuitively which elements of the topological semigroup $\calS$
are not realized as divisors of  functions (do not belong to $\calS_{an}$)
 in different analytic structures: in the Goresntein case $Z_{min}$, in the non--Goresntein cases $Z_K$.
\end{example}

\begin{remark}
When we fix a topological type and we compare the possible graded roots (and analytic lattice cohomologies)
associated with different analytic structures then we realize that  the `simplest' object is
provided by the generic analytic structure. However, it is very hard to identify those analytic structures
which provide the `most complicated' modules/roots. Of course, the candidate for the most complicated module/root
is provided by $\bH^*_{top,0}(M)$ and $\RR_{top,0}(M)$ respectively. However, they are not always
realized by a certain analytic structure, see Example \ref{ex:CONTtwocuspsAN}.
\end{remark}
\begin{problem} {\it
Which $\bH^*_{top,0}(M)$ and $\RR_{top,0}(M)$ is  realized by a certain analytic structure?
If yes, for which (very special) analytic structure are they realized?
(Can this family be identified universally by certain analytic properties?) }
\end{problem}
\begin{remark}\label{rem:sq}
The Gorenstein case from Example \ref{ex:CONTtwocuspsAN}  shows that even for {\bf splice quotient} analytic types
$\bH^q_{an,0}\not=\bH^q_{top,0}$   might happen. (In  \ref{ex:CONTtwocuspsAN} they differ for both $q=0, 1$.)
However, the validity of SWIC guarantees that $eu(\bH^q_{an,0})=eu(\bH^q_{top,0})$ for any splice quotient.
\end{remark}

\begin{problem}{\it
Consider the equisingular family of splice quotient singularities associated with a (convenient) graph $\Gamma$.
Is it true that $\bH^*_{an,0}$ is constant along this family?
If yes (hence $\bH^*_{an,0}$  is determined by $\Gamma)$ describe it combinatorially from $\Gamma$.
(Note that usually it is \underline{not} $\bH^*_{top,0}(M(\Gamma))$, cf. Remark \ref{rem:sq}.)}
\end{problem}
\begin{example}\label{ex:ELLipticAN}  
 Assume that $(X,o)$ is a {\bf numerically Gorenstein elliptic
singularity}  with rational homology sphere link.
In the next discussion we use the  notations  of  \cite{weakly}.

In the Gorenstein case we already know (see \ref{ex:ARAN}) that $\bH^*_{an,0}(X,o)=\bH^*_{top,0}(M)$.

Next, consider a graph with $m=2$. In the Gorenstein case $p_g=3$, and $\bH^*_{an,0}$ is determined above.
Next, assume that $p_g=2$. 
Since $(X_1,o_1)$ is Gorenstein, $h^1(\calO_{C'_1})=2$, hence $Z_{coh}\leq C'_1=C+Z_{B_1}$.
On the other hand, $H^0(\calO_{\tX}(-C_1))=H^0(\calO_{\tX}(-Z_{min}))$, hence $Z_{max}=C_1=Z_{min}+Z_{B_1}$.
Therefore, $Z_{coh}<Z_{max}$, and $w_{an}(Z_{coh})=1-p_g=-1$.

Since $eu=p_g=2$, this shows that the analytic graded root is necessarily
$R_{gen}(-1)$. The cohomology groups are $\bH^{\geq 1}_{an,0}=0$ and $\bH^0_{an,0}=\calt^+_{-2}\oplus \calt_0(1)$.
Surprisingly, even if $\Gamma$ is elliptic with $\min\,\chi=0$, for this analytic structure $\min\,w_{an}<0$.

For the discussion of the general elliptic case see \cite{AgNeEll}.

\end{example}

\begin{question}\label{question:AN}
Let us fix a non-rational topological type $\Gamma$ with $M$ a $\Q HS^3$.
What are the possible values of $\min\,w_{an}$ when we consider all the analytic types supported on $\Gamma$?
(Note that $1-p_g\leq \min\,w_{an}\leq \min\,\chi$.)
\end{question}

\subsection{$\bH^*_{an,0}$ and $p_g$--constant deformations} \label{bek:DEFAN}
\bekezdes
First we formulate the following conjecture. 

\begin{conjecture}\label{conj:defAN}
$\bH^*_{an,0}$ is constant along flat $p_g$--constant deformations of normal surface singularities
(with $\Q HS^3$ links).
\end{conjecture}


\begin{example}\label{ex:CONJANNSIpelda}
Consider the superisolated singularity $(X,o)$
associated with a rational projective plane curve $C$ degree 5 with two cusps, both with
one  Puiseux pair, namely  (3,4) and (2,7) respectively  (for notations see \cite{Spany}.
Denote the two singular  points of $C$ by $p$ and $q$ respectively.
The curve $C$ is projective equivalent with
$y^3z^2+x^4z+x^3yz-x^2y^2z/2-x^5/4+x^4y/16$ \cite{Y4}. The resolution graph of $(X,o)$ is

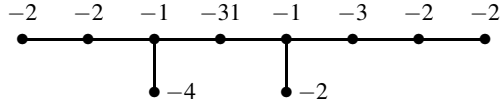
\begin{picture}(300,45)(20,0)
\put(125,25){\circle*{4}} \put(150,25){\circle*{4}}
\put(175,25){\circle*{4}} \put(200,25){\circle*{4}}
\put(225,25){\circle*{4}} \put(150,5){\circle*{4}}
\put(200,5){\circle*{4}} \put(100,25){\line(1,0){175}}
\put(150,25){\line(0,-1){20}} \put(200,25){\line(0,-1){20}}
\put(125,35){\makebox(0,0){\small{$-2$}}}
\put(150,35){\makebox(0,0){\small {$-1$}}}
\put(175,35){\makebox(0,0){\small{$-31$}}}
\put(200,35){\makebox(0,0){\small{$-1$}}}
\put(225,35){\makebox(0,0){\small{$-3$}}} \put(160,5){\makebox(0,0){\small{$-4$}}}
\put(210,5){\makebox(0,0){\small{$-2$}}} \put(100,25){\circle*{4}}
\put(250,25){\circle*{4}} \put(275,25){\circle*{4}}
\put(100,35){\makebox(0,0){\small{$-2$}}}
\put(250,35){\makebox(0,0){\small{$-2$}}}
\put(275,35){\makebox(0,0){\small{$-2$}}}
\end{picture}

This is a numerically Gorenstein graph, and we can apply the topological reduction theorem for
 the set of nodes.  Then the topological lattice cohomology
computation is reduced to this  2-dimensional rectangle $R((0,0),(30,34))$. For the complete list of
 $\chi$-weights  see \cite{LThesis}.  At $(0,0)$ and $(30,34)$
one has two symmetric  local minima.
They will generate a summand $\calt_0(1)^2$ in $\bH^0_{top}(\Gamma,-Z_K)$.
The other generators and relations can be read from
the rectangle $R((12,14),(18,20))$ with the corresponding $\chi$--values:\\

{\small

\hspace*{3cm}$\begin{matrix}
{\bf -2} & {\bf -2} & {\bf -3} & -4 & -4 & -4 & \mathbf{-5} \\
{\bf -2} & -1 & {\bf -2} & -3 & -3 & -3 & -4\\
{\bf -3} & {\bf -2} & {\bf -2} & -3 & -3 & -3 & -4\\
-3 & -2 & -2 & -2 & -2 & -2 & -3\\
-4 & -3 & -3 & -3 & {\bf -2} & {\bf -2} & {\bf -3}\\
-4 & -3 & -3 & -3 & {\bf -2} & -1 & {\bf -2}\\
\mathbf{-5} & -4 & -4 & -4 & {\bf -3} & {\bf -2} & {\bf -2}\end{matrix}$  }\\

The two boldface $\mathbf{ -5}$'s are new generators of $\bH^0_{top}(M, -Z_K)  $ of degree $-10$,
the middle $-2$ is a saddle point corresponding to the meeting of the two long legs of the graded root,
and the boldface loops generate $\bH^1_{top}(M, -Z_K)  $.
Hence $\bH^0_{top}(\Gamma, -Z_K)=
\calt^+_{-10}\oplus \calt_{-10}(3)\oplus \calt_0(1)^2$ and
 $\bH^1_{top}=\calt_{-2}(1)^2$ (see also
 \cite{NSig}).  The topological (canonical) graded root is

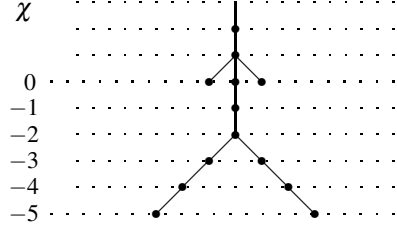
\begin{picture}(200,100)(-100,320)

\dashline{1}(40,410)(160,410) \dashline{1}(40,400)(160,400)
\put(20,380){\makebox(0,0){\small{\ \ 0}}} \dashline{1}(40,390)(160,390)
\put(20,370){\makebox(0,0){\small{$-1$}}} \dashline{1}(30,380)(160,380)
\put(20,360){\makebox(0,0){\small{$-2$}}} \dashline{1}(40,370)(160,370)
\put(20,350){\makebox(0,0){\small{$-3$}}} \dashline{1}(40,360)(160,360)
\put(20,340){\makebox(0,0){\small{$-4$}}} \dashline{1}(40,350)(160,350)
\dashline{1}(40,340)(160,340) \dashline{1}(30,330)(160,330)
\put(20,330){\makebox(0,0){\small{$-5$}}} \put(90,380){\circle*{3}}
\put(110,380){\circle*{3}} \put(100,390){\circle*{3}}
\put(100,400){\circle*{3}}
\put(100,360){\circle*{3}}\put(100,370){\circle*{3}}\put(100,380){\circle*{3}}
\put(90,350){\circle*{3}} \put(110,350){\circle*{3}}
\put(80,340){\circle*{3}} \put(120,340){\circle*{3}}
\put(70,330){\circle*{3}} \put(130,330){\circle*{3}}
\put(100,410){\line(0,-1){50}} \put(100,390){\line(1,-1){10}}
\put(100,390){\line(-1,-1){10}} \put(100,360){\line(-1,-1){30}}
\put(100,360){\line(1,-1){30}}
 \put(20,410){\makebox(0,0)[t]{$\chi$}}
\end{picture}

The point is that  there exists a $K_{min}^2$-- and $p_g$--constant deformation from $(X,o)$ to a singularity $(X_{5,5,6},o)$,
(equisingular to the Brieskorn hypersurface $x^5+y^5+z^6$), whose link is $\Sigma(5,5,6)$.

Indeed,
assume that the rational unicuspidal curve is given by $f_d(x,y,z)=0$ in $\C{\mathbb P}_2$.
We can fix the homogeneous coordinates in $\C{\mathbb P}_2$
in such a way that $z=0$ intersects $C$ generically.
A possible choice for the superisolated singularity
$f:(\C^3,0)\to(\C,0)$ (up to equisingularity) is $f=f_d+z^{d+1}$. Write $f_d$ as $\sum_{i=0}^d  g_{d-i}(x,y)z^i$.
Then $g_d$ is a product of $d$ linear
factors corresponding to the  points $C\cap \{z=0\}$, hence the germ
$g_d:(\C^2,0)\to (\C,0)$ is
equisingular with $(x,y)\mapsto x^d+y^d$.
Next, consider the following deformation $f_t:(\C^3,0)\to (\C,0)$ of isolated hypersurface germs, given by $f_t(x,y,z)=f_d(x,y,tz)+z^{d+1}=
\sum_i g_{d-i}(x,y)z^it^i+z^{d+1}$. For $t\not=0$ the deformation is
$\mu$-constant, the embedded topological type stays also constant, and it is equivalent
(up to such equivalences) to the type of  $f$. However, for $t=0$ it is
equivalent (in similar sense) to the germ $x^d+y^d+z^{d+1}$.

 The minimal good
resolution graph of $X_{d,d,d+1}$ is star-shaped with a $(-1)$ central vertex, and
it has $d$ identical legs, each consisting of one vertex with self-intersection number
$-(d+1)$. We invite the reader to verify that the deformation is $K_{min}^2$--constant (that means that $Z_K^2$ computed in the
minimal resolution is $t$--independent), and also $p_g$--constant.

 Since $(X_{5,5,6},o)$ is AR and for it the SWIC holds, by \ref{ex:ARAN}
$\bH^*_{an,0}(X_{5,5,6},o)=\bH^*_{top,0}(\Sigma(5,5,6))$. That is, $\bH^{\geq 1}_{an,0}(X_{5,5,6},o)=0$ and the
 {\it analytic} graded root of
$(X_{5,5,6},o)$ agrees with the topological graded root of $\Sigma(5,5,6)$.
On the other hand, one verifies (see e.g. \cite{Spany,BLMN2,Book}
that the topological lattice cohomologies and graded roots of  $M(\Gamma)$ and $\Sigma(5,5,6)$ agree
(for the graded root see the root above).  Hence, the conjecture predicts that the same facts are valid for
$(X,o)$ too, just like for  $(X_{5,5,6},o)$. This, in particular means that $\bH^0_{an,0}(X,o)=\bH^0_{top,0}(M)$ and
$\bH^1_{an,0}(X,o)=0$. This is what we will verify next.

By the {\it analytic} Reduction Theorem \ref{th:REDAN} (applied for the two nodes)
we have to complete the $w_{an}$--table on $R((0,0),(30,34))$. Since $(X,o)$ is Gorenstein,
$\hh^\circ(l)=\hh(Z_K-l)$, cf. \ref{ex:GorAN},
hence we can concentrate only on $\hh$. Note also that $Z_{max}=Z_{min}$, and the generic line has multiplicities
$(12,14)$ along the nodes. Therefore, $\hh(l)=1$ for any $0<l\leq (12,14)$. Furthermore (by Laufer's algorithm \cite{Laufer72})
$h^1(\calO_{Z_{min}})=6$. Hence, basically it is enough to analyse
the symmetric rectangle $R((12,14), (18,20)) $, as in the above topological case.
The values of $\hh$ can be computed using computation sequences
via the following principles: (a) if $E_n$ is a node, $x(\ell)\in\calS_{an}$ and  $(x(\ell),E_n)=0$
 then ($\dag$) $\hh(x(\ell)+E_v)\geq \hh(x(\ell))+1$; (b)
$\hh$ is constant along a computation sequence which connects $x(\ell)+E_n$ with $s(x(\ell)+E_n)$.
In this rectangle we have the following elements of $\calS_{an}$, all with the above properties:
$x(12,14)$ being the divisor of the  generic line,  $x(15,14)$ of the generic line which contains $p$,
$x(16,14)$ of the tangent line at $p$, $x(12,16)$ of  the generic line which contains $q$,
$x(15,16)$ of  the  line which contains both  $p$ and  $q$, and finally,
$x(12,18)$ of the tangent line at $q$.
(Note also that $\hh((18,20))=\hh(Z_K-Z_{min})=\hh^\circ(Z_{min})=p_g-h^1(\calO_{Z_{min}})=4$,
hence no other semigroup element can  contribute and in $(\dag)$ we always have equality.)
These semigroup places are underlined in the next diagram from left,
which  shows the  $\hh$ values in
the rectangle $R((12,14),(18,20))$:\\

{\small

\hspace*{2cm}$\begin{matrix}
4 & \ \ 4 & \ \ 4 & \ \ 4 & \ \ 4 & \ \ 4 & \ \ 4 \\
4  & \ \ 4 & \ \ 4 & \ \ 4 & \ \ 4 & \ \ 4 & \ \ 4\\
\underline{3}  & \ \ 4 & \ \ 4 & \ \ 4 & \ \ 4 & \ \ 4 & \ \ 4\\
3 & \ \ 4 & \ \ 4 & \ \ 4 & \ \ 4 & \ \ 4 & \ \ 4\\
\underline{2} & \ \ 3 & \ \ 3 & \ \ \underline{3} & \ \ 4 & \ \ 4& \ \ 4\\
2 & \ \ 3 & \ \ 3 & \ \ 3 & \ \ 4 & \ \ 4& \ \ 4\\
\underline{1} & \ \ 2 & \ \ 2 & \ \ \underline{2} & \ \ \underline{3} & \ \ 4 & \ \ 4\end{matrix}$
\hspace*{1cm}$\begin{matrix}
-2 & \ -2 & \ -3 & \ -4 & \ -4 & \ -4 & -5 \\
-2 & -2 & -2 & -3 & -3 & -3 & -4\\
-3 & -2 & -2 & -3 & -3 & -3 & -4\\
-3 & -2 & -2 & -2 & -2 & -2 & -3\\
-4 & -3 & -3 & -3 &  -2 &  -2 &  -3\\
-4 & -3 & -3 & -3 & -2 & -2 &  -2\\
-5 & -4 & -4 & -4 &  -3 & -2 & -2\end{matrix}$  }\\

\noindent
The $w_{an}$--table  is given on the right diagram.
Hence the wished statement can be read from it.

It is instructive to compare the above  analytic $w_{an}$--table with the
$w_{top}=\chi$--table from above.

\vspace{2mm}

We note that this topological type (identified by the graph $\Gamma$) admits a splice quotient analytic structure as well.
WE expect that for this analytic structure
$\bH^*_{an,}(X,o)=\bH^*_{top,0}(M)$ and $\RR_{an,0}(X,o)=\RR_{top,0}(M)$.

\end{example}

\begin{example} The following deformation  was communicated to us by Ignacio Luengo.
Let us consider the deformation of hypersurface singularities
$zy^3+x^5+z^{11}+tx^2y^2=0$, where the deformation parameter $t$ is small.
Along the deformation the following objects stay constant: the (non--degenerate)
Newton diagram,  the Milnor number, the geometric genus,
the link (and even the embedded topological type). The stable resolution graph is

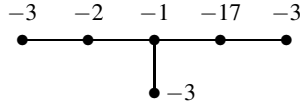
\begin{picture}(300,45)(20,0)
\put(125,25){\circle*{4}} \put(150,25){\circle*{4}}
\put(175,25){\circle*{4}} \put(200,25){\circle*{4}}
\put(100,25){\circle*{4}} \put(150,5){\circle*{4}}

 \put(100,25){\line(1,0){100}}
\put(150,25){\line(0,-1){20}}
\put(100,35){\makebox(0,0){\small{$-3$}}}
\put(125,35){\makebox(0,0){\small{$-2$}}}
\put(150,35){\makebox(0,0){\small {$-1$}}}
\put(175,35){\makebox(0,0){\small{$-17$}}}
\put(200,35){\makebox(0,0){\small{$-3$}}}
 \put(160,5){\makebox(0,0){\small{$-3$}}}
\end{picture}

Since the graph is AR, and  $p_g=\min_{\gamma}\, eu \bH^0_{top}(\gamma, \Gamma, -Z_K)$, cf. \cite{NSig},
 by \ref{ex:ARAN},
$\bH^*_{an,0}=\bH^*_{top,0}$ for any $t$. Since $\bH^*_{top,0}$ is $t$-stable, we obtain that
 $\bH^*_{an,0}$ must stay  $t$-stable as well.
In this example the point is that the deformation is not
$\mu^*$--stable ($\mu_1^{t=0}=11$ while $\mu_1^{t\not=0}=10$) (that is,
the deformation does not admit a strong simultaneous resolution). However, it still supports the above conjecture:
the stability of $\bH_{an,0}^*$.

\end{example}

\begin{example}\label{ex:Topjumps}
Consider the deformation of hypersurface singularities
$(X_t,o)=\{x^7+y^5+z^3+tx^4y^2=0\}$, $t\in (\C,0)$.
If $t=0$ then $(X_{t=0},0)$ is weighted homogeneous with the following minimal resolution graph, where the unmarked
vertices have  $(-2)$ decorations.

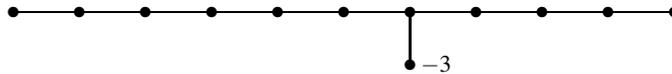
\begin{picture}(300,45)(20,0)
\put(125,25){\circle*{4}} \put(150,25){\circle*{4}}
\put(175,25){\circle*{4}} \put(200,25){\circle*{4}}
\put(225,25){\circle*{4}}
\put(350,25){\circle*{4}} \put(250,25){\circle*{4}}
\put(275,25){\circle*{4}} \put(300,25){\circle*{4}}
\put(325,25){\circle*{4}}\put(375,25){\circle*{4}}
\put(275,5){\circle*{4}}

\put(125,25){\line(1,0){250}}
\put(275,25){\line(0,-1){20}}
\put(285,5){\makebox(0,0){\small{$-3$}}}
\end{picture}

For $t\not=0$ the germ $(X_t,0)$ has nondegenerate Newton principal part, its graph is

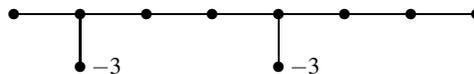
\begin{picture}(300,45)(-20,0)
\put(125,25){\circle*{4}} \put(150,25){\circle*{4}}
\put(175,25){\circle*{4}} \put(200,25){\circle*{4}}
\put(225,25){\circle*{4}}
 \put(250,25){\circle*{4}}
\put(275,25){\circle*{4}} \put(300,25){\circle*{4}}
\put(225,5){\circle*{4}}
\put(150,5){\circle*{4}}

\put(125,25){\line(1,0){175}}
\put(150,25){\line(0,-1){20}}
\put(225,25){\line(0,-1){20}}
\put(235,5){\makebox(0,0){\small{$-3$}}}
\put(160,5){\makebox(0,0){\small {$-3$}}}
\end{picture}

Note that the deformation modifies the Newton diagram, and also the topological type.
However, along the deformation both $K_{\tX}^2$ and $p_g$ remain constant:
$K_{\tX}^2=-12$ and $p_g=4$.
(I.e., the deformation does not admit a weak simultaneous resolution, but it admits
a very weak simultaneous resolution, cf. \cite{LauferWeak}.)
Note also that both graphs are AR. Since both analytic types are Newton nondegenerate,
for both of them $p_g=\min_{\gamma}\, eu \bH^0_{top}(\gamma, \Gamma, -Z_K)$, cf. \cite{NSig}.
Hence, for any fixed $t$, by Proposition \ref{prop:AnLatAR}
the analytic and topological lattice cohomologies agree. In fact, by AR property
$\bH^{\geq 1}_{an,0}=0$ in all cases. On the other hand, for both $t=0$ and $t\not=0$,
the topological graded roots can be computed by the AR--algorithm  and it turns out that they agree.
Hence $\bH^0_{an,0}$ and the analytic graded root
is independent of $t$ as well.  The (common) graded root  agrees with the topological graded root from the Example
\ref{ex:CONTtwocuspsAN}.

In particular, since $\bH^*_{an,0}$ is $t$--independent, the example supports the Conjecture, and it is also
compatible with Conjecture 11.3.51 from \cite{Book}, which predicts that along such  deformations
the topological graded root stays  constant.
\end{example}

\begin{problem} {\it
Consider the equisingular family of hypersurface singularities with Newton nondegenerate principal part associated with
a Newton diagram $N\Gamma$.
Is it true that $\bH^*_{an,0}$ is constant along this family?
If yes (hence $\bH^*_{an,0}$  is determined by $N\Gamma)$ describe it combinatorially from the Newton diagram $N\Gamma$.}
\end{problem}

\section{Combinatorial lattice cohomology} \label{ss:CombLattice}

In this section we prove several combinatorial statements regarding the lattice cohomology
associated with any weight function with certain combinatorial  properties.

Here we also indicate that the geometric situations where some
analytic weight function can be defined  (hence a lattice cohomology too) is very diverse and rich.

\subsection{The combinatorial setup}\label{ss:combsetup}

\bekezdes \label{bek:comblattice}
Fix  $\Z^s$ with a fixed basis $\{E_v\}_{v\in\cV}$.
Write $E_I=\sum_{v\in I}E_v$ for $I\subset \cV$ and  $E=E_{\cV}$.
Fix also  an element
$c\in \Z^s$, $c\geq E$.
Consider the lattice points $R=R(0,c):=\{l\in\Z^s\,:\, 0\leq l\leq c\}$, and assume that
to each $l\in R$ we assign

(i)   an integer $h(l)$ such that $h(0)=0$ and $h(l+E_v)\geq h(l)$
for any $v$,

(ii)  an integer $h^\circ (l)$ such that $h^\circ (l+E_v)\leq  h^\circ (l)$
for any $v$.


\noindent
Once  $h$ is fixed with (i),
a possible choice for $h^\circ $ is
 $h^{sym}$, where $h^{sym}(l)=h(c-l)$. Clearly, it depends on $c$.

We  consider the set of cubes $\{\calQ_q\}_{q\geq 0}$ of $R$ as in \ref{9complex} and
the weight function
$$w_0:\calQ_0\to\Z\ \ \mbox{by} \ \ w_0(l):=h(l)+h^\circ (l)-h^\circ (0).$$
Clearly  $w_0(0)=0$.
Furthermore, similarly as in \ref{9dEF1}, we define
$w_q:\calQ_q\to \Z$ by $ w_q(\square_q)=\max\{w_0(l)\,:\, l \
 \mbox{\,is a vertex of $\square_q$}\}$. We will use the symbol $w$ for the
 system $\{w_q\}_q$.
 The compatible weight functions define the lattice cohomology $\bH^*(R,w)$.
Moreover, for any increasing path $\gamma$ connecting 0 and $c$
 we also have
 a path lattice cohomology $\bH^0(\gamma,w)$ as in \ref{bek:pathlatticecoh}. Accordingly,  we have the numerical
Euler characteristics  $eu(\bH^*(R,w))$, $eu(\bH^0(\gamma,w))$ and $\min_\gamma eu(\bH^0(\gamma,w))$ too.

\begin{lemma}\label{lem:comblat}  We have   $0\leq eu(\bH^0(\gamma,w))\leq h^\circ (0)-h^\circ (c)$
for any increasing path $\gamma$  connecting  0 to $c$.
 The equality $eu(\bH^0(\gamma,w))=h^\circ (0)-h^\circ (c)$
holds if and only if for any $i$
the differences $h(x_{i+1})-h(x_i)$ and $h^\circ (x_{i})-h^\circ (x_{i+1})$ simultaneously are not nonzero.
\end{lemma}
\begin{proof} {\it (a)} Since $w_0(0)=0$ we have
$eu(\bH^0(\gamma,w))=-\min\, w_0+{\rm rank} \bH_{red}^0(\gamma,w)\geq 0$.
Next, by  \ref{eq:pathweights} we have  $eu(\bH^0(\gamma,w))=\sum_{i=0}^{t-1}
\max\{0,w(x_i)-w(x_{i+1})\}$.
On the other hand,
$w(c)+\sum_i
\max\{0,w(x_i)-w(x_{i+1})\}=-\sum_i
\min \{0,w(x_i)-w(x_{i+1})\}$, hence
$w(c)+ 2\cdot eu(\bH^0(\gamma,w))=\sum_i |w(x_i)-w(x_{i+1}|=
\sum_i |h(x_{i})+h^\circ (x_i)-h(x_{i+1})-h^\circ (x_{i+1})|\leq
\sum_i ( h(x_{i+1})-h(x_i))+\sum_i (h^\circ (x_i)-h^\circ (x_{i+1}))=h(c)+h^\circ (0)-h^\circ (c)$.
\end{proof}

\begin{remark}\label{rem:ketpelda}
The inequality  $0\leq eu(\bH^*(R,w))$ is not true in general. Take e.g. the following table
($s=2$, $c=(2,2)$)
with $h^\circ =h^{sym}$:

\begin{picture}(100,40)(-30,0)

\put(10,10){\makebox(0,0){\small{$0$}}}
\put(25,10){\makebox(0,0){\small{$0$}}}
\put(40,10){\makebox(0,0){\small{$1$}}}
\put(10,20){\makebox(0,0){\small{$0$}}}
\put(25,20){\makebox(0,0){\small{$1$}}}
\put(40,20){\makebox(0,0){\small{$1$}}}
\put(10,30){\makebox(0,0){\small{$0$}}}
\put(25,30){\makebox(0,0){\small{$1$}}}
\put(40,30){\makebox(0,0){\small{$1$}}}

\put(110,10){\makebox(0,0){\small{$0$}}}
\put(125,10){\makebox(0,0){\small{$0$}}}
\put(140,10){\makebox(0,0){\small{$0$}}}
\put(110,20){\makebox(0,0){\small{$0$}}}
\put(125,20){\makebox(0,0){\small{$1$}}}
\put(140,20){\makebox(0,0){\small{$0$}}}
\put(110,30){\makebox(0,0){\small{$0$}}}
\put(125,30){\makebox(0,0){\small{$0$}}}
\put(140,30){\makebox(0,0){\small{$0$}}}

\put(-10,20){\makebox(0,0){\small{$h:$}}}
\put(90,20){\makebox(0,0){\small{$w_0:$}}}
\put(210,30){\makebox(0,0){\small{$\min w_0=0$}}}
\put(208,10){\makebox(0,0){\small{$eu=-1$}}}
\put(222,20){\makebox(0,0){\small{$\bH^0_{red}=0$,\ $\bH^1=\Z$}}}
\end{picture}

\noindent
The inequality $ eu(\bH^*(R,w))\leq h^\circ (0)-h^\circ (c)$ is not true either, see  e.g. 

\begin{picture}(100,40)(-40,0)

\put(10,10){\makebox(0,0){\small{$0$}}}
\put(25,10){\makebox(0,0){\small{$0$}}}
\put(40,10){\makebox(0,0){\small{$0$}}}
\put(10,20){\makebox(0,0){\small{$0$}}}
\put(25,20){\makebox(0,0){\small{$1$}}}
\put(40,20){\makebox(0,0){\small{$1$}}}
\put(10,30){\makebox(0,0){\small{$0$}}}
\put(25,30){\makebox(0,0){\small{$1$}}}
\put(40,30){\makebox(0,0){\small{$2$}}}

\put(110,10){\makebox(0,0){\small{$0$}}}
\put(125,10){\makebox(0,0){\small{$-1$}}}
\put(140,10){\makebox(0,0){\small{$-2$}}}
\put(110,20){\makebox(0,0){\small{$-1$}}}
\put(125,20){\makebox(0,0){\small{$0$}}}
\put(140,20){\makebox(0,0){\small{$-1$}}}
\put(110,30){\makebox(0,0){\small{$-2$}}}
\put(125,30){\makebox(0,0){\small{$-1$}}}
\put(140,30){\makebox(0,0){\small{$0$}}}

\put(-10,20){\makebox(0,0){\small{$h:$}}}
\put(90,20){\makebox(0,0){\small{$w_0:$}}}

\put(200,20){\makebox(0,0){\small{\mbox{(with $h^\circ =h^{sym}$)}}}}
\end{picture}

\noindent
Then $eu(\bH^0(\gamma,w))=2$, while  $eu(\bH^0(R,w))=eu(\bH^*(R,w))=4$.
\end{remark}


\begin{example}\label{ex:latgrfilt} {\bf Graded and filtered vector spaces.}
In several geometrical constructions we face the following situation:
we have $\Z^s$ and $c$ as above, and  a finite dimensional vector space $M$  with a $\Z^s$-grading  $\{M_{{\bf a}}\}_{{\bf a}}$
 such that $M_{{\bf a}}=0$ whenever either ${\bf a}\not\geq 0$ or ${\bf a}\geq c$.  Let
$\hh$ be the Hilbert function 
$\hh(l)=\sum_{{\bf a}\not \geq l}\,\dim M_{{\bf a}}$ and let $h$ be its restriction to $R(0,c)$.
Then $h(0)=0$ and $h(c)=\dim\, M$.


More generally, assume that $M$ is a finite dimensional vector space endowed  with a decreasing  $\Z^s$-filtration  such that
$F(0)=M$  and $F(c)=0$ and define $\hh(l)=\dim (M/F(l))$ for any $l\geq 0$.
Again, define $h$ as the restriction of
$\hh$ to $R$.

 The $h$--function associated with a filtration satisfies
 the  {\it `matroid rank inequality'}
 \begin{equation}\label{eq:matroid}
 h(l_1)+h(l_2)\geq h(\min\{l_1,l_2\})+h(\max\{l_1,l_2\}), \ \ l_1,l_2\in R.
 \end{equation}
 This implies the {\it `stability property'},
 valid for any $\bar{l}\geq 0$ with $|\bar{l}|\not\ni E_v$
 \begin{equation}\label{eq:stability}
 h(l)=h(l+E_v)\ \ \Rightarrow\  \ h(l+\bar{l})=h(l+\bar{l}+E_v).
 \end{equation}
 Such a filtration might appear as follows. Take a ring $\cO$,  endowed with  $s$ valuations
 $\frakv_i :\cO\to \Z_{\geq 0}\cup \{\infty\}$ and set for any $l=(l_1,\ldots, l_s)\in \Z^s$ the ideal
 $F_{\cO}(l):=\{f\in\cO\,:\, \frakv_i(f)\geq l_i \ \mbox{for all $i$}\}$. Then, for a conveniently chosen
 $c\geq E$  we take $M=\cO/F(c)$ and the filtration of $M$ given by
 $F(l)=(F_{\cO}(l)+F_{\cO}(c))/F_{\cO}(c)$.

 However, sometimes is preferable to keep the original setup $(\cO, \{F_{\cO}(l)\}_l)$ $(l\geq 0)$
since it can be connected directly to several  classical invariants  defined at the level of $\cO$.
 This is possible since
 $\dim ( M/F(l))=\dim \cO/F_{\cO}(l)$  for any $l\in R$.

For example, the {\it valuation semigroup}  is defined as $\calS_{\cO}:=\{(\frakv_1,\ldots ,\frakv_s)(f)\ :\ f\in\cO\}$.
 In our case (once $c$ is fixed) we will be interested only in  its part $\{l\, :\  l\not\geq c\}$.

In the case of the divisorial filtration of surface singularities, or
in the case of the valuative filtration associated with plane curve singularities (via its normalization),
  $\calS_{\cO}$ and the Hilbert function $\{\hh(l)\}_{l\geq 0}$ are related as follows:
  \begin{equation}\label{eq:Sfromh}
  \calS_{\cO}=\{l\ :\  \hh(l+E_v)>\hh(l)\ \mbox{for all $v$}\};
  \end{equation}
  and, for any $v\in\cV$,
\begin{equation}\label{eq:hfromS}
\hh(l+E_v)>\hh(l) \ \mbox{exactly when there exists $s\in \calS_{\cO}$ with  $s\geq l$, $s_v=l_v$.}
  \end{equation}
Additionally, in the presence of certain (e.g. Gorenstein)  duality, the ring $\cO$ and the filtration $F_\cO$
 might have the following properties:

 \vspace{2mm}

 \noindent {\it Combinatorial Duality Property of \ $\calS_\cO$}:
 \begin{equation}\label{eq:CombGorS}
 \mbox{there is no   $v\in \cV$ and $s\in\calS_\cO$ so that $s_v=c_v-1$ and $s\geq c-E_v$;}
 \end{equation}
 \noindent {\it Combinatorial Duality Property of \ $\hh$}:
 \begin{equation}\label{eq:CombGorh}
 \left\{
 \begin{array}{ll}\mbox{there is no  $v\in \cV$ and $l\in R$ so that}\\
  \mbox{$\hh(l+E_v)>\hh(l)$ and $\hh(c-l)>\hh(c-l-E_v)$.}\end{array}\right.
 \end{equation}
 Note that  in the presence of
 (\ref{eq:Sfromh}), (\ref{eq:hfromS}) the properties  (\ref{eq:CombGorS}) and
   (\ref{eq:CombGorh}) are equivalent (since  \, $0\in\calS_{\cO}$).
   \end{example}

\bekezdes\label{bek:LCgen}
In order to analyse more general cases, which are
 not necessarily provided by  filtrations, we  wish to adjust  the above
 properties
 in the  language of an arbitrary system  $(h,h^\circ,R)$ as in \ref{bek:comblattice}.
\begin{definition}\label{def:COMPGOR}
Fix  $(h,h^\circ,R)$ as in \ref{bek:comblattice}.
 We say that the pair $h$ and $h^\circ$ satisfy the `Combinatorial Duality  Property' (CDP) if
$h(l+E_v)-h(l)$ and $h^\circ (l+E_v)-h^\circ (l)$ simultaneously cannot be nonzero
for $l,\, l+E_v\in R$. Furthermore,
 we say that $h$  satisfies  the CDP  if
 the pair $(h,h^{sym})$ satisfies  it.
\end{definition}

\begin{example}\label{ex:CDP3pelda}
(1) If $\phi$ is the resolution of a normal surface singularity, $\hh$ is associated with the divisorial filtration and
$\hh^\circ(l)=p_g-h^1(\calO_l)$  then  the pair $(\hh,\hh^\circ)$ satisfies
the CDP by
\ref{lem:hsimult}.

(2)  If $\phi$ is the resolution of a normal surface singularity with $Z_K\in L$. If
 $\hh$ is given by  the divisorial filtration, and $\hh^{sym}(l)=\hh(Z_K-l)$,  then
 $\hh(l+E_v)-\hh(l)$ and $\hh^{sym} (l)-\hh^{sym} (l+E_v)$ cannot be simultaneously nonzero.
Indeed, otherwise there exists $s', s''\in\calS_{an}$ such that ${\rm div}_E(s')\geq l$, $s'_v=l_v$,
 ${\rm div}_E(s'')\geq Z_K-l- E_v$, $s''_v=(Z_K-l)_v$. Hence,
${\rm div}_E(s's'')\geq Z_K-E_v$ with equality at $E_v$--coordinate.
 But this contradicts
$H^0(\calO_{E_v}(-Z_K+E_v)) =0$. (Compare with (\ref{eq:CombGorS})
applied for  $s=s_1s_2\in \calS_{an}$  and  $c=Z_K$.)

(3) Let $\hh$ be the  Hilbert function  associated with a Gorenstein  curve singularity
$(C,o)$ (via valuations given by the normalization).
Let $c$ be the conductor. Then  $\hh(l+E_v)-\hh(l)\in \{0,1\}$. Furthermore,
 $\hh(l+E_v)-\hh(l)=1$ if and only if
$\hh^{sym} (l)-\hh^{sym} (l+E_v)=0$. This follows from  the identity $h(c-l)-h(l)=\delta(C)-\sum_vl_v$, cf. \cite{cdk}.
For the  general treatment of the analytic lattice cohomology associated with curve singularities see
\cite{AgNeCurves}.


\end{example}
\begin{example}\label{ex:rankone} {\bf Rank one case.} Take $R$ and $h$ as in \ref{bek:comblattice}.
Assume that $s=1$.
In this case, from the point of view of $h$,
 we can assume that we are in the graded vector space case, cf. \ref{ex:latgrfilt}.
Indeed,  in this case $h$ is associated with  $M=\C^{h(c)}$ graded as  
$M_l=\C^{h(l+1)-h(l)}$ for $l<c$ and $M_c=0$.  Moreover, if $h^\circ=h^{sym}$,  then
$$w_0(l)=\sum_{s<l} \,\dim M_s+ \sum_{s<c-l}\, \dim M_s \ -h(c).$$
If  $h$ satisfies the CDP   then
(cf. \ref{lem:comblat})
$$w_0(l+1)-w_0(l)=\left\{\begin{array}{ll}
\ \ \ \ \dim\, M_l & \mbox{if $h(l+1)>h(l)$}\\
-\dim\, M_{c-1-l} \ \ \ \  &  \mbox{if \  $h(c-l)>h(c-l-1)$}\\
\ \ \ \ 0 & \mbox{otherwise.}
\end{array}\right.$$
Therefore, by a similar identity as in (\ref{eq:Ecal})
\begin{equation}\label{eq:eurankone}
eu ( \bH^*(R,h))=\sum_{l=0}^{c-1}
\big(\, w_1([l,l+1])-w_0(l)\,\big)=\sum_{l=0}^{c-1}\dim \, M_l=\dim\, M.
\end{equation}
\end{example}
\begin{definition}\label{def:comblat}
We say that the pair   $(h, h^\circ) $ satisfy the

(a) {\it `path eu-coincidence'} if $eu(\bH^0(\gamma,w))=h^\circ (0)-h^\circ (c)$
for any increasing path $\gamma$.

(b)  {\it `eu-coincidence'} if $eu(\bH^*(R,w))=h^\circ (0)-h^\circ (c)$.
\end{definition}

\begin{remark}\label{rem:UJ}
In the Example \ref{ex:twonodesB}
we  present  a  $w_0$--rectangle  $R=R((0,0),(14,14))$
can be realized as the $w_0$-table associated with a certain $h$ and $h^\circ=h^{sym}$
(provided by a
graded vector space). This diagram satisfies both the `path eu-coincidence' and `eu-coincidence'
 properties, and it shows
the following two facts.

Even if $h$ satisfies the path eu-coincidence (and $h^\circ =h^{sym}$),
in general it is not true that $\bH^0(\gamma,w)$
is independent of the choice of the increasing path.
(This statement remains valid even if we consider only the symmetric increasing paths, where a
 path $\gamma=\{x_i\}_{i=0}^t$ is symmetric if $x_{t-l}=c-x_l$ for any $l$.)

Even if $h$ satisfies both the path eu-coincidence and the eu-coincidence,
in general it is not true that $\bH^*(R,w)$ equals  any of the path lattice cohomologies
$\bH^0(\gamma,w)$ associated with a certain  increasing  path.
Indeed, in the present mentioned case of \ref{ex:twonodesB} we have $\bH^1(R,w)\not=0$, a fact which does not hold for any
path lattice cohomology. However, amazingly, all the Euler characteristics agree.
\end{remark}

\begin{remark}\label{ex:WCGP}  In the next discussion assume that  $h^\circ =h^{sym}$.

(a)
The CDP of $h$ implies the `path eu-coincidence' of $h$, see
\ref{lem:comblat}.
However, the  CDP of $h$ does not
imply the `eu-coincidence' of $h$.
As an example consider the second case  from \ref{rem:ketpelda}.
 Note that in this case the matroid or stabilization properties are not satisfied by $h$.

(b) On the other hand,  an  eu-coincidence type property cannot be hoped without
(some type of) CDP. Indeed, in the next example
 $s=2$, $c=(2,2)$, and $h$ is associated with a
graded vector space of dimension 2  supported in $(0,1)$ and $(1,2)$.
In this case $h(c)=2$,  $eu(\bH^0(R,w))=0$ and for any symmetric increasing path
$eu(\bH^0(\gamma,w))=0$ too, and for
any  non-symmetric paths
$eu(\bH^0(\gamma,w))=1$.

\begin{picture}(100,40)(-40,0)

\put(10,10){\makebox(0,0){\small{$0$}}}
\put(25,10){\makebox(0,0){\small{$1$}}}
\put(40,10){\makebox(0,0){\small{$2$}}}
\put(10,20){\makebox(0,0){\small{$0$}}}
\put(25,20){\makebox(0,0){\small{$1$}}}
\put(40,20){\makebox(0,0){\small{$2$}}}
\put(10,30){\makebox(0,0){\small{$1$}}}
\put(25,30){\makebox(0,0){\small{$1$}}}
\put(40,30){\makebox(0,0){\small{$2$}}}

\put(110,10){\makebox(0,0){\small{$0$}}}
\put(125,10){\makebox(0,0){\small{$0$}}}
\put(140,10){\makebox(0,0){\small{$1$}}}
\put(110,20){\makebox(0,0){\small{$0$}}}
\put(125,20){\makebox(0,0){\small{$0$}}}
\put(140,20){\makebox(0,0){\small{$0$}}}
\put(110,30){\makebox(0,0){\small{$1$}}}
\put(125,30){\makebox(0,0){\small{$0$}}}
\put(140,30){\makebox(0,0){\small{$0$}}}

\put(-10,20){\makebox(0,0){\small{$h:$}}}
\put(90,20){\makebox(0,0){\small{$w_0:$}}}
\end{picture}

\end{remark}

\subsection{The Euler characteristic formulae}\label{ss:combEU}

\begin{theorem}\label{th:comblattice}
Assume that $h$ satisfies the stability property, and the pair $(h,h^\circ)$
satisfies the Combinatorial Duality  Property. Then $(h,h^\circ)$
satisfies both the path eu- and the eu-coincidence properties:  for any increasing $\gamma$  we have
$$eu(\bH^*(\gamma,w))=eu(\bH^*(R,w))=h^\circ (0)-h^\circ (c).$$
\end{theorem}
\begin{proof}
The identity $eu(\bH^*(\gamma,w))=h^\circ (0)-h^\circ (c)$ follows from \ref{lem:comblat}. Next we focus on the second identity.
We claim that for  any $I\subset \calv$  we have
\begin{equation}\label{eq:KEYIDEN}
w((l,I))-w(l)=h(l+E_I)-h(l).\end{equation}
We use induction over  the cardinality $|I|$ of $I$. If $I=\{v\}$, then $w((l,I))-w(l)=\max\{ 0, w(l+E_v)-w(l)\}$. But if
$h(l+E_v)>h(l)$ then $h^\circ (l+E_v)=h^\circ (l)$
hence $w(l+E_v)-w(l)=h(l+E_v)-h(l)$. Otherwise $w(l+E_v)\leq w(l)$ and $w((l,I))=w(l)$.

Next, assume that $|I|>1$ and $h(l+E_v)=h(l)$ for every  $v\in \calv$. Then by iterated use of the
stability property of $h$, $h(l+E_J)
=h(l)$ for any $J\subset I$. Moreover, $w((l,I))-w(l)=0=h(l+E_I)-h(l)$.

Finally, assume that $|I|>1$ and $h(l+E_v)>h(l)$ for a certain  $v\in \calv$.
This means that $w(l+E_v)>w(l)$, hence $w((l,I))>w(l)$ ($\dag$).
Assume that $w((l,I))$ is  realized for a certain $J\subset I$.
Let $J$ be minimal by this property. By ($\dag$) we know that $J\not=\emptyset$.

By minimality of $J$, $w(l+E_{J\setminus u})<w(l+E_J)$ for any $u\in J$. By CDP $h(l+E_{J\setminus u})<h(l+E_J)$ too,
hence by the stability of $h$  we also have $h(l)<h(l+E_u)$.

In particular we found $u\in I$  such that $h(l)<h(l+E_u)$
(hence $h^\circ (l)=h^\circ(l+E_u)$)  and $w((l, I))=w((l+E_u, I^*))$, where $I^*:= I\setminus u$.
Now,  we use induction applied for
the cube $(l+E_v, I^*)$. In particular,
$w((l,I))-w(l)=
w((l+E_v, I^*))-w(l)= w(l+E_v)+h(l+E_I)-h(l+E_v)-w(l)=
h(l+E_I)-h(l)$.

This ends the proof of the claim.

Let us denote by  $\calQ$ the set of cubes of $R$.
For $eu(\bH^*(R,w))$ we use the  following formula (with identical proof as in  \ref{bek:LCSW})
$$eu(\bH^*(R,w))=\sum_{(l,I)\in\calQ}\, (-1)^{|I|+1} w((l,I)).$$
We will subdivide the lattice points of $R$ into the following disjoint subsets.
For any $0\leq t\leq s$  we denote by $R_t$ the set of those points $l\in R\cap L$ for which
the cardinality of $\{v\in\calv\,:\, l_v=c_v\}$ is $t$. If $l\in R_t$ set $I(l):=\{v\,:\, l_v<c_v\}$.
In particular, the set $\calQ$ is  also a  disjoint union
of the sets
$ \{(l, I)\,:\, l\in R_t, \ I\subset I(l)\}_t$. Then
$$eu(\bH^*(R,w))=\sum_t \, \sum_{l\in R_t}\, \sum_{I\subset I(l)}\, (-1)^{|I|+1} w((l,I)).$$
If $0\leq t<s$  and $l\in R_t$ then $I(l)\not=\emptyset$. In this case for any
$I$-independent (but maybe $l$-dependent) `constant' $a(l) $ we have $\sum _{I\subset I(l)}
(-1)^{|I|+1}a(l)=0$.
Therefore, by (\ref{eq:KEYIDEN}), for any such $l$,
\begin{equation}\label{eq:SUMt} \sum_{I\subset I(l)}\, (-1)^{|I|+1} w((l,I))=
 \sum_{I\subset I(l)}\, (-1)^{|I|+1} h(l+E_I).\end{equation}
 On the other hand, if $t=s$, then $R_t=\{c\}$, $I(c)=\emptyset$, hence
 $\sum _{I\subset I(c)}
(-1)^{|I|+1}((w(c,I))=-w(c)$.
Hence, corresponding to $t=s$  (\ref{eq:SUMt}) fails, i.e. it must be corrected by
 $-w(c)+h(c)=h^\circ (0)-h^\circ(c)$.
Therefore,
\begin{equation}\label{eq:SUMtt}
eu(\bH^*(R,w))=h^\circ (0)-h^\circ(c)\, +\, \sum_{(l,I)\in\calQ}\, (-1)^{|I|+1} h(l+E_I).
\end{equation}
For any fixed $\widetilde{l}\in R$, consider the following   summation over  $\{(l,I)\,:\,
l+E_I=\widetilde{l}\}$:
$$S(\widetilde{l}):=\sum \, (-1)^{|I|+1}h(l+E_I)=h(\widetilde{l})\cdot \sum \, (-1)^{|I|+1}.$$
Then, whenever the cardinality of   $\{(l,I)\,:\,
l+E_I=\widetilde{l}\}$ is $>1$,
the above sum $\sum \, (-1)^{|I|+1}=0$.
Thus,  $S(\widetilde{l})=0$ except when $\widetilde{l}=0$. But, for $\widetilde{l}=0$,
 $S(0)=-h(0)=0$ too. Thus
$eu(\bH^*(R,w))=h^\circ (0)-h^\circ(c)+\sum_{\widetilde{l}}S(\widetilde{l})=h^\circ (0)-h^\circ(c)$.
\end{proof}
\begin{remark}\label{rem:P0AN}
The very same argument as in the previous proof provides the following fact.
Assume that instead of the rectangle $R(0,c)$  we take $L_{\geq 0}$ (i.e., $c=\infty$), and we also have
two functions $h, \, h^\circ:L_{\geq 0}\to \Z$  with the very same properties (i) and (ii) as in \ref{bek:comblattice}.
Furthermore, assume both that  $h$  satisfies the stability property, and the pair $(h,h^\circ)$
satisfies the Combinatorial Duality  Property (as in \ref{th:comblattice}). Then (\ref{eq:KEYIDEN}) implies
\begin{equation}\label{eq:sumsum}\sum_{l\geq 0}\,\sum _I\, (-1)^{|I|+1} w((l,I))\, \bt^{l}=
\sum_{l\geq 0}\,\sum _I\, (-1)^{|I|+1} h(l+E_I)\,\bt^{l}.\end{equation}
In the case of a normal surface singularity as in \ref{ss:anR}, the right hand side of (\ref{eq:sumsum})
is the $(h=0)$--component $P_0(\bt)$ of the multivariable Poincar\'e series (obtained from the Hilbert series),
 cf. \cite{CDG,CHR,NJEMS}. Hence, the above identity recovers $P_0(\bt)$
in terms of  the analytic weighted cubes.
This formula can be compared with its topological analogue:
there is an identical formula (of identical nature) for the topological multivariable series,
which is recovered in terms of  the topological  weighted cubes.
\end{remark}

\subsection{Examples}\label{ss:combExa}

\begin{example}\label{ex:NNGamma} {\bf Graded vector spaces determined by a Newton diagram.}

Assume that $N\Gamma$ is a convenient Newton diagram in $\R_{\geq 0}^{n+1}$ (see e.g. \cite{NSig}).
Let $\triangle^{(1)}$ denote the $n$-dimensional faces of $N\Gamma$, and for each $\sigma\in \triangle^{(1)}$
let $\ell_\sigma$ be the normal  primitive integral vector of the corresponding face $F_\sigma$ (with all
entries positive). Then
$F_\sigma$ is on the affine $n$-plane $\langle \ell_\sigma,p\rangle=m_\sigma$.
 Let $M$ be the $\C$--vector space  generated by lattice points
$p\in \Z_{>0}^{n+1}$, which sit either on or below $N\Gamma$. This means that
 $p\in \Z_{>0}^{n+1}$  should satisfy
$\langle \ell_\sigma,p\rangle\leq m_\sigma$ for at least one $\sigma$.
Then $M\not=0$  if and only if ${\bf 1}=(1,\ldots , 1)$ is such a point; in the sequel we will assume this fact.

We order $\triangle ^{(1)}$ as $\{\sigma_1,\ldots, \sigma_s\}$. We introduce a $\Z^s$-grading of $M$ by
$${\rm deg}(p)=(\, \langle \ell_{\sigma_1}, p-{\bf 1}\rangle, \ldots,
 \langle \ell_{\sigma_s}, p-{\bf 1}\rangle \, )\in \Z^s.$$
Define  $c=(c_1,\ldots , c_s)$ by
 $c_i:= m_{\sigma_i}+1-\langle \ell_{\sigma_i}, {\bf 1}\rangle$.

Since  $p\geq {\bf 1}$ for any such point,
we also have ${\rm deg}(p)\in \Z_{\geq 0}^s$. Moreover, for any $p$ there exists
$\sigma_i$ with $\langle \ell_{\sigma_i}, p\rangle\leq m_{\sigma_i}$, hence
$\langle \ell_{\sigma_i}, p-{\bf 1}\rangle <c_i$. Therefore, ${\rm deg}(p)\not\geq c$.

In particular,  the conditions of \ref{ex:latgrfilt} are satisfied.
\end{example}

\begin{example}\label{ex:2552}
Consider the situation of \ref{ex:NNGamma} with $n=1$ and $N\Gamma$ generated by the lattice points
$(0,7), (2,2), (7,0)$. In fact, this is the Newton diagram of the Newton nondegenerate
plane curve singularity $f=(x^2+y^5)(y^2+x^5)$.
The normal vectors $(2,5)$ and $(5,2)$ of the two faces provide the degrees.
The next diagram shows the points $p$, which generate $M$ as base elements,  and their degrees.

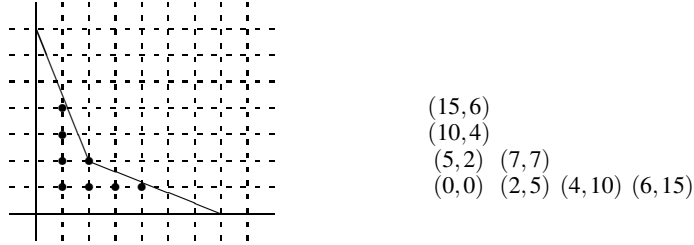
\begin{picture}(300,110)(-100,-20)

\put(-10,0){\line(1,0){100}}
\put(0,-10){\line(0,1){90}}

\dashline{2}(-10,10)(90,10)
\dashline{2}(-10,20)(90,20)
\dashline{2}(-10,30)(90,30)
\dashline{2}(-10,40)(90,40)
\dashline{2}(-10,50)(90,50)
\dashline{2}(-10,60)(90,60)
\dashline{2}(-10,70)(90,70)

\dashline{2}(10,-10)(10,80)
\dashline{2}(20,-10)(20,80)
\dashline{2}(30,-10)(30,80)
\dashline{2}(40,-10)(40,80)
\dashline{2}(50,-10)(50,80)
\dashline{2}(60,-10)(60,80)
\dashline{2}(70,-10)(70,80)
\dashline{2}(80,-10)(80,80)
\put(20,20){\line(5,-2){50}}
\put(20,20){\line(-2,5){20}}
\put(10,10){\circle*{3}}
\put(10,20){\circle*{3}}
\put(10,30){\circle*{3}}
\put(10,40){\circle*{3}}
\put(20,10){\circle*{3}}
\put(20,20){\circle*{3}}
\put(30,10){\circle*{3}}
\put(40,10){\circle*{3}}

\put(160,10){\makebox(0,0){\small{$(0,0)$}}}
\put(160,20){\makebox(0,0){\small{$(5,2)$}}}
\put(160,30){\makebox(0,0){\small{$(10,4)$}}}
\put(160,40){\makebox(0,0){\small{$(15,6)$}}}
\put(185,10){\makebox(0,0){\small{$(2,5)$}}}
\put(185,20){\makebox(0,0){\small{$(7,7)$}}}
\put(210,10){\makebox(0,0){\small{$(4,10)$}}}
\put(237,10){\makebox(0,0){\small{$(6,15)$}}}

\end{picture}

\noindent
The vector $c$ is $(8,8)$:  this is exactly the conductor of the plane curve singularity $f$.
Furthermore,  $\dim M=8$ is the delta-invariant $\delta(f)$ of $f$.
The next tables show the function $h$  and the weight function $w_0$ on $R$.

\begin{picture}(320,120)(-50,-20)

\put(-15,0){\line(1,0){145}}
\put(-5,-10){\line(0,1){100}}

\put(5,-5){\makebox(0,0){\small{$0$}}}
\put(20,-5){\makebox(0,0){\small{$1$}}}
\put(35,-5){\makebox(0,0){\small{$2$}}}
\put(50,-5){\makebox(0,0){\small{$3$}}}
\put(65,-5){\makebox(0,0){\small{$4$}}}
\put(80,-5){\makebox(0,0){\small{$5$}}}
\put(95,-5){\makebox(0,0){\small{$6$}}}
\put(110,-5){\makebox(0,0){\small{$7$}}}
\put(125,-5){\makebox(0,0){\small{$8$}}}

\put(-10,5){\makebox(0,0){\small{$0$}}}
\put(-10,15){\makebox(0,0){\small{$1$}}}
\put(-10,25){\makebox(0,0){\small{$2$}}}
\put(-10,35){\makebox(0,0){\small{$3$}}}
\put(-10,45){\makebox(0,0){\small{$4$}}}
\put(-10,55){\makebox(0,0){\small{$5$}}}
\put(-10,65){\makebox(0,0){\small{$6$}}}
\put(-10,75){\makebox(0,0){\small{$7$}}}
\put(-10,85){\makebox(0,0){\small{$8$}}}

\put(5,5){\makebox(0,0){\small{$0$}}}
\put(5,15){\makebox(0,0){\small{$1$}}}
\put(5,25){\makebox(0,0){\small{$1$}}}
\put(5,35){\makebox(0,0){\small{$2$}}}
\put(5,45){\makebox(0,0){\small{$2$}}}
\put(5,55){\makebox(0,0){\small{$3$}}}
\put(5,65){\makebox(0,0){\small{$4$}}}
\put(5,75){\makebox(0,0){\small{$5$}}}
\put(5,85){\makebox(0,0){\small{$6$}}}

\put(20,5){\makebox(0,0){\small{$1$}}}
\put(20,15){\makebox(0,0){\small{$1$}}}
\put(20,25){\makebox(0,0){\small{$1$}}}
\put(20,35){\makebox(0,0){\small{$2$}}}
\put(20,45){\makebox(0,0){\small{$2$}}}
\put(20,55){\makebox(0,0){\small{$3$}}}
\put(20,65){\makebox(0,0){\small{$4$}}}
\put(20,75){\makebox(0,0){\small{$5$}}}
\put(20,85){\makebox(0,0){\small{$6$}}}

\put(35,5){\makebox(0,0){\small{$1$}}}
\put(35,15){\makebox(0,0){\small{$1$}}}
\put(35,25){\makebox(0,0){\small{$1$}}}
\put(35,35){\makebox(0,0){\small{$2$}}}
\put(35,45){\makebox(0,0){\small{$2$}}}
\put(35,55){\makebox(0,0){\small{$3$}}}
\put(35,65){\makebox(0,0){\small{$4$}}}
\put(35,75){\makebox(0,0){\small{$5$}}}
\put(35,85){\makebox(0,0){\small{$6$}}}

\put(50,5){\makebox(0,0){\small{$2$}}}
\put(50,15){\makebox(0,0){\small{$2$}}}
\put(50,25){\makebox(0,0){\small{$2$}}}
\put(50,35){\makebox(0,0){\small{$3$}}}
\put(50,45){\makebox(0,0){\small{$3$}}}
\put(50,55){\makebox(0,0){\small{$4$}}}
\put(50,65){\makebox(0,0){\small{$4$}}}
\put(50,75){\makebox(0,0){\small{$5$}}}
\put(50,85){\makebox(0,0){\small{$6$}}}

\put(65,5){\makebox(0,0){\small{$2$}}}
\put(65,15){\makebox(0,0){\small{$2$}}}
\put(65,25){\makebox(0,0){\small{$2$}}}
\put(65,35){\makebox(0,0){\small{$3$}}}
\put(65,45){\makebox(0,0){\small{$3$}}}
\put(65,55){\makebox(0,0){\small{$4$}}}
\put(65,65){\makebox(0,0){\small{$4$}}}
\put(65,75){\makebox(0,0){\small{$5$}}}
\put(65,85){\makebox(0,0){\small{$6$}}}

\put(80,5){\makebox(0,0){\small{$3$}}}
\put(80,15){\makebox(0,0){\small{$3$}}}
\put(80,25){\makebox(0,0){\small{$3$}}}
\put(80,35){\makebox(0,0){\small{$4$}}}
\put(80,45){\makebox(0,0){\small{$4$}}}
\put(80,55){\makebox(0,0){\small{$5$}}}
\put(80,65){\makebox(0,0){\small{$5$}}}
\put(80,75){\makebox(0,0){\small{$6$}}}
\put(80,85){\makebox(0,0){\small{$7$}}}

\put(95,5){\makebox(0,0){\small{$4$}}}
\put(95,15){\makebox(0,0){\small{$4$}}}
\put(95,25){\makebox(0,0){\small{$4$}}}
\put(95,35){\makebox(0,0){\small{$4$}}}
\put(95,45){\makebox(0,0){\small{$4$}}}
\put(95,55){\makebox(0,0){\small{$5$}}}
\put(95,65){\makebox(0,0){\small{$5$}}}
\put(95,75){\makebox(0,0){\small{$6$}}}
\put(95,85){\makebox(0,0){\small{$7$}}}

\put(110,5){\makebox(0,0){\small{$5$}}}
\put(110,15){\makebox(0,0){\small{$5$}}}
\put(110,25){\makebox(0,0){\small{$5$}}}
\put(110,35){\makebox(0,0){\small{$5$}}}
\put(110,45){\makebox(0,0){\small{$5$}}}
\put(110,55){\makebox(0,0){\small{$6$}}}
\put(110,65){\makebox(0,0){\small{$6$}}}
\put(110,75){\makebox(0,0){\small{$7$}}}
\put(110,85){\makebox(0,0){\small{$8$}}}

\put(125,5){\makebox(0,0){\small{$6$}}}
\put(125,15){\makebox(0,0){\small{$6$}}}
\put(125,25){\makebox(0,0){\small{$6$}}}
\put(125,35){\makebox(0,0){\small{$6$}}}
\put(125,45){\makebox(0,0){\small{$6$}}}
\put(125,55){\makebox(0,0){\small{$7$}}}
\put(125,65){\makebox(0,0){\small{$7$}}}
\put(125,75){\makebox(0,0){\small{$8$}}}
\put(125,85){\makebox(0,0){\small{$8$}}}


\put(160,0){\line(1,0){145}}
\put(170,-10){\line(0,1){100}}


\put(180,5){\makebox(0,0){\small{$0$}}}
\put(180,15){\makebox(0,0){\small{$1$}}}
\put(180,25){\makebox(0,0){\small{$0$}}}
\put(180,35){\makebox(0,0){\small{$1$}}}
\put(180,45){\makebox(0,0){\small{$0$}}}
\put(180,55){\makebox(0,0){\small{$1$}}}
\put(180,65){\makebox(0,0){\small{$2$}}}
\put(180,75){\makebox(0,0){\small{$3$}}}
\put(180,85){\makebox(0,0){\small{$4$}}}

\put(195,5){\makebox(0,0){\small{$1$}}}
\put(195,15){\makebox(0,0){\small{$0$}}}
\put(195,25){\makebox(0,0){\small{$-1$}}}
\put(195,35){\makebox(0,0){\small{$0$}}}
\put(195,45){\makebox(0,0){\small{$-1$}}}
\put(195,55){\makebox(0,0){\small{$0$}}}
\put(195,65){\makebox(0,0){\small{$1$}}}
\put(195,75){\makebox(0,0){\small{$2$}}}
\put(195,85){\makebox(0,0){\small{$3$}}}

\put(210,5){\makebox(0,0){\small{$0$}}}
\put(210,15){\makebox(0,0){\small{$-1$}}}
\put(210,25){\makebox(0,0){\small{$-2$}}}
\put(210,35){\makebox(0,0){\small{$-1$}}}
\put(210,45){\makebox(0,0){\small{$-2$}}}
\put(210,55){\makebox(0,0){\small{$-1$}}}
\put(210,65){\makebox(0,0){\small{$0$}}}
\put(210,75){\makebox(0,0){\small{$1$}}}
\put(210,85){\makebox(0,0){\small{$2$}}}

\put(225,5){\makebox(0,0){\small{$1$}}}
\put(225,15){\makebox(0,0){\small{$0$}}}
\put(225,25){\makebox(0,0){\small{$-1$}}}
\put(225,35){\makebox(0,0){\small{$0$}}}
\put(225,45){\makebox(0,0){\small{$-1$}}}
\put(225,55){\makebox(0,0){\small{$0$}}}
\put(225,65){\makebox(0,0){\small{$-1$}}}
\put(225,75){\makebox(0,0){\small{$0$}}}
\put(225,85){\makebox(0,0){\small{$1$}}}

\put(240,5){\makebox(0,0){\small{$0$}}}
\put(240,15){\makebox(0,0){\small{$-1$}}}
\put(240,25){\makebox(0,0){\small{$-2$}}}
\put(240,35){\makebox(0,0){\small{$-1$}}}
\put(240,45){\makebox(0,0){\small{$-2$}}}
\put(240,55){\makebox(0,0){\small{$-1$}}}
\put(240,65){\makebox(0,0){\small{$-2$}}}
\put(240,75){\makebox(0,0){\small{$-1$}}}
\put(240,85){\makebox(0,0){\small{$0$}}}

\put(255,5){\makebox(0,0){\small{$1$}}}
\put(255,15){\makebox(0,0){\small{$0$}}}
\put(255,25){\makebox(0,0){\small{$-1$}}}
\put(255,35){\makebox(0,0){\small{$0$}}}
\put(255,45){\makebox(0,0){\small{$-1$}}}
\put(255,55){\makebox(0,0){\small{$0$}}}
\put(255,65){\makebox(0,0){\small{$-1$}}}
\put(255,75){\makebox(0,0){\small{$0$}}}
\put(255,85){\makebox(0,0){\small{$1$}}}

\put(270,5){\makebox(0,0){\small{$2$}}}
\put(270,15){\makebox(0,0){\small{$1$}}}
\put(270,25){\makebox(0,0){\small{$0$}}}
\put(270,35){\makebox(0,0){\small{$-1$}}}
\put(270,45){\makebox(0,0){\small{$-2$}}}
\put(270,55){\makebox(0,0){\small{$-1$}}}
\put(270,65){\makebox(0,0){\small{$-2$}}}
\put(270,75){\makebox(0,0){\small{$-1$}}}
\put(270,85){\makebox(0,0){\small{$0$}}}

\put(285,5){\makebox(0,0){\small{$3$}}}
\put(285,15){\makebox(0,0){\small{$2$}}}
\put(285,25){\makebox(0,0){\small{$1$}}}
\put(285,35){\makebox(0,0){\small{$0$}}}
\put(285,45){\makebox(0,0){\small{$-1$}}}
\put(285,55){\makebox(0,0){\small{$0$}}}
\put(285,65){\makebox(0,0){\small{$-1$}}}
\put(285,75){\makebox(0,0){\small{$0$}}}
\put(285,85){\makebox(0,0){\small{$1$}}}

\put(300,5){\makebox(0,0){\small{$4$}}}
\put(300,15){\makebox(0,0){\small{$3$}}}
\put(300,25){\makebox(0,0){\small{$2$}}}
\put(300,35){\makebox(0,0){\small{$1$}}}
\put(300,45){\makebox(0,0){\small{$0$}}}
\put(300,55){\makebox(0,0){\small{$1$}}}
\put(300,65){\makebox(0,0){\small{$0$}}}
\put(300,75){\makebox(0,0){\small{$1$}}}
\put(300,85){\makebox(0,0){\small{$0$}}}

\put(225,35){\circle{10}}\put(255,55){\circle{10}}
\put(202,20){\framebox(15,10){}}
\put(232,20){\framebox(15,10){}}
\put(202,40){\framebox(15,10){}}
\put(232,40){\framebox(15,10){}}
\put(232,60){\framebox(15,10){}}
\put(262,40){\framebox(15,10){}}
\put(262,60){\framebox(15,10){}}

\put(172,1){\framebox(15,10){}}
\put(292,80){\framebox(15,10){}}

\end{picture}

\noindent
The circled points show the generators of $\bH^1(R,w)$, while the boxes the local minima of $w_0$.
 Hence $\bH^1(R,w)=\calt_{-2}(1)^2$ and $\bH^0(R,w)=\calt^+_{-4}\oplus
 \calt_{-4}(1)^6\oplus \calt_0(1)^2$.
The graded root is

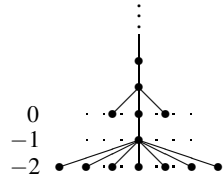
\begin{picture}(300,82)(80,350)

\put(180,380){\makebox(0,0){\small{$0$}}} \put(177,370){\makebox(0,0){\small{$-1$}}}
\put(177,360){\makebox(0,0){\small{$-2$}}}
\dashline{1}(200,370)(240,370)
\dashline{1}(200,380)(240,380) 
\dashline{1}(200,360)(240,360)
\put(220,420){\makebox(0,0){$\vdots$}} \put(220,400){\circle*{3}}
\put(220,390){\circle*{3}} \put(210,380){\circle*{3}}
\put(230,380){\circle*{3}} \put(220,380){\circle*{3}}
\put(220,370){\circle*{3}} \put(220,410){\line(0,-1){40}}
\put(210,380){\line(1,1){10}} \put(220,390){\line(1,-1){10}}

\put(190,360){\circle*{3}} \put(200,360){\circle*{3}}
\put(210,360){\circle*{3}} \put(220,360){\circle*{3}}
\put(230,360){\circle*{3}} \put(240,360){\circle*{3}}
\put(250,360){\circle*{3}}
 \put(220,370){\line(0,-1){10}}
 \put(220,370){\line(1,-1){10}} \put(220,370){\line(-1,-1){10}}
 \put(220,370){\line(2,-1){20}} \put(220,370){\line(-2,-1){20}}
  \put(220,370){\line(3,-1){30}} \put(220,370){\line(-3,-1){30}}

\end{picture}

\noindent
$\min w_0=-2$, ${\rm rank}\bH^0_{red}(R,w)=8$, ${\rm rank}\bH^1(R,w)=2$, and $eu(\bH^*(R,w))=8=\dim M=h(c)$.
In particular, we constructed a  set of cohomology groups $\bH^*(R, w)$ whose Euler characteristic is the
delta-invariant. This is a `categorification of the $\delta$--invariant'.
This case of curves will  be fully exploited in \cite{AgNeCurves}.

From analytic point of view, one can show that
the above function $h$ is the Hilbert function associated with the
filtration of valuations associated with the normalization.
This can be verified e.g. by using the
formula of the analytic Hilbert function via the multivariable  Alexander polynomials, see \cite{GorNem2015}.

In particular, in this case, the combinatorial lattice cohomology associated with
the lattice points below the diagram agrees with the analytic lattice cohomology
associated with the valuations given by the normalization.
(Nevertheless, we warn the reader
that for a more general $N\Gamma$, several filtrations --- projected, order, valuative ---
might all be  different.)

\end{example}

\begin{example}\label{ex:twonodesB}
Consider the hypersurface $x^{13}+y^{13}+x^2y^2+z^3$ with Newton nondegenerate principal part.
Its dual resolution graph is

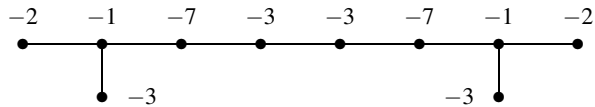
\begin{center}
\begin{picture}(140,40)(80,35)
\put(110,60){\circle*{4}}
\put(140,60){\circle*{4}}
\put(170,60){\circle*{4}}
\put(200,60){\circle*{4}}
\put(80,60){\circle*{4}}
\put(50,60){\circle*{4}}
\put(230,60){\circle*{4}}
\put(260,60){\circle*{4}}
\put(50,60){\line(1,0){210}}
\put(80,60){\line(0,-1){20}}
\put(80,40){\circle*{4}}
\put(230,60){\line(0,-1){20}}
\put(230,40){\circle*{4}}
\put(50,70){\makebox(0,0){\small{$-2$}}}
\put(80,70){\makebox(0,0){\small{$-1$}}}
\put(110,70){\makebox(0,0){\small{$-7$}}}
\put(140,70){\makebox(0,0){\small{$-3$}}}
\put(170,70){\makebox(0,0){\small{$-3$}}}
\put(200,70){\makebox(0,0){\small{$-7$}}}
\put(230,70){\makebox(0,0){\small{$-1$}}}
\put(260,70){\makebox(0,0){\small{$-2$}}}
\put(95,40){\makebox(0,0){\small{$-3$}}}
\put(215,40){\makebox(0,0){\small{$-3$}}}
\end{picture}
\end{center}

Let  $N\Gamma$ be its
Newton diagram.  It has two faces with normal vectors
$(6,33,26)$ and $(33,6,26)$. The number of lattice point under the diagram is $p_g=5$, they are
$(1,1,1)$, $(2,1,1)$, $(3,1,1)$, $(1,2,1)$, $(1,3,1)$. Furthermore, $c$ (defined as in \ref{ex:NNGamma})
is exactly the restriction of $Z_K$ to the coordinates of the nodes (cf.  \cite{NSig,NSigNN}).
Hence, in this case $R=R(0,Z_K|_{\tiny{\mbox{nodes}}})=R((0,0), (14,14))$. Then a computation shows that
the combinatorial weight function $w_0:R\to \Z$ associated with $(M, {\rm deg})$ from \ref{ex:NNGamma}
(that is, associated with the Newton diagram) is exactly the
table  obtained by the procedure of the (topological)  Reduction Theorem applied for the topological
RR-function  $\chi$. Hence we have the coincidence of the corresponding  lattice cohomologies as well:
$\bH^*(R,w)=\bH^*_{top}(M)$.
Furthermore, for this diagram, the weight function associated with the
divisorial filtration  restricted   to $R$ agrees with combinatorial  weight function  as well (see e.g. \cite{Lem,BaldurFiltr}).
These three common tables are the following:

\begin{center}
\begin{picture}(0,160)(190,-5)

\put(270,100){\makebox(0,0)[l]{\small{$\bH^1=\calt_0(1)$}}}
\put(270,80){\makebox(0,0)[l]{\small{$\bH^0=\calt_{-2}^+\oplus \calt_{-2}(1)^3\oplus \calt_0(1)^2$}}}
\put(270,60){\makebox(0,0)[l]{\small{$\min w_0=-1$}}}
\put(270,40){\makebox(0,0)[l]{\small{$eu=5$}}}

\put(15,0){\small$0$}\put(30,0){\small$1$}\put(45,0){\small$0$}\put(60,0){\small$0$}\put(75,0){\small$0$}
\put(90,0){\small$0$}\put(105,0){\small$0$}\put(120,0){\small$1$}\put(135,0){\small$0$}\put(150,0){\small$0$}
\put(165,0){\small$0$}\put(180,0){\small$0$}\put(195,0){\small$0$}\put(210,0){\small$1$}\put(225,0){\small$1$}
\put(12,-2){\line(1,0){11}}\put(23,-2){\line(0,1){10}}\put(23,8){\line(-1,0){11}}\put(12,8){\line(0,-1){10}}

\put(15,10){\small$1$}\put(30,10){\small$1$}\put(45,10){\small$0$}\put(60,10){\small$0$}\put(75,10){\small$0$}
\put(90,10){\small$0$}\put(105,10){\small$0$}\put(120,10){\small$1$}\put(135,10){\small$0$}\put(150,10){\small$0$}
\put(165,10){\small$0$}\put(180,10){\small$0$}\put(195,10){\small$0$}\put(210,10){\small$1$}\put(225,10){\small$1$}

\put(15,20){\small$0$}\put(30,20){\small$0$}\put(38,20){\small$-1$}\put(53,20){\small$-1$}\put(68,20){\small$-1$}
\put(83,20){\small$-1$}\put(98,20){\small$-1$}\put(120,20){\small$0$}\put(128,20){\small$-1$}\put(143,20){\small$-1$}
\put(158,20){\small$-1$}\put(173,20){\small$-1$}\put(188,20){\small$-1$}\put(210,20){\small$0$}\put(225,20){\small$0$}
\put(36,18){\line(1,0){76}}\put(112,18){\line(0,1){50}}\put(112,68){\line(-1,0){76}}\put(36,68){\line(0,-1){50}}
\put(126,18){\line(1,0){76}}\put(202,18){\line(0,1){50}}\put(202,68){\line(-1,0){76}}\put(126,68){\line(0,-1){50}}

\put(15,30){\small$0$}\put(30,30){\small$0$}\put(38,30){\small$-1$}\put(53,30){\small$-1$}\put(68,30){\small$-1$}
\put(83,30){\small$-1$}\put(98,30){\small$-1$}\put(120,30){\small$0$}\put(128,30){\small$-1$}\put(143,30){\small$-1$}
\put(158,30){\small$-1$}\put(173,30){\small$-1$}\put(188,30){\small$-1$}\put(210,30){\small$0$}\put(225,30){\small$0$}

\put(15,40){\small$0$}\put(30,40){\small$0$}\put(38,40){\small$-1$}\put(53,40){\small$-1$}\put(68,40){\small$-1$}
\put(83,40){\small$-1$}\put(98,40){\small$-1$}\put(120,40){\small$0$}\put(128,40){\small$-1$}\put(143,40){\small$-1$}
\put(158,40){\small$-1$}\put(173,40){\small$-1$}\put(188,40){\small$-1$}\put(210,40){\small$0$}\put(225,40){\small$0$}

\put(15,50){\small$0$}\put(30,50){\small$0$}\put(38,50){\small$-1$}\put(53,50){\small$-1$}\put(68,50){\small$-1$}
\put(83,50){\small$-1$}\put(98,50){\small$-1$}\put(120,50){\small$0$}\put(128,50){\small$-1$}\put(143,50){\small$-1$}
\put(158,50){\small$-1$}\put(173,50){\small$-1$}\put(188,50){\small$-1$}\put(210,50){\small$0$}\put(225,50){\small$0$}

\put(15,60){\small$0$}\put(30,60){\small$0$}\put(38,60){\small$-1$}\put(53,60){\small$-1$}\put(68,60){\small$-1$}
\put(83,60){\small$-1$}\put(98,60){\small$-1$}\put(120,60){\small$0$}\put(128,60){\small$-1$}\put(143,60){\small$-1$}
\put(158,60){\small$-1$}\put(173,60){\small$-1$}\put(188,60){\small$-1$}\put(210,60){\small$0$}\put(225,60){\small$0$}

\put(15,70){\small$1$}\put(30,70){\small$1$}\put(45,70){\small$0$}\put(60,70){\small$0$}\put(75,70){\small$0$}
\put(90,70){\small$0$}\put(105,70){\small$0$}\put(120,70){\small$1$}\put(135,70){\small$0$}\put(150,70){\small$0$}
\put(165,70){\small$0$}\put(180,70){\small$0$}\put(195,70){\small$0$}\put(210,70){\small$1$}\put(225,70){\small$1$}
\put(122,73){\circle{12}}

\put(15,80){\small$0$}\put(30,80){\small$0$}\put(38,80){\small$-1$}\put(53,80){\small$-1$}\put(68,80){\small$-1$}
\put(83,80){\small$-1$}\put(98,80){\small$-1$}\put(120,80){\small$0$}\put(128,80){\small$-1$}\put(143,80){\small$-1$}
\put(158,80){\small$-1$}\put(173,80){\small$-1$}\put(188,80){\small$-1$}\put(210,80){\small$0$}\put(225,80){\small$0$}
\put(36,78){\line(1,0){76}}\put(112,78){\line(0,1){50}}\put(112,128){\line(-1,0){76}}\put(36,128){\line(0,-1){50}}
\put(126,78){\line(1,0){76}}\put(202,78){\line(0,1){50}}\put(202,128){\line(-1,0){76}}\put(126,128){\line(0,-1){50}}

\put(15,90){\small$0$}\put(30,90){\small$0$}\put(38,90){\small$-1$}\put(53,90){\small$-1$}\put(68,90){\small$-1$}
\put(83,90){\small$-1$}\put(98,90){\small$-1$}\put(120,90){\small$0$}\put(128,90){\small$-1$}\put(143,90){\small$-1$}
\put(158,90){\small$-1$}\put(173,90){\small$-1$}\put(188,90){\small$-1$}\put(210,90){\small$0$}\put(225,90){\small$0$}

\put(15,100){\small$0$}\put(30,100){\small$0$}\put(38,100){\small$-1$}\put(53,100){\small$-1$}\put(68,100){\small$-1$}
\put(83,100){\small$-1$}\put(98,100){\small$-1$}\put(120,100){\small$0$}\put(128,100){\small$-1$}\put(143,100){\small$-1$}
\put(158,100){\small$-1$}\put(173,100){\small$-1$}\put(188,100){\small$-1$}\put(210,100){\small$0$}\put(225,100){\small$0$}

\put(15,110){\small$0$}\put(30,110){\small$0$}\put(38,110){\small$-1$}\put(53,110){\small$-1$}\put(68,110){\small$-1$}
\put(83,110){\small$-1$}\put(98,110){\small$-1$}\put(120,110){\small$0$}\put(128,110){\small$-1$}\put(143,110){\small$-1$}
\put(158,110){\small$-1$}\put(173,110){\small$-1$}\put(188,110){\small$-1$}\put(210,110){\small$0$}\put(225,110){\small$0$}

\put(15,120){\small$0$}\put(30,120){\small$0$}\put(38,120){\small$-1$}\put(53,120){\small$-1$}\put(68,120){\small$-1$}
\put(83,120){\small$-1$}\put(98,120){\small$-1$}\put(120,120){\small$0$}\put(128,120){\small$-1$}\put(143,120){\small$-1$}
\put(158,120){\small$-1$}\put(173,120){\small$-1$}\put(188,120){\small$-1$}\put(210,120){\small$0$}\put(225,120){\small$0$}

\put(15,130){\small$1$}\put(30,130){\small$1$}\put(45,130){\small$0$}\put(60,130){\small$0$}\put(75,130){\small$0$}
\put(90,130){\small$0$}\put(105,130){\small$0$}\put(120,130){\small$1$}\put(135,130){\small$0$}\put(150,130){\small$0$}
\put(165,130){\small$0$}\put(180,130){\small$0$}\put(195,130){\small$0$}\put(210,130){\small$1$}\put(225,130){\small$1$}

\put(15,140){\small$1$}\put(30,140){\small$1$}\put(45,140){\small$0$}\put(60,140){\small$0$}\put(75,140){\small$0$}
\put(90,140){\small$0$}\put(105,140){\small$0$}\put(120,140){\small$1$}\put(135,140){\small$0$}\put(150,140){\small$0$}
\put(165,140){\small$0$}\put(180,140){\small$0$}\put(195,140){\small$0$}\put(210,140){\small$1$}\put(225,140){\small$0$}
\put(222,138){\line(1,0){11}}\put(233,138){\line(0,1){10}}\put(233,148){\line(-1,0){11}}\put(222,148){\line(0,-1){10}}
\end{picture}
\end{center}

In particular, for this germ, $\bH^*_{an,0}(X,o)=\bH^*_{top,0}(M)$.

From the point of view of deformations, let us consider the following Newton nondegenerate germs:
$\{x^{13}+y^{13}+x^3y^2+x^2y^3+z^3=0\}$, or $\{x^{14}+y^{14}+x^2y^2+z^3=0\}$, or
$\{x^{13}+y^{9}z+x^2y^2+z^3=0\}$. All of them are $p_g$-constant deformations of
the germ $\{x^{13}+y^{13}+x^2y^2+z^3=0\}$ treated above. Then for any  of them,  the
 corresponding  $w_0$--table (determined by the Newton filtration)
in the rectangle given by $Z_K$ reduced to the nodes  agree with the corresponding $\chi$--table
(though, the resolution graphs are rather  different, and the size of $R(0,Z_K|_{\tiny{\mbox{nodes}}})$ is also changing). Furthermore, in all these cases,
the identity
$\bH^*_{an,0}(X,o)=\bH^*_{top,0}(M)$ holds, and this module  stays stable under the corresponding deformations.
\end{example}

\end{document}